\documentclass[12pt]{article}

\usepackage{graphicx}
\usepackage{latexsym,amsmath,amsfonts,amscd, amsthm}
\usepackage{changebar}
\usepackage{color}
\usepackage{bm}
\usepackage{tikz}
\usepackage{multirow}
\usetikzlibrary{arrows,backgrounds,snakes}
\usepackage{epsfig,subfigure}

\newtheoremstyle{plainNoItalics}{}{}{\normalfont}{}{\bfseries}{.}{ }{}

\theoremstyle{plain}
\newtheorem{thm}{Theorem}[section]

\theoremstyle{plainNoItalics}

\newtheorem{defn}[thm]{Definition}
\newtheorem{rem}[thm]{Remark}
\newtheorem{prop}[thm]{Proposition}
\newtheorem{exa}[thm]{Example}

\newcommand{\f}{\frac}

\newcommand{\beq}{\begin{equation}}
\newcommand{\eeq}{\end{equation}}
\newcommand{\beqa}{\begin{eqnarray}}
\newcommand{\eeqa}{\end{eqnarray}}
\newcommand{\bit}{\begin{itemize}}
\newcommand{\eit}{\end{itemize}}
\newcommand{\bedef}{\begin{defn}}
\newcommand{\edefn}{\end{defn}}
\newcommand{\bpro}{\begin{prop}}
\newcommand{\epro}{\end{prop}}

\newcommand{\df}{\partial}

\newcommand{\Dt}{\Delta t}

\newcommand{\eps}{\varepsilon}
\newcommand{\mD}{{\mathcal D}}

\newcommand{\xL}{{x_{i-\frac{1}{2}}}}
\newcommand{\xR}{{x_{i+\frac{1}{2}}}}
\newcommand{\iL}{{i-\frac{1}{2}}}
\newcommand{\iR}{{i+\frac{1}{2}}}

\begin{document}



\begin{center}
{\bf
High Order Asymptotic Preserving Nodal Discontinuous Galerkin IMEX Schemes for the BGK Equation}
\end{center}
\vspace{.2in}

\centerline{
Tao Xiong \footnote{Department of
Mathematics, University of Houston, Houston, 77004. E-mail:
txiong@math.uh.edu}
Juhi Jang \footnote{Department of Mathematics, University of California Riverside, Riverside, CA 92521. E-mail: juhijang@math.ucr.edu. Supported in part by NSF grant DMS-1212142.}
Fengyan Li \footnote{Department of Mathematical Sciences, Rensselaer Polytechnic Institute, Troy, NY 12180. E-mail: lif@rpi.edu. Supported in part by NSF grants DMS-0847241 and DMS-1318409.}
Jing-Mei Qiu \footnote{Department of Mathematics, University of Houston,
Houston, 77004. E-mail: jingqiu@math.uh.edu.
The first and last authors are supported by Air Force Office of Scientific Computing YIP grant FA9550-12-0318, NSF grant DMS-1217008 and University of Houston.}
}

\bigskip
\centerline{\bf Abstract}

In this paper, we develop high-order asymptotic preserving (AP) {schemes for the} BGK equation in a hyperbolic scaling, which leads to the macroscopic models such as the Euler and compressible Navier-Stokes equations in the asymptotic limit. Our approaches are based on the so-called micro-macro formulation of the kinetic equation which involves {a} natural decomposition of the problem to {the equilibrium  and the non-equilibrium parts}. The proposed methods are formulated for the BGK equation with constant or spatially variant Knudsen number.
The new ingredients for the proposed methods to achieve high order accuracy  are the following: we introduce discontinuous Galerkin (DG) discretization of arbitrary order of accuracy with nodal Lagrangian basis functions in space; we employ a high order globally stiffly accurate implicit-explicit (IMEX) Runge-Kutta (RK) scheme as time discretization. Two versions of the schemes are proposed:
Scheme I is a direct formulation based on the micro-macro decomposition of the BGK equation, while Scheme II, motivated by the asymptotic analysis for the continuous problem, utilizes certain properties of the projection operator. Compared with Scheme I, Scheme II not only has better computational efficiency (the computational cost is reduced by half roughly), but also allows the establishment of a formal asymptotic analysis. 
Specifically, it is demonstrated that when $0<\eps \ll 1$, Scheme II, up to $\mathcal{O}(\eps^2)$,  becomes a local DG discretization with an explicit RK method for the macroscopic compressible Navier-Stokes equations, a method in a similar spirit to the ones   
in [Bassi \& Rabey 1997, Cockburn \& Shu 1998].
Numerical results are presented for a wide range of Knudsen number to illustrate the effectiveness and high order accuracy of the methods.

\vfill

\newpage

\section{Introduction}
\label{sec1}
\setcounter{equation}{0}
\setcounter{figure}{0}
\setcounter{table}{0}

In this paper, we are interested in numerically solving the BGK equation, {a simpler relaxation model associated with the Boltzmann equation for the kinetic description of gases, introduced by Bhatnagar, Gross and Krook \cite{bhatnagar1954model}, in a hyperbolic scaling.} Knudsen number $\eps$ is an important dimensionless parameter in such description,
defined as the ratio of the molecular mean free path length to a representative physical length scale, characterizing the frequency of molecular collisions hence how rarefied a gas is. In the zero limit of Knudsen number, a sufficient macroscopic model is the Euler system describing the conservation of mass, moment and energy; {when the Knudsen number is sufficiently small but not necessarily zero}, the compressible Navier-Stokes equations are needed which include a correction term on viscosity and heat conductivity.

By far, there have been many research works in numerically simulating the Boltzmann and BGK equations with a wide range of Knudsen number.
{An elegant method based on the micro-macro decomposition framework was proposed by Bennoune, Lemou, Mieussen} \cite{bennoune2008uniformly}, and it correctly captures the macroscopic Navier-Stokes limit when the Knudsen number is sufficiently small. There are various versions of implicit-explicit schemes proposed for the BGK equations in  \cite{pieraccini2007implicit, pieraccini2012microscopically}, as well as for the ES-BGK equation \cite{filbet2011asymptotic}. For the general Boltzmann collisional operator, a novel BGK-penalization strategy was proposed by Filbet and Jin \cite{filbet2010class}. These methods are all related to the asymptotic preserving (AP) concept, for its recent development and review, see \cite{jin2010asymptotic}. Particularly AP schemes are designed to mimic the asymptotic limit from the kinetic to the hydrodynamic models on the PDE level as $\eps$ goes to 0. On the other hand, the macroscopic Navier-Stokes equations have been well studied in the computational fluid dynamics (CFD) community  by many high order shock capturing schemes \cite{shu2003high}; among others, the discontinuous Galerkin methods have been widely used \cite{bassi1997high, cockburn1998local, baumann1999discontinuous, lomtev1999discontinuous, bassi2002numerical}. There is also an interesting work on developing gas-kinetic BGK schemes for the Navier-Stokes equations by taking advantage of the kinetic distribution function as the solution of the BGK equation \cite{xu2001gas}.

Our main focus of this work is to develop a family of high order AP schemes for the BGK equation that works for a wide range of
 Knudsen number based on the micro-macro decomposition framework. The proposed methods are presented and numerically tested for constant Knudsen number $\eps$ and spatially variant $\eps=\eps(x)$. The high order spatial accuracy is achieved by nodal discontinuous Galerkin (DG) finite element approaches, and the high order temporal accuracy is achieved by globally stiffly accurate implicit-explicit (IMEX) Runge-Kutta (RK) methods. The proposed schemes become DG methods with explicit RK time discretizations for the Euler system in the zero limit of the Knudsen number. In order to capture the compressible Navier-Stokes limit for sufficiently small $\eps$, some novel ingredient, inspired by the asymptotic analysis for the continuous problem, is incorporated to further revise the schemes (see Section \ref{sec:3.1}).
 A formal asymptotic analysis shows that the resulting methods not only become DG approximations of the Euler system as $\eps\rightarrow 0$, they also give rise to local DG discretizations, up to $\mathcal{O}(\eps^2)$, of the Navier-Stokes equations, and this is verified by numerical tests in Section \ref{sec:numerical}. These local DG methods are in a similar spirit of that proposed in \cite{bassi1997high, cockburn1998local} based on a mixed formulation of the equations.

DG discretizations are widely-known {in many applications in science and engineering} for their advantages of being $h$-$p$ adaptive, compact, highly efficient in parallel implementations, and flexible for problems with complicated geometries, see \cite{cockburn2001runge, cockburn2000development} and references therein. When the methods are applied to solve PDEs with second or higher order spatial derivatives, local DG methods can be formulated based on the mixed form of the equations \cite{bassi1997high,cockburn1998local}. {Nodal (local) DG methods}, on the other hand, can be considered as DG methods with the discrete spaces represented by Lagrangian nodal basis functions \cite{hesthaven2008nodal}. In the context of solving the BGK equation based on the micro-macro decomposition, nodal DG  methods allow convenient implementations of a projection operator that is spatially dependent, as well as many kinetic and macroscopic quantities.
In addition, it simplifies the treatment when $\eps=\eps(x)$ is spatially dependent.
 For the time discretization, we treat stiff terms implicitly and non-stiff terms explicitly by adopting the high order globally stiffly accurate IMEX RK methods developed in \cite{ascher1997implicit, pareschi2005implicit}. Our proposed methodology differs from the implicit-explicit strategy in \cite{pieraccini2007implicit, pieraccini2012microscopically} by working with the macroscopic variables $U$ given in \eqref{U-vector} as well as the microscopic one $g = (f - M_U)/\eps$ (see \eqref{maxwellian}), instead of directly working with the probability distribution function $f$. Because of this, the necessity of performing the moments realignment as in \cite{pieraccini2012microscopically} is avoided.
Finally, we test the proposed schemes with a collection of smooth and non-smooth examples for a wide range of Knudsen number which can be spatially dependent.  Expected high order accuracy and correct asymptotic behavior are validated.
Superior performance, when compared with lower order schemes, is observed in terms of accuracy for smooth test cases as well as
the solution resolution when there are shock structures.

The rest of the paper is organized as follows. In Section 2, we provide the BGK equation in a hyperbolic scaling and its micro-macro decomposition. In Section 3, high order AP schemes are formulated with a nodal DG spatial discretization and a globally stiffly accurate IMEX temporal discretization. A formal asymptotic analysis is performed for the proposed methods to capture the Euler and Navier-Stokes limits. In Section 4, numerical results are presented. 
Finally, conclusions are given in Section 5. 
\section{Formulation}
\label{sec2}
\setcounter{equation}{0}
\setcounter{figure}{0}
\setcounter{table}{0}

\newcommand \smu {\sqrt{\mu}}
\newcommand \del {\partial}
\newcommand\ep {\varepsilon}


We consider the BGK equation in a hyperbolic scaling:
\begin{equation}
\partial _{t}f+v\cdot \nabla _{x}f=  \frac{1}{\ep}(M_U-f)
\label{bgk}
\end{equation}
with the initial data $f_0$ and suitable boundary conditions, where $f = f(x,v,t)$ is the distribution function of particles that depends on time $t>0$, position $x\in \Omega_x\subset \mathbb{R}^d$ and velocity $v\in  \mathbb{R}^d$ for $d\geq 1$. The parameter $\ep>0$ is the Knudsen number proportional to the mean free path. And $M_U$ is a local Maxwellian defined by
\begin{equation}
M_U=M_U(x,v,t)=\frac{\rho(x,t)}{(2\pi T(x,t))^{d/2}} \exp \left(-\frac{|v-u(x,t)|^2}{2T(x,t)} \right)
\label{maxwellian}
\end{equation}
where $\rho$, $u$, $T$ represent the macroscopic density, the mean velocity, and the temperature, respectively, and they are obtained by taking the first few moments of $f$:
\begin{equation}\label{U-vector}
U:=\left(
\rho,
\rho u,
\frac12\rho|u|^2 +\frac{d}{2}\rho T
\right)^t = \int_{\mathbb{R}^d} \left(
1,
v,
\frac12 |v|^2
\right)^t   f(v)dv.
\end{equation}
Here the components of $U$ represent the density, momentum, and energy.  {We refer to \cite{cercignani1969mathematical, cercignani1988boltzmann, cercignani2000rarefied} for more details on the model.}

In what follows we derive the fluid equations starting from \eqref{bgk}.
For notational convenience, we use $
m=m(v):= \left( 1,v, \frac12 |v|^2 \right)^t$  and $\langle g \rangle: = \int_{\mathbb{R}^d} g(v) dv$.
It is easy to check that $\langle m M_U  \rangle = \left(
\rho,
\rho u,
\frac12\rho|u|^2 +\frac{d}{2}\rho T
\right)^t = U $ and hence we see that $\langle  m (M_U-f) \rangle =0$, namely the BGK operator satisfies the conservation of mass, momentum and energy. Moreover, it enjoys the entropy dissipation:
$\langle (M_U-f) \log f\rangle \leq 0$. From the conservation properties of the BGK operator, we get at least formally the local conservation of mass, momentum, and energy:
\beq
\label{eq: Euler}
\begin{split}
\del_t \left(
\begin{array}{c} \rho\\ \rho u \\ E
\end{array}
 \right)+\nabla_x\!\cdot\! \left(
 \begin{array}{c}
 \rho u \\ \rho u\otimes u +P \\ E u+ P u +Q
 \end{array}
 \right) =0
\end{split}
\eeq
where $E=\frac{1}{2}\rho |u|^2+\frac{d}{2}\rho T $, the pressure tensor $P$ is given by $P=\langle (v-u)\otimes (v-u) f\rangle$, and the heat flux vector is given by $Q=\frac12\langle  (v-u)|v-u|^2f\rangle$.
When $\ep\rightarrow 0$ in
\eqref{bgk}, $f$ approaches the Maxwellian $M_U$ in \eqref{maxwellian}. Hence, for sufficiently small $\ep$, $f$ can be approximated by this Maxwellian.
 In such an approximation,  $P=pI$, with $p=\rho T$ and $I$ as the $d\times d$ identity matrix, $Q=0$, and  thus the above local conservation laws form a closed system, which is the compressible Euler system. The compressible Navier-Stokes equations are obtained by the classical Chapman-Enskog expansion {\cite{cercignani1969mathematical, cercignani1988boltzmann,  cercignani2000rarefied, chapman1970mathematical}.} Next we will present the micro-macro decomposition of \eqref{bgk}, which has a similar spirit of the Chapman-Enskog expansion, to derive the compressible Navier-Stokes equations. This decomposition will provide the starting point of the proposed numerical methods in this work.

\subsection{Micro-macro formulation}

Let $M$ be a given local Maxwellian. We use $L^2_M$ to denote the Hilbert space equipped with the following weighted inner product
\[
( f,g)_M:= \langle f  g M^{-1} \rangle.
\]
Then any function $f\in L^2_M$ can be written as the unique orthogonal decomposition as follows
\[
f= \Pi_M f +(\mathbf{I}-\Pi_M)f
\]
where $\mathbf{I}$ is the identity operator, ${\Pi_M}f$ is the orthogonal projection in $L^2_M$ onto
$\mathcal{N}:=\text{span}\,\{M,vM,|v|^2M\}$ and its explicit form by using the orthogonal basis of $\mathcal{N}$ is given by
\begin{equation}\label{pi}
 {\Pi_M}f = \left( \frac1\rho\langle f\rangle +\frac{\langle (v-u)f\rangle}{\rho T}\cdot(v-u) +\frac{2}{d\rho} \langle (\frac{|v-u|^2}{2T}-\frac{d}{2}) f \rangle (\frac{|v-u|^2}{2T}-\frac{d}{2})  \right) M.
\end{equation}

We are now ready to present the micro-macro decomposition {\cite{bennoune2008uniformly, liu2004boltzmann}} for the BGK equation \eqref{bgk}. {Here and below, we will use $M$ to denote $M_U$ in \eqref{maxwellian}.}
 The starting point  is to seek the solution $f$ of \eqref{bgk}  as
\begin{equation}\label{mm}
f := M+\ep g
\end{equation}
so that $g$ is only microscopic:
$
\langle m g \rangle =0
$.
In a sense, the solution $f$ is decomposed into the macroscopic part $M$ and the microscopic part $\eps g$. 
If we insert \eqref{mm} into \eqref{bgk}, we obtain
\begin{equation}\label{mg}
\del_t M + v\!\cdot\!\nabla_x M +\ep( \del_t g + v\!\cdot\!\nabla_x g ) = -g \,.
\end{equation}
The idea of the micro-macro decomposition is to decompose \eqref{mg} through the orthogonal projections
${\Pi_{M}}$ and $\mathbf{I}-\Pi_{M}$. To do so, it is useful to recall the followings:
\begin{equation}\label{prop}
{\Pi_{M}}\,g=0, \quad (\mathbf{I}-\Pi_{M}) \,\del_tM=0 , \quad {\Pi_{M}}\,\del_t g=0
\end{equation}
which follow from the definition and direct computations {(for instance, see Lemma 3.1 in \cite{bennoune2008uniformly})}. The orthogonal projection $\mathbf{I}-\Pi_{M}$ of the equation \eqref{mg} reads as
\begin{equation*}
(\mathbf{I}-\Pi_{M}) (\del_t M +v\!\cdot\!\nabla_x M) +\ep (\mathbf{I}-\Pi_{M})  (\del_t g + v\!\cdot\!\nabla_x g)
= -(\mathbf{I}-\Pi_{M}) g
\end{equation*}
which in turn can be, by using \eqref{prop}, written as
\begin{equation}\label{geq}
\eps \del_t g + \eps (\mathbf{I}-\Pi_{M}) (v\!\cdot\!\nabla_x g)
=- \big( g +(\mathbf{I}-\Pi_{M})  (v\!\cdot\!\nabla_x M)\big).
\end{equation}
On the other hand, the projection ${\Pi_{M}}$ of the equation \eqref{mg} gives rise to
\begin{equation}
\begin{split}
\del_t \left(
\begin{array}{c} \rho\\ \rho u \\ E
\end{array}
 \right)+\nabla_x\!\cdot\! \left(
 \begin{array}{c}
 \rho u \\ \rho u\otimes u +p I  \\   (E+p)u
 \end{array}
 \right)  + \ep \nabla_x\!\cdot\!  \left(
 \begin{array}{c} \langle v g\rangle \\ \langle  v\otimes v g \rangle \\ \langle v \frac{|v|^2}{2} g \rangle
\end{array}
 \right) =0.
\end{split}\label{cns-bgk}
\end{equation}
Denoting the flux terms in \eqref{cns-bgk} by $F(U)$, we have derived the following micro-macro decomposition of \eqref{bgk}:
\begin{subequations}
\label{eq:mmdc}
\begin{align}
&\partial_t U + \nabla_x\!\cdot\!  F(U)+\eps  \nabla_x\!\cdot\!  \langle vmg \rangle=0, \label{equ}\\
&\eps \del_t g +  \eps (\mathbf{I}-\Pi_{M}) (v\!\cdot\!\nabla_x g)
=- \big( g +(\mathbf{I}-\Pi_{M})  (v\!\cdot\!\nabla_x M)\big). \label{eqg}
\end{align}
\end{subequations}

In a more general setting where the Knudsen number $\ep$ depends on the position $x$: $\ep=\ep(x)$, since $\nabla_x(\ep g) \neq  \ep\nabla_xg  $ in general, the micro-macro formulation \eqref{eq:mmdc} should be written as follows:

\begin{subequations}
\label{eq:mmd}
\begin{align}
&\partial_t U +\nabla_x\!\cdot\!  F(U) + \nabla_x \!\cdot\! \big( \eps(x)  \langle vmg \rangle \big)=0, \label{equx} \\
&\eps(x)\partial_t g + ({\mathbf I}-\Pi_M)\nabla_x \!\cdot\!\big(\eps(x)vg\big)=-\big( g+({\mathbf  I}-\Pi_M)(v \!\cdot\!  \nabla_x M) \big). \label{eqgx}
\end{align}
\end{subequations}


\subsection{Compressible Navier-Stokes limit}

The first two terms in \eqref{cns-bgk} form the Euler system and we see that as $\ep\rightarrow 0$, the equations \eqref{cns-bgk} at least formally converge to the Euler system. In this subsection, we want to examine the contribution of the third term in \eqref{cns-bgk} and indeed, we will show that the inclusion of the leading order ($\ep$ term)  gives rise to  the set of  the compressible Navier-Stokes equations. From \eqref{geq}, we see that
\begin{equation}
g = -(\mathbf{I}-\Pi_{M}) (v\!\cdot\!\nabla_x M) +{\mathcal O}(\ep)
\end{equation}
and the direct computation shows that
\begin{equation}
\label{eq: analytic_deri}
\frac{(\mathbf{I}-\Pi_{M}) (v\!\cdot\!\nabla_x M) }{M}= B: \left( \nabla_xu+(\nabla_xu)^t -\frac{2}{d}(\nabla_x\!\cdot\! u)I \right) + A\cdot \frac{\nabla_xT}{\sqrt{T}}
\end{equation}
where
\beq
\label{eq: AB}
A=\left(  \frac{|v-u|^2}{2T}-\frac{d+2}{2} \right)\frac{v-u}{\sqrt{T}}\;\text{ and }\;B=\frac12
\left( \frac{(v-u)\otimes(v-u)}{2T} -\frac{|v-u|^2}{dT}I \right)
\eeq
and therefore, we deduce that
\begin{equation}\label{AB}
g =-  B: \left( \nabla_xu+(\nabla_xu)^t -\frac{2}{d}(\nabla_x\!\cdot\! u)I \right) M - A\cdot \frac{\nabla_xT}{\sqrt{T}} M +{\mathcal O}(\ep).
\end{equation}
Here $X:Y = \sum_{i,j}X_{ij}Y_{ij}$ is the Frobenius inner product for matrices.
As we insert this expression \eqref{AB} into \eqref{cns-bgk}, we obtain
\begin{equation}
\begin{split}
\del_t \left(
\begin{array}{c} \rho\\ \rho u \\ E
\end{array}
 \right)+\nabla_x\!\cdot\! \left(
 \begin{array}{c}
 \rho u \\ \rho u\otimes u +p I  \\ (E+p)u
 \end{array}
 \right) = \ep  \left(
 \begin{array}{c} 0 \\ \nabla_x\!\cdot\!\sigma   \\ \nabla_x\!\cdot\!(\sigma u+ q)
\end{array}
 \right) +{\mathcal O}(\ep^2)
\end{split}\label{cns-bgk2}
\end{equation}
where
\beq
\label{eq: sigma_q}
\sigma = \mu \left( \nabla_xu+(\nabla_xu)^t -\frac{2}{d}(\nabla_x\!\cdot\! u)I \right)\text{ and } q=\kappa \nabla_xT
\eeq
and
\[
\mu =T \langle B : B M \rangle \text{ and }\kappa = T\langle A\cdot A M \rangle.
\]
{We refer to \cite{bardos1991fluid} for more details on the derivation.} The above system \eqref{cns-bgk2} is the compressible Navier-Stokes equations if we disregard high order terms ${\mathcal O}(\ep^2)$. We note that when $d=1$, $\sigma=0$ and $\kappa=\frac{3}{2}\rho T$. It is worthwhile pointing out that while the BGK equation shares the basic properties of hydrodynamics with the Boltzmann equation, the Navier-Stokes equations derived from those equations display different viscosity and heat conductivity coefficients {\cite{cercignani1969mathematical, cercignani1988boltzmann, cercignani2000rarefied}.}



\begin{rem}
 From our derivation of Euler or Navier-Stokes system from the BGK equation, we have obtained
 \beq
 p=\rho T, \quad E=\frac{1}{2}\rho |u|^2+\frac{d}{2}\rho T.
 \label{eq: T}
 \eeq
On the other hand, in gas dynamics for an ideal polytropic gas, the total energy is given by $E=\frac12 \rho |u|^2 + \frac{p}{\gamma-1}$ via the constitutive relation between the pressure and internal energy. Therefore, we obtain $\gamma= \frac{d+2}{d}$ which represents the constant ratio of specific heats. 
\end{rem}


\section{NDG-IMEX Methods}
\label{sec:NDG-IMEX:3}

\setcounter{equation}{0}
\setcounter{figure}{0}
\setcounter{table}{0}

In this section, we propose numerical schemes to solve the system \eqref{eq:mmd}, and they are based on nodal discontinuous Galerkin (NDG) methods in space together with implicit-explicit (IMEX) time discretizations. Since the purpose of the present work is to introduce new algorithms, we will focus on the one-dimensional case with $d=1$, $\Omega_x=[a, b]$ and $\Omega_v=[-V_c, V_c]$. $V_c$
is chosen sufficiently large so that the Maxwellian
defined in \eqref{maxwellian} can be regarded as zero outside $\Omega_v$ numerically.
Most ingredients of the proposed methods can be applied directly to higher dimensions (see Remark \ref{rem:highD}).

\subsection{Semi-discrete NDG methods}
\label{sec:3.1}

Start with a partition of $\Omega_x$, $a=x_{\frac12}<x_{\frac{3}{2}}<\cdots <x_{N_x+\frac12}=b$. Denote an element as
$I_i=[\xL, \xR]$ with length $h_i$, and let $h=\max_{i=0}^{N_x}h_i$. Given any non-negative integer $K$, we define a finite dimensional discrete space,
\begin{equation}
Z_h^K=\left\{z\in L^2(\Omega_x): z|_{I_i}\in P^K(I_i), \forall i\right\},
\label{eq:DiscreteSpace}
\end{equation}
and its vector version is denoted as ${\bf Z}_h^K$. The local space $P^K(I)$ consists of polynomials of degree at most $K$ on $I$.
Note that functions in $Z_h^K$ are piecewise defined. To distinguish the left and right limits of a function $z\in Z_h^K$ at a grid point $x_{i+\frac{1}{2}}$, we let
$z_{i+\frac{1}{2}}^\pm=\lim_{\Delta x\rightarrow \pm 0}z(x_{i+\frac{1}{2}}+\Delta x)$, and we also let $[z]_{i+\frac{1}{2}}=z_{i+\frac{1}{2}}^+-z_{i+\frac{1}{2}}^-$ as the jump.

Following the general procedure to formulate DG discretizations and the development in \cite{JLQX_diffusive}, we first propose a semi-discrete DG method for the micro-macro system \eqref{eq:mmd}. Find
$U_h(\cdot,t)\in {\bf Z}_h^K$ and $g_h(\cdot,v,t)\in Z_h^K$, such that
$\forall \phi, \psi\in Z_h^K$ and $\forall i$,
\begin{subequations}
\label{eq:SDG}
\begin{align}
\int_{I_i} \partial_t U_h \phi dx=&\int_{I_i}\left( F(U_h)+\eps(x) \langle vm g_h\rangle\right)  \frac{d \phi(x)}{dx} dx - \hat{F}_\iR \phi^-_\iR
+ \hat{F}_\iL \phi^+_\iL   \notag \\
&-\eps(x_\iR)  \widehat{\langle vmg_h\rangle}_\iR \phi_\iR^-+\eps(x_\iL) \widehat{\langle vmg_h\rangle}_\iL \phi^+_\iL,
\label{eq:SDG:a} \\
\int_{I_i} \eps(x) \partial_t g_h\psi dx+&
\int_{I_i} ({\mathbf I}-\Pi_{M_h})\left(\mD_{h,1}(\eps(x) vg_h)\right)\psi dx =- \int_{I_i}  g_h\psi dx
- \int_{I_i} ({\mathbf I}-\Pi_{M_h})\left(\mD_{h,2}(vM_h)\right)\psi dx. \label{eq:SDG:b}
\end{align}
\end{subequations}
Here $M_h=M_{U_h}$ according to \eqref{maxwellian}. In addition, $\mD_{h,1}(\eps(x)vg_h)(\cdot,v,t)\in Z_h^K$ and $\mD_{h,2}(v M_h)(\cdot,v,t)\in Z_h^K$ are approximations of the spatial derivative of $\eps(x)vg$ and $vM$, respectively, based on DG discretizations. Particularly, $\forall \psi\in Z_h^K$ and $\forall i$,
\begin{equation}
\int_{I_i} \mD_{h,1}(\eps(x) vg_h)\psi dx:=
-\int_{I_i} \eps(x) vg_h \frac{d \psi}{dx} dx+\eps(x_\iR) \widetilde{(vg_h)}_{\iR} \psi^-_\iR
-\eps(x_\iL) \widetilde{(vg_h)}_{\iL} \psi^+_\iL,
\label{eq:Dd1}
\end{equation}
where $\widetilde{vg}$ is an upwind numerical flux consistent to $vg$,
\beq
\label{eq:vg:upwind:L-1}
\widetilde{vg}:=
\left\{
\begin{array}{ll}
v g^-,&\mbox{if}\; v>0,\\
v g^+,&\mbox{if}\; v<0,
\end{array}
\right.
\eeq
and
\beq
\int_{I_i} \mD_{h,2}(vM_h)\psi dx:=
-\int_{I_i} vM_h \frac{d \psi}{dx} dx+v\widehat{M}_{h,\iR} \psi^-_\iR
-v \widehat{M}_{h,\iL} \psi^+_\iL.\label{eq:Dd2}
\eeq
The hatted functions in \eqref{eq:SDG:a} and \eqref{eq:Dd2} are also consistent numerical fluxes. In this work,
we take one of the following alternating fluxes,
\begin{equation}
\label{eq: flux}
\mbox{alternating left-right}: \widehat{\langle vmg\rangle} = {\langle vmg\rangle}^-, \widehat{M} = M^+; \quad
\mbox{right-left}:  \widehat{\langle vmg\rangle} = {\langle vmg\rangle}^+, \widehat{M} = M^-.
\end{equation}
Similarly as in \cite{JLQX_diffusive}, one can also use the central fluxes
\begin{equation}
\mbox{central}:  \widehat{\langle vmg\rangle}=\frac12(\langle vmg\rangle^++\langle vmg\rangle^-), \quad
\widehat{M}=\frac12(M^++M^-).
\end{equation}
The numerical flux $\hat{F}=\hat{F}(U_h^-, U_h^+)$ in \eqref{eq:SDG:a} is taken to be the global Lax-Friedrichs flux \cite{cockburn1989tvb2}.
Here the subscripts $i\pm\frac{1}{2}$ are temporarily omitted for simplicity. From now on, we will call the numerical method introduced above
Scheme I.

An alternative discretization is to take advantage of the relation \eqref{eq: analytic_deri}, which, in one dimension, is
\beq
\label{eq: analytic_deri:1d}
(\mathbf{I}-\Pi_{M})  (v\partial_x M)=A\frac{\partial_x T}{\sqrt{T}}M.
\eeq
With this, the equation \eqref{eqgx} is equivalent to
\beq
\eps(x) \partial_t g + ({\mathbf I}-\Pi_M)\partial_x(\eps(x) vg)=-\left( g+A\f{\partial_x T}{\sqrt{T}} M \right). \label{eqgx2}
\eeq
Our Scheme II is formulated by replacing \eqref{eq:SDG:b} with the following DG discretization of \eqref{eqgx2}:
find $g_h(\cdot, v, t) \in Z_h^K$, such that $\forall \psi \in Z_h^K$ and $\forall i$,
\beq
\int_{I_i} \eps(x) \partial_t g_h\psi dx+
\int_{I_i} ({\mathbf I}-\Pi_{M_h})\left(\mD_{h,1}(\eps(x) vg_h)\right)\psi dx =- \int_{I_i}  g_h\psi dx
- \int_{I_i} A_h \frac{r_h}{\sqrt{T_h}} M_h \psi dx,\label{eq:SDG:c}
\eeq
where $r_h$ is to approximate $\partial_x T$ through a DG discretization: find $r_h \in Z_h^K$ such that $\forall \varphi \in Z_h^K$ and $\forall i$
\beq
\label{eq: D3}
\int_{I_i} r_{h}\varphi dx = -\int_{I_i} T_h \frac{d \varphi}{dx} dx+\widehat{T}_{h,\iR} \varphi^-_\iR- \widehat{T}_{h,\iL} \varphi^+_\iL.
\eeq
Here $T_h$, a macroscopic quantity, and $A_h$
can be obtained from $U_h$ based on \eqref{eq: T} and \eqref{eq: AB}.
   Similar to Scheme I, choices of fluxes for  the pair $\widehat{ \langle vmg\rangle}$ and $\widehat{T}$ include the alternating and the central ones,
\begin{subequations}
\label{eq: flux_2}
\begin{align}
&\mbox{alternating left-right}: \widehat{\langle vmg\rangle} = {\langle vmg\rangle}^-, \widehat{T} = T^+; \quad
\mbox{right-left}:  \widehat{\langle vmg\rangle} = {\langle vmg\rangle}^+, \widehat{T} = T^-,\\
&\mbox{central}: \widehat{\langle vmg\rangle}=\frac12(\langle vmg\rangle^++\langle vmg\rangle^-), \quad \widehat{T} = \frac12(T^+ + T^-).
\end{align}
\end{subequations}

Scheme I is formulated very intuitively, yet Scheme II shows several advantages in both computational cost and in
asymptotic analysis.
First of all, by using the analytical formula \eqref{eq: analytic_deri},
the projection operator is avoided in actual implementation and this will save some computational cost.
Secondly, the spatial derivative on the right side of \eqref{eq: analytic_deri:1d} is
for a macroscopic variable $T$ independent of $v$, in contrast with the one on the
left, this will further reduce the computational cost of the scheme by computing $r_h$ in \eqref{eq: D3}
instead of $\mD_{h,2}(vM_h)$ in \eqref{eq:Dd2}. More importantly,
 a formal asymptotic analysis will be available for establishing that the proposed Scheme II
  for the micro-macro decomposition of the kinetic equations \eqref{equx} and \eqref{eqgx2}
  becomes a local DG discretization for the Navier-Stokes system,
  a discretization in a similar spirit to that proposed in \cite{bassi1997high} (see Section~\ref{sec:3.3} for the analysis and for more discussions).


To implement the proposed schemes, we further use the nodal basis to represent functions in the discrete space $Z_h^K$,
in conjunction with rewriting and/or approximating the integrals in the schemes by numerical quadratures.
Note that the discrete space $Z_h^K|_{I_i}$ is simply $P^K(I_i)$. We particularly choose the local nodal basis
(also called Lagrangian basis) $\{\phi_i^k(x)\}_{k=0}^K$ associated with the $K+1$ Gaussian quadrature points $\{x^k_i\}^K_{k=0}$ on $I_i$, defined as below
\begin{equation}
\phi_i^k(x) \in P^K(I_i), \quad\textrm{and} \quad \phi_i^k (x_i^{k'})=\delta_{k k'},\quad k, k'=0,1,\cdots, K.
\end{equation}
Here $\delta_{k k'}$ is the Kronecker delta function. We further let $\{\omega_k\}^K_{k=0}$ denote the corresponding quadrature weights on the reference element
$(-\frac{1}{2}, \frac{1}{2})$.

Once the basis functions are specified, {Scheme I  in the integral form}, defined by \eqref{eq:SDG}-\eqref{eq:Dd1} and \eqref{eq:Dd2}, can be equivalently stated with the test functions $\phi$, $\psi$ both being taken as $\phi_i^k, k=0, 1, \cdots, K$. We also replace all the integral terms in {\eqref{eq:SDG}-\eqref{eq:Dd1} and \eqref{eq:Dd2}} by their numerical integrations based on $(K+1)$-point Gaussian quadrature. The scheme now becomes: find
$U_h(\cdot,t)\in {\bf Z}_h^K$ and $g_h(\cdot,v,t)\in Z_h^K$, with
$U_h(x,t)|_{I_i}=\sum_{k=0}^K U_i^k(t) \phi_i^k(x)$, $g_h(x,v,t)|_{I_i}=\sum_{k=0}^K g_i^k(v,t)\phi_i^k(x)$, such that $\forall i, k$,
\begin{subequations}
\label{eq:SDG1}
\begin{align}
\omega_k h_i \frac{d U^k_i}{dt}=& \sum_{k'=0}^K \omega_{k'} h_i F(U^{k'}_i)\frac{d \phi^k_i(x)}{dx}\big |_{x=x^{k'}_i}-\hat F_\iR \phi_i^k(x_\iR^-)+\hat F_\iL \phi_i^k(x_\iL^+)\nonumber \\
&+\sum_{k'=0}^K \omega_{k'}h_i\eps(x^{k'}_i)\langle vm g^{k'}_i\rangle \frac{d \phi^k_i(x)}{dx}\big |_{x=x^{k'}_i}
-\eps(x_{\iR})\widehat{\langle vmg_h\rangle}_\iR \phi_i^k(x_\iR^-) \nonumber\\
&+\eps(x_\iL)\widehat{\langle vmg_h\rangle}_\iL \phi_i^k(x_\iL^+), \label{eq:SDG1:a} \\
\eps(x^k_i)\omega_k h_i \df_t g^k_i=&({\mathbf I}-\Pi_i^k)\Bigg(v\sum_{k'=0}^K \omega_{k'} h_i \eps(x^{k'}_i)g^{k'}_i \frac{d \phi^k_i(x)}{dx}\big |_{x=x^{k'}_i}
-\eps(x_{\iR})\widetilde {(vg_h)}_\iR \phi_i^k(x_\iR^-) \nonumber \\
&+\eps(x_{\iL})\widetilde{(vg_h)}_\iL \phi_i^k(x_\iL^+)\Bigg)- \omega_k h_i g^k_i  +({\mathbf  I}-\Pi_i^k)v \Bigg(\sum_{k'=0}^K \omega_{k'} h_i M^{k'}_i \frac{d \phi^k_i(x)}{dx}\big |_{x=x^{k'}_i}\nonumber \\
&-\widehat{M}_{h,\iR} \phi_i^k(x_\iR^-)+\widehat{M}_{h,\iL} \phi_i^k(x_\iL^+) \Bigg). \label{eq:SDG1:b}
\end{align}
\end{subequations}
Here $M_i^{k'}={M_h}|_{x=x^{k'}_i}$ and $\Pi_i^k=\Pi_{M_i^k}$. Similarly, for Scheme II in the integral form,
instead of the equation \eqref{eq:SDG1:b},
we have nodal discretizations of equations \eqref{eq:SDG:c}-\eqref{eq: D3} given below,
\beqa
\eps(x^k_i)\omega_k h_i \df_t g^k_i=&({\mathbf I}-\Pi_i^k)\Bigg(v\sum_{k'=0}^K \omega_{k'} h_i \eps(x^{k'}_i)g^{k'}_i \frac{d \phi^k_i(x)}{dx}\big |_{x=x^{k'}_i}
-\eps(x_{\iR})\widetilde {(vg_h)}_\iR \phi_i^k(x_\iR^-) \nonumber \\
\label{eq:SDG2:c}
&+\eps(x_{\iL})\widetilde{(vg_h)}_\iL \phi_i^k(x_\iL^+)\Bigg)- \omega_k h_i g^k_i  + A_i^k \omega_k h_i {r_i^k}M_i^k/{\sqrt{T_i^k}},
\eeqa
where the nodal values of $T_i^k=T_h|_{x=x^{k}_i}$ and $A_i^k=A_h|_{x=x^{k}_i}$ are obtained from $U_i^k$ based on
\eqref{eq: T} and \eqref{eq: AB}, and $r_i^k=r_h|_{x=x^{k}_i}$ is computed with the following scheme,
\beq
\label{eq:SDG2:d}
\omega_k h_i r_i^k = - \sum_{k'=0}^K \omega_{k'} h_i T_i^{k'} \frac{d \phi^k_i(x)}{dx}\big |_{x=x^{k'}_i}
+ \widehat{T}_{h,\iR} \phi_i^k(x_\iR^-) -\widehat{T}_{h,\iL} \phi_i^k(x_\iL^+).
\eeq

Note our final Scheme I (\eqref{eq:SDG1:a}-\eqref{eq:SDG1:b}) and Scheme II (\eqref{eq:SDG1:a},\eqref{eq:SDG2:c}, \eqref{eq:SDG2:d})
in their nodal forms are obtained by applying $(K+1)$-point Gaussian quadrature to the integral terms in the original scheme \eqref{eq: compact}. 
Since such quadrature rule is exact for polynomials of degree up to $2K+1$, the final schemes will maintain the same formal accuracy 
as \eqref{eq: compact} (see \cite{cockburn2001runge}).




Due to $\langle mg \rangle=0$, the first two components of $\langle vmg \rangle$, namely with $m=1, v$, are zero for the exact solution $g$. In our numerical implementation,
we still keep these two components of $\langle vmg \rangle$ in \eqref{equx} and in \eqref{eq:SDG:a}, \eqref{eq:SDG1:a}.
To implement the proposed methods, we also need to discretize the $v$-direction. In this work, $\Omega_v=[-V_c, V_c]$ is discretized uniformly with $N_v$ points,  $\{v_j\}_{j=1}^{N_v}$. For the integration in $v$,
the mid-point rule is applied, which is spectrally accurate for smooth functions with periodic 
boundary conditions  or with a compact support \cite{boyd2001chebyshev}.
Such approach does not preserve the conservation properties of mass, moment and energy at the discrete level as in \cite{mieussens2000discrete},
yet we have found it is a sufficiently accurate discretization for all test cases that we have performed. We also demonstrate
how such conservation quantities behave
over time in the numerical section.

\begin{rem}
One advantage to work with the nodal basis with respect to $(K+1)$ Gaussian quadrature points is to avoid inverting element mass matrix  in \eqref{eq:SDG:b} when $\eps(x)$ is not a constant function.
With the nodal DG discretization, we also greatly simplify the treatments of
  the projection operator ${\mathbf I}-\Pi_{M_h}$ and many kinetic and macroscopic quantities such as $A_h, M_h$ and $T_h$ in the scheme.
\end{rem}

\begin{rem}
\label{rem:highD}
Though the proposed method is formulated for one dimension, most ingredients can be extended directly to higher dimensions.
More specifically, Scheme I in its integral form, \eqref{eq:SDG}-\eqref{eq:Dd1} and \eqref{eq:Dd2}, can be formulated
for high dimensional cases straightforwardly.  To get the nodal version as defined in \eqref{eq:SDG1:a}-\eqref{eq:SDG1:b},
one would need to choose the points to define the nodal basis functions. Tensor-structured Gaussian points can be used directly
on Cartesian meshes when the approximating functions are piecewise  tensor polynomials. One can refer to \cite{hesthaven2008nodal}
for more discussions on higher dimensions.
As for Scheme II, the relation in \eqref{eq: analytic_deri:1d} needs to be replaced
by the general one in \eqref{eq: analytic_deri}. It can be discretized similarly as in one dimensional case, with a few more numerical fluxes to be specified.
\end{rem}


\begin{rem}
The proposed spatial discretization above shares some similarity with the method proposed in \cite{bennoune2008uniformly} in terms of utilizing the micro-macro decomposition framework.
The proposed nodal DG methods belong to the class of finite element methods and can be designed to be of arbitrary order of accuracy;
the methods are based on one set of computational grid. This is in contrast to the first order finite difference method
in  \cite{bennoune2008uniformly} with $U$ and $g$ defined on different meshes.
Moreover, the novel ingredient newly proposed for Scheme II offers not only computational saving,
but also an asymptotical analysis to capture the compressible Navier-Stokes limit for small $\eps$, see Sections~\ref{sec:3.3} and~\ref{sec:numerical} for more discussion.
\end{rem}

\subsection{IMEX time discretization}
\label{sec:3.2}

In this section, we will formulate the IMEX Runge-Kutta (RK) time discretizations for the semi-discrete schemes introduced in Section \ref{sec:3.1}. With similarity and for notational clarity,  this will be presented only for Scheme II in its integral form. Scheme I and the nodal form of Scheme II can be  discussed similarly. Note that Scheme II in the integral form can be given compactly as follows. Find $U_h(\cdot,t)\in {\bf Z}_h^K$, $g_h(\cdot,v,t), r_h(\cdot, t)\in Z_h^K$, such that
$\forall \phi, \psi, \varphi\in Z_h^K$ and $\forall i$,
\begin{subequations}
\label{eq: compact}
\begin{align}
(\partial_t U_h, \phi) + F_h(U_h, \phi) &= D_h( \eps(x) g_h, \phi), \label{eq: compact1}\\
(\eps(x) \partial_t g_h, \psi) + b_{h, v}(\eps(x)g_h, \psi) &= s^{(1)}_h(g_h, \psi) + s^{(2)}_{h, v}(U_h, r_h, \psi), \label{eq: compact2}\\
(r_h, \varphi) =  H_h(U_h, \varphi), \label{eq: compact3}
\end{align}
\end{subequations}
where
\begin{subequations}
\label{eq: S2}
\begin{align}
F_h(U_h, \phi) &= - \int_{\Omega_x} F(U_h)\frac{d \phi(x)}{dx} dx - \sum_i \hat{F}_{h, i+\frac12} [\phi]_{i+\frac12},\\
D_h( \eps(x) g_h, \phi) & =  \int_{\Omega_x} \eps(x) \langle vm g_h\rangle\frac{d \phi(x)}{dx} dx + \sum_i \eps(x_{i+\frac12}) \widehat{\langle vmg_h\rangle}_\iR [\phi]_{i+\frac12}, \\
b_{h, v}(\eps(x)g_h, \psi) &= \int_{\Omega_x} ({\mathbf I}-\Pi_{M_h}) \mD_{h,1}(\eps(x) vg_h)\psi dx,\\
s^{(1)}_h(g_h, \psi) &=-\int_{\Omega_x} g_h\psi dx, \quad
s^{(2)}_{h, v}(U_h, r_h, \psi) =  -\int_{\Omega_x} A_h \frac{r_h}{\sqrt{T_h}} M_h \psi dx, \\
H_h(U_h, \varphi)&= {-\left(\int_{\Omega_x} T_h \frac{d\varphi}{dx} dx + \sum_i \hat{T}_{h, i+\frac12} [\varphi]_{i+\frac12}\right)\big |_{T_h=T_h(U_h)}.}
\end{align}
\end{subequations}

To discretize in time for the scheme in \eqref{eq: compact}, we start with a first order IMEX scheme to introduce our implicit-explicit strategy. Given $U_h^n \in {\bf Z}_h^K$ and $g_h^n \in Z_h^K$ that approximating the solutions $U$ and $g$ at $t=t^n$, respectively,
we look for $U_h^{n+1} \in {\bf Z}_h^K$ and $g_h^{n+1}, {r_h^{n+1}} \in Z_h^K$, such that $\forall \phi, \psi, \varphi \in Z_h^K$,

\begin{subequations}
\label{eq: compact_d}
\begin{align}
\left(\frac{U^{n+1}_h-U^n_h}{\Dt}, \phi \right) + F_h(U^n_h, \phi) &= D_h( \eps(x) g^n_h, \phi), \label{eq: compact_d1}\\
\left(\eps(x) \frac{g^{n+1}_h-g^n_h}{\Dt}, \psi\right) + b_{h, v}(\eps(x)g^n_h, \psi) &= s^{(1)}_h(g^{n+1}_h, \psi) + s^{(2)}_{h, v}(U^{n+1}_h, r_h^{n+1},  \psi), \label{eq: compact_d2} \\
{(r_h^{n+1}, \varphi)} &{= H_h(U_h^{n+1}, \varphi).} \label{eq: compact_d3}
\end{align}
\end{subequations}
This fully discrete scheme can be implemented efficiently. Specifically, one can solve the equation \eqref{eq: compact_d1} for macroscopic variables $U_h^{n+1}$ at the updated time level first, then solve for $r_h^{n+1}$ from \eqref{eq: compact_d3}. Finally one can solve for $g_h^{n+1}$ from \eqref{eq: compact_d2}.

This implicit-explicit procedure can be easily extended to high order globally stiffly accurate IMEX schemes, which can be characterized by a double Butcher Tableau
\beq
\label{eq: B_table}
\begin{array}{c|c}
\tilde{c} & \tilde{A}\\
\hline
 & \tilde{b}^ t\end{array} \ \ \ \ \
\begin{array}{c|c}
{c} & {A}\\
\hline
 & {b^t} \end{array},
\eeq
where $\tilde{A} = (\tilde{a}_{ij})$ is an $s\times s$  lower triangular matrix with zero diagonal for an explicit scheme, and $A = (a_{ij})$ is an $s\times s$ lower triangular matrix with the diagonal entries not all being zero for a diagonally implicit RK (DIRK) method. The coefficients $\tilde{c}$ and $c$ are given by the standard relations
\begin{eqnarray}\label{eq:candc}
\tilde{c}_i = \sum_{j=1}^{i-1} \tilde a_{ij}, \ \ \ c_i = \sum_{j=1}^{i} a_{ij},
\end{eqnarray}
and vectors $\tilde{b} = (\tilde{b}_j)$ and $b = (b_j)$ represent the quadrature weights for internal stages of the RK method. The IMEX RK scheme is defined to be {\em globally stiffly accurate} if $\tilde{c}_s = c_s = 1$ and {$a_{sj} = b_j$, $\tilde{a}_{sj} = \tilde{b}_j$}, $\forall j=1, \cdots, s$. The fully-discrete scheme using the Butcher notation can be written as
follows. Given $U_h^n\in {\bf Z}_h^K$ and $g_h^n \in Z_h^K$, we look for $U_h^{n+1}\in {\bf Z}_h^K$ and $g_h^{n+1} \in Z_h^K$, such that $\forall \phi, \psi \in Z_h^K$,
\begin{subequations}
\label{eq: compact_d_rk}
\begin{align}
\left(U^{n+1}_h, \phi \right) &= \left(U^{n}_h, \phi \right) -\Dt \sum_{l=1}^{s} \tilde{b}_l \left(F_h(U^{(l)}_h, \phi) -D_h( \eps(x) g^{(l)}_h, \phi)\right), \label{eq: compact_d_rk1}\\
\left(\eps(x) g^{n+1}_h, \psi\right) & =\left(\eps(x) g^{n}_h, \psi\right)- \Dt \sum_{l=1}^{s} \tilde{b}_l b_{h, v}(\eps(x)g^{(l)}_h, \psi) + \Dt \sum_{l=1}^{s} b_l \left(s^{(1)}_h(g^{(l)}_h, \psi) + s^{(2)}_{h, v}(U^{(l)}_h, r^{(l)}_h, \psi)\right). \label{eq: compact_d_rk2}
\end{align}
\end{subequations}
Here the approximations at the internal stages of one RK step, $U_h^{(l)}\in {\bf Z}_h^K$ and $g_h^{(l)}, {r_h^{(l)}} \in Z_h^K$ with $l=1, \cdots, s$, satisfy
\begin{subequations}
\label{eq: compact_rk}
\begin{align}
\left(U^{(l)}_h, \phi \right) &= \left(U^{n}_h, \phi \right) -\Dt \sum_{j=1}^{l-1} \tilde{a}_{lj} \left(F_h(U^{(j)}_h, \phi) -D_h( \eps(x) g^{(j)}_h, \phi)\right), \label{eq: compact_rk1}\\
\left(\eps(x) g^{(l)}_h, \psi\right) & =\left(\eps(x) g^{n}_h, \psi\right)- \Dt \sum_{j=1}^{l-1} \tilde{a}_{lj} b_{h, v}(\eps(x)g^{(j)}_h, \psi) + \Dt \sum_{j=1}^{l} a_{lj} \left(s^{(1)}_h(g^{(j)}_h, \psi) + s^{(2)}_{h, v}(U^{(j)}_h, r^{(j)}_h, \psi)\right), \label{eq: compact_rk2}\\
{(r_h^{(l)}, \varphi)} &{= H_h(U_h^{(l)}, \varphi),} \label{eq: compact_rk3}
\end{align}
\end{subequations}
for any $\phi, \psi, \varphi \in Z_h^K$. Similar to the first order IMEX scheme, in a stage-by-stage fashion for {$l=1, \cdots, s$}, one can first solve $U^{(l)}_h$ explicitly from the equation \eqref{eq: compact_rk1}, then plug $U^{(l)}_h$ into \eqref{eq: compact_rk3} to solve $r_h^{(l)}$, and finally solve $g^{(l)}_h$
from \eqref{eq: compact_rk2}.

The third order IMEX scheme we use in our simulations is the globally stiffly accurate ARS(4, 4, 3) scheme \cite{ascher1997implicit} with a double Butcher Tableau
\beq
\label{eq: ars443}
\begin{array}{c|c c c c c}
0 & 0&0&0& 0 & 0\\
1/2 &1/2&0&0& 0 &0\\
2/3 &11/18&1/18&0& 0 &0\\
1/2 &5/6&-5/6&1/2& 0 &0\\
1 &1/4&7/4&3/4& -7/4 &0\\
\hline
&1/4&7/4&3/4& -7/4 &0\\
 \end{array} \ \ \ \ \
\begin{array}{c|c c c c c}
0 & 0&0&0& 0 & 0\\
1/2 &0&1/2&0&0 &0\\
2/3 &0&1/6&1/2& 0 &0\\
1/2 &0&-1/2&1/2& 1/2 &0\\
1 &0&3/2&-3/2& 1/2 &1/2\\
\hline
 &0&3/2&-3/2& 1/2 &1/2\\
 \end{array}
\eeq

\subsection{Formal asymptotic analysis}
\label{sec:3.3}

It is straightforward to see that when $\eps$ goes to $0$, the proposed Schemes I and II
will give rise to a DG method satisfied by $U_h$ with the Lax-Friedrichs flux for the Euler system (see \eqref{cns-bgk} with $\eps=0$), and in fact the DG scheme in the limit is essentially the same as the one in \cite{cockburn1989tvb2}. This can be easily shown for both the semi-discrete and fully-discrete versions of the methods in their integral or nodal form.

In this subsection, we will focus on the formal asymptotic analysis to capture the compressible Navier-Stokes limit for small $\eps$. It is assumed that Knudsen number $\eps$ is constant, and the $v$-direction is continuous without being discretized.
Though both Schemes I and II perform well numerically for a wide range of Knudsen number (see Section \ref{sec:numerical}), only Scheme II can be formally shown to have the correct compressible Navier-Stokes limit. More specifically, in Proposition~\ref{prop: asymptotics} below, we will establish a formal asymptotic analysis, showing that for $0<\eps\ll1$, the fully-discrete DG-IMEX Scheme II in its nodal form will become a local DG method in its nodal form with some explicit RK time discretization, up to $\mathcal{O}(\eps^2)$,  for the compressible Navier-Stokes system.
The limiting scheme is in a similar spirit to the highly cited work proposed by Bassi and Rebay in 1997 \cite{bassi1997high} which was later generalized and analyzed in \cite{cockburn1998local}.

To establish such connection, we first formulate a local DG method with an explicit RK time discretization for the compressible Navier-Stokes system \eqref{cns-bgk2}, which in one dimension is given as follows,
\begin{equation}
\begin{split}
\del_t \left(
\begin{array}{c} \rho\\ \rho u \\ E
\end{array}
 \right)+
 \del_x \left(
 \begin{array}{c}
 \rho u \\ \rho u^2 +p I  \\ (E+p)u
 \end{array}
 \right) = \ep  \del_x \left(
 \begin{array}{c} 0 \\
 0   \\
 \frac32 \rho T \partial_x T
\end{array}
 \right).
\end{split}\label{cns-bgk-1d}
\end{equation}
The semi-discrete local DG method is to find $U_h(\cdot, t) \in {\bf Z}_h^K$ and  $r_h(\cdot, t) \in Z_h^K$, such that $\forall \phi, \varphi \in Z_h^K$
\begin{subequations}
\label{eq:scheme:cns}
\begin{align}
\label{eq: compact_cns}
\left({\partial_t}U_h, \phi \right) + F_h(U_h, \phi) &= \eps F^{(vis)}_h(U_h, r_h, \phi),\\
\label{eq: q}
(r_h, \varphi)&=
-\sum_i\left(\int_{I_i} T_h \frac{d\varphi}{dx} dx + \widehat{T}_{h,i+\frac12} [\varphi]_{\iR}\right).
\end{align}
\end{subequations}
Here $F^{(vis)}_h(U_h, r_h, \phi) = (0, 0,   f^{(vis)}_{E,h})^t$ with
\begin{align}
f^{(vis)}_{E,h}
& = -\frac32\sum_i \left(\int_{I_i}  \rho_h T_h r_h \frac{d\phi}{dx} dx +  \widehat{(\rho_h T_h r_h)}_\iR  [\phi]_{\iR} \right), \label{eq: vis_flux}
\end{align}
where $T_h$ is obtained from $U_h$ based on \eqref{eq: T}.
Note that the right hand side of \eqref{eq: q} is just $H_h(U_h, \varphi)$ defined in \eqref{eq: S2}, and it is explicitly written here to emphasize that  \eqref{eq: q} is to approximate the auxiliary variable $r := \partial_x T$.
The numerical fluxes in \eqref{eq: q} and \eqref{eq: vis_flux} can be taken to be either alternating or central fluxes.
Via the method of line approach, the local DG scheme \eqref{eq:scheme:cns}
can be evolved in time by an explicit RK method characterized by a Butcher table $\tilde{A}$, $\tilde{b}$ and $\tilde{c}$ in \eqref{eq: B_table} as follows: find $U^{n+1}_h(\cdot, t), U^{(l)}_h \in {\bf Z}_h^K$ and  $r^{(l)}_h(\cdot, t) \in Z_h^K$ with $l=1, \cdots, s$, such that $\forall \phi, \varphi \in Z_h^K$,
\beq
\label{eq: compact_cns_rk0}
\left(U^{n+1}_h, \phi \right) = \left(U^{n}_h, \phi \right) -\Dt \sum_{l=1}^{s} \tilde{b}_l \left(F_h(U^{(l)}_h, \phi) -\eps F^{(vis)}_h(U^{(l)}_h, r^{(l)}_h, \phi)\right),
\eeq
with
\beq
\label{eq: compact_cns_rk}
\left(U^{(l)}_h, \phi \right) = \left(U^{n}_h, \phi \right) -\Dt \sum_{j=1}^{l-1} \tilde{a}_{lj} \left(F_h(U^{(j)}_h, \phi) -\eps F^{(vis)}_h (U^{(j)}_h,  r^{(j)}_h, \phi)\right),
\eeq
and
\beq
\label{eq: q_rk}
(r_h^{(l)}, \varphi)=
-\sum_i\left(\int_{I_i} T_h^{(l)} \frac{d\varphi}{dx} dx + \widehat{T}^{(l)}_{h,i+\frac12} [\varphi]_{\iR}\right).
\eeq
The method above can be further given in its nodal form similarly as in Section \ref{sec:3.1}, and this is omitted for brevity.
Next we will state and show a formal asymptotic analysis for the proposed Scheme II in its nodal form. {We consider the third order IMEX scheme ARS(4, 4, 3) as an example to illustrate the analysis, while the result can be extended to general IMEX methods of type A, CK or ARS (for definitions of these types see \cite{boscarino2008error}).}

\begin{prop}
\label{prop: asymptotics}
For the BGK equation based on the micro-macro formulation \eqref{eq:mmdc}, we consider the fully discrete DG-IMEX scheme \eqref{eq: compact_d_rk}-\eqref{eq: compact_rk}, with operators specified in \eqref{eq: S2} for Scheme II in its nodal form, and {the third order globally stiffly accurate IMEX time discretization ARS(4,4,3) characterized by the double Butcher table \eqref{eq: B_table} and given in \eqref{eq: ars443}}.
 For $0<\eps\ll 1$, the scheme is asymptotically equivalent, up to $\mathcal{O}(\eps^2)$,
 to the local DG method \eqref{eq: compact_cns_rk0}-\eqref{eq: q_rk} in its nodal form for compressible Navier-Stokes equations \eqref{cns-bgk-1d}, coupled with the explicit RK time discretization  characterized by the Butcher table $\tilde{A}$, $\tilde{b}$ and $\tilde{c}$ in \eqref{eq: B_table}.
\end{prop}
\begin{proof}
It is sufficient to prove the equivalence of the discretizations, namely, the nodal forms of the $D_h$ term in \eqref{eq: compact_d_rk}-\eqref{eq: compact_rk} and the $\eps F^{(vis)}_h$ term in \eqref{eq: compact_cns_rk0}-\eqref{eq: compact_cns_rk}, for the viscous term in \eqref{cns-bgk-1d}
for sufficiently small $\eps$.
Consider the fully discrete scheme \eqref{eq: compact_d_rk}-\eqref{eq: compact_rk} in its nodal form (e.g. see equations  \eqref{eq:SDG1:a}, \eqref{eq:SDG2:c}, \eqref{eq:SDG2:d}),
when $0<\eps\ll1$ and for {the RK stage, $l = 2, \cdots, s$},
we have
\beq
\label{eq: ap_g}
(g_i^k)^{(l)} = -(A_i^k)^{(l)} (r_i^k)^{(l)} (M_i^k)^{(l)}/\sqrt{(T_i^k)^{(l)}} + \mathcal{O}(\eps), \quad \forall i=1, \cdots, N_x, \quad k = 0, \cdots, K
\eeq
with $(r_i^k)^{(l)}$ obtained from the equation \eqref{eq:SDG2:d}.
Applying $\langle vm \cdot \rangle$ to the above equation gives
\begin{align}
\langle vm(g_i^k)^{(l)} \rangle & = - \langle vm (A_i^k)^{(l)} (M_i^k)^{(l)} \rangle {(r_i^k)^{(l)}}/{\sqrt{(T_i^k)^{(l)}}} + \mathcal{O}(\eps), \nonumber\\
& = (0, 0,  \frac32 (\rho_i^k))^{(l)} (T_i^k)^{(l)} (r_i^k)^{(l)})^t +\mathcal{O}(\eps), \quad \forall i=1, \cdots, N_x, \quad k = 0, \cdots, K.
\label{eq: vmg}
\end{align}
Here the second equality is obtained along the same line of deriving the viscous term in the equation \eqref{cns-bgk-1d}.
{Since the IMEX method is globally stiffly accurate, $(g_i^k)^{n+1}= (g_i^k)^{(s)}$, hence \eqref{eq: ap_g} holds also at $t^{n+1}$. Such property is needed to justify \eqref{eq: ap_g} for $l=1$ when $a_{11} = 0$ for the next time step evolution, see the Butcher table \eqref{eq: ars443}.}
Now one can plug \eqref{eq: vmg} into the last term in \eqref{eq: compact_d_rk1} and in \eqref{eq: compact_rk1} in the nodal form (see e.g. \eqref{eq:SDG1:a}), and up to $\mathcal{O}(\eps^2)$, this will result in  $\eps F^{(vis)}_{h} (\cdot, \cdot, \cdot)$ in \eqref{eq: compact_cns_rk0}-\eqref{eq: compact_cns_rk} in its nodal form. This, in addition to \eqref{eq:SDG2:d}, determines the discretization of the viscous term up to $\mathcal{O}(\eps^2)$. The explicit part of the IMEX time discretization will be naturally carried over.
\end{proof}
\begin{rem}
The local DG method in \eqref{eq: compact_cns_rk0}-\eqref{eq: q_rk} is in a similar spirit to, yet different from the one proposed in \cite{bassi1997high}. In \cite{bassi1997high}, auxiliary variables $S := {\partial_x U}$ are introduced and approximated together with $U$  by local DG discretizations with central fluxes, while here only one auxiliary variable $r := {\partial_x T}$ is introduced {for one-dimensional problems}.
\end{rem}

\section{Numerical Examples}
\label{sec:numerical}
\setcounter{equation}{0}
\setcounter{figure}{0}
\setcounter{table}{0}

In this section, we consider the proposed nodal DG methods with the IMEX time discretizations
for solving the micro-macro decomposed equations \eqref{eq:mmdc} with constant $\eps$, or \eqref{eq:mmd} with variable $\eps(x)$. The alternating left-right numerical flux is taken.
We use NDG($K$) to denote the method formulated based on $K$-point Gaussian quadrature, while in time the third order IMEX scheme in Section \ref{sec:3.2} is applied. The time step is chosen to be $\Delta t= C_{CFL} \Delta x/\max(\Lambda,V_c)$ for $K=1, 2, 3$, and  $\Delta t= C_{CFL} x^{4/3}/\max(\Lambda,V_c)$ for $K=4$,  where $\Lambda=\| |u|+\sqrt{\gamma T} \|_\infty$ is the maximal absolute eigenvalue of  $\df F(U)/\df U$ over the spatial domain, and the CFL number $C_{CFL}$ is taken to be $0.2, 0.1, 0.05, 0.01$ for $K=1, 2, 3, 4$ respectively.  The velocity domain $\Omega_v=[-V_c, V_c]$ is set to be large enough.
For the one dimensional problem considered in this paper,  $\gamma=\frac{d+2}{d}=3$, and the numerical examples are slightly different from the classical ones with $\gamma=1.4$ \cite{shu1989efficient}, see particularly the setup of the initial conditions and the final time $t$.
 For discontinuous solutions, the TVB limiter \cite{cockburn1989tvb2} is used and it is only applied on $U_h$, with the TVB parameter $M_{tvb}=20$ unless otherwise specified.

As discussed in Section \ref{sec:3.1}, Scheme II is computationally more efficient than Scheme I. Numerically,
we observe that around half CPU time is saved when Scheme II is used. Both schemes produce similar results. In the following, we will only present the results from Scheme II.

\begin{exa}
\label{ex41}
{\bf(Accuracy tests.)}
We first consider an example with smooth exact solutions. The initial conditions are
\begin{equation}
\label{smooth}
\rho(x,0)=1+0.2 \sin(x), \quad p=1, \quad u=1,
\end{equation}
with
\begin{equation}
\label{smooth2}
g(x,v,0)=-A\frac{\partial_x T}{\sqrt{T}}M,
\end{equation}
on the domain $[-\pi, \pi] \times [-12, 12]$ with periodic boundary conditions in the $x$ direction.
$\Omega_v=[-12, 12]$ is discretized with $N_v=100$ uniform points. Since the exact solution is not available,
the $L^1$ errors are computed as the difference of the numerical solutions on two consecutive meshes,
\begin{subequations}
\label{l1error}
\begin{align}
L^1 \text{ error of $\rho$ } (h) &=\frac{1}{2\pi}\sum_i \int_{I_i}|\rho_{h}(\cdot,T)-\rho_{h/2}(\cdot,T)|dx \;, \label{l1errorr} \\
L^1 \text{ error of $g$ } (h) &=\frac{1}{2\pi N_v}\sum_{i,j} \int_{I_i}|g_{h}(\cdot,v_j,T)-g_{h/2}(\cdot,v_j,T)|dx \;.
\label{l1errorg}
\end{align}
\end{subequations}
Here $w_h$ is the numerical solution when the mesh size is $h$, with $w$ to be $\rho$ or $g$, and $I_i$ is an element from the finer mesh with the mesh size $h/2$. The corresponding convergence order is computed by
\begin{equation}
\label{l1order}
\text{order}=\frac{\log \left ( L^1 \text{ error of $w$ } (h) / L^1 \text{ error of $w$ } (h/2) \right)}{\log 2}.
\end{equation}
We show the $L^1$ errors and orders of NDG($K$), with $\eps=1, 10^{-2}, 10^{-6}$ and $K=1, 2, 3, 4$, at time $t=0.001$ in Table \ref{tab1} with initial conditions \eqref{smooth} and \eqref{smooth2}. From these results, we can see that $K$-th order of accuracy for $\rho$ has been obtained for NDG($K$). However, for $g$, only $(K-1)$-th order can be observed except for the case of $K=1$. For this smooth problem, we also show the conserved properties of the methods by presenting $\eps\langle mg \rangle$ from NDG3 with $m=(1, v, |v|^2/2)^t$ for $\eps=1$ and $\eps=10^{-6}$ in Fig. \ref{fig0}. Analytically $\langle mg \rangle = 0$ however numerically they are often not. We can see that the conservation errors can be greatly improved by doubling the domain $\Omega_v$, even with $N_v$ unchanged. Similar results hold for the other examples in this section.

\begin{table}
\centering
\caption{$L^1$ errors and orders for $\rho$ and $g$ of Example \ref{ex41} with initial conditions \eqref{smooth} and \eqref{smooth2}. $t=0.001$. }
\vspace{0.2cm}
  \begin{tabular}{|c|c||p{1.6cm}|c|p{1.6cm}|c||p{1.6cm}|c|p{1.6cm}|c|}
    \hline
 &   N    &  $L^1$ error of $\rho$ & order   & $L^1$ error of $g$  & order
          &  $L^1$ error of $\rho$ & order   & $L^1$ error of $g$  & order \\\hline
 &        & \multicolumn{4}{|c||}{NDG1}   & \multicolumn{4}{|c|}{NDG2}   \\\hline
\multirow{5}{*}{$\eps=1$}
 &     10 &     1.97E-02 &       --&     6.63E-04 &       --
          &     1.52E-03 &       --&     5.12E-05 &       --  \\  \cline{2-10}
 &     20 &     1.00E-02 &     0.98&     3.30E-04 &     1.01
          &     3.76E-04 &     2.02&     1.32E-05 &     1.95  \\  \cline{2-10}
 &     40 &     5.00E-03 &     1.00&     1.65E-04 &     1.00
          &     9.42E-05 &     2.00&     3.37E-06 &     1.97  \\  \cline{2-10}
 &     80 &     2.50E-03 &     1.00&     8.25E-05 &     1.00
          &     2.38E-05 &     1.98&     8.78E-07 &     1.94  \\  \cline{2-10}
 &    160 &     1.25E-03 &     1.00&     4.13E-05 &     1.00
          &     6.08E-06 &     1.97&     2.39E-07 &     1.88  \\  \hline
\multirow{5}{*}{$\eps=10^{-2}$}
 &     10 &     1.97E-02 &       --&     6.60E-04 &       --
          &     1.52E-03 &       --&     7.94E-05 &       --  \\  \cline{2-10}
 &     20 &     1.00E-02 &     0.98&     3.30E-04 &     1.00
          &     3.76E-04 &     2.02&     3.27E-05 &     1.28  \\  \cline{2-10}
 &     40 &     5.00E-03 &     1.00&     1.65E-04 &     1.00
          &     9.42E-05 &     2.00&     1.51E-05 &     1.11  \\  \cline{2-10}
 &     80 &     2.50E-03 &     1.00&     8.25E-05 &     1.00
          &     2.38E-05 &     1.98&     7.29E-06 &     1.05  \\  \cline{2-10}
 &    160 &     1.25E-03 &     1.00&     4.13E-05 &     1.00
          &     6.08E-06 &     1.97&     3.48E-06 &     1.07  \\  \hline
\multirow{5}{*}{$\eps=10^{-6}$}
 &     10 &     1.97E-02 &       --&     7.04E-04 &       --
          &     1.52E-03 &       --&     6.20E-04 &       --  \\  \cline{2-10}
 &     20 &     1.00E-02 &     0.98&     3.54E-04 &     0.99
          &     3.76E-04 &     2.02&     3.20E-04 &     0.96  \\  \cline{2-10}
 &     40 &     5.00E-03 &     1.00&     1.78E-04 &     0.99
          &     9.42E-05 &     2.00&     1.63E-04 &     0.97  \\  \cline{2-10}
 &     80 &     2.50E-03 &     1.00&     8.87E-05 &     1.00
          &     2.38E-05 &     1.98&     8.42E-05 &     0.95  \\  \cline{2-10}
 &    160 &     1.25E-03 &     1.00&     4.44E-05 &     1.00
          &     6.08E-06 &     1.97&     4.43E-05 &     0.93  \\  \hline

 &        & \multicolumn{4}{|c||}{NDG3}   & \multicolumn{4}{|c|}{NDG4}   \\\hline
 \multirow{5}{*}{$\eps=1$}
  &     10 &     7.76E-05 &       --&     3.73E-06 &       --
           &     3.79E-06 &       --&     2.62E-07 &       --  \\  \cline{2-10}
  &     20 &     1.00E-05 &     2.96&     4.75E-07 &     2.97
           &     2.35E-07 &     4.01&     1.77E-08 &     3.89  \\  \cline{2-10}
  &     40 &     1.31E-06 &     2.93&     6.15E-08 &     2.95
           &     1.46E-08 &     4.00&     1.16E-09 &     3.93  \\  \cline{2-10}
  &     80 &     1.79E-07 &     2.87&     8.31E-09 &     2.89
           &     9.16E-10 &     4.00&     1.05E-10 &     3.47  \\  \cline{2-10}
  &    160 &     2.56E-08 &     2.81&     1.23E-09 &     2.76
           &     5.75E-11 &     3.99&     2.49E-11 &     2.08  \\  \hline
 \multirow{5}{*}{$\eps=10^{-2}$}
  &     10 &     7.76E-05 &       --&     5.80E-06 &       --
           &     3.79E-06 &       --&     3.87E-07 &       --  \\  \cline{2-10}
  &     20 &     1.00E-05 &     2.96&     1.08E-06 &     2.43
           &     2.35E-07 &     4.01&     3.79E-08 &     3.35  \\  \cline{2-10}
  &     40 &     1.31E-06 &     2.93&     2.29E-07 &     2.23
           &     1.46E-08 &     4.00&     4.11E-09 &     3.21  \\  \cline{2-10}
  &     80 &     1.79E-07 &     2.87&     5.05E-08 &     2.18
           &     9.16E-10 &     4.00&     4.58E-10 &     3.16  \\  \cline{2-10}
  &    160 &     2.56E-08 &     2.81&     1.09E-08 &     2.22
           &     5.75E-11 &     3.99&     4.86E-11 &     3.24  \\  \hline
 \multirow{5}{*}{$\eps=10^{-6}$}
  &     10 &     7.76E-05 &       --&     4.05E-05 &       --
           &     3.79E-06 &       --&     2.81E-06 &       --  \\  \cline{2-10}
  &     20 &     1.00E-05 &     2.96&     9.50E-06 &     2.09
           &     2.35E-07 &     4.01&     3.63E-07 &     2.95  \\  \cline{2-10}
  &     40 &     1.31E-06 &     2.93&     2.23E-06 &     2.09
           &     1.46E-08 &     4.00&     4.76E-08 &     2.93  \\  \cline{2-10}
  &     80 &     1.79E-07 &     2.87&     5.14E-07 &     2.12
           &     9.16E-10 &     4.00&     6.34E-09 &     2.91  \\  \cline{2-10}
  &    160 &     2.56E-08 &     2.81&     1.18E-07 &     2.12
           &     5.75E-11 &     3.99&     8.48E-10 &     2.90  \\  \hline
\end{tabular}
\label{tab1}
\end{table}

\begin{figure}[ht]
\centering
\includegraphics[totalheight=2.5in]{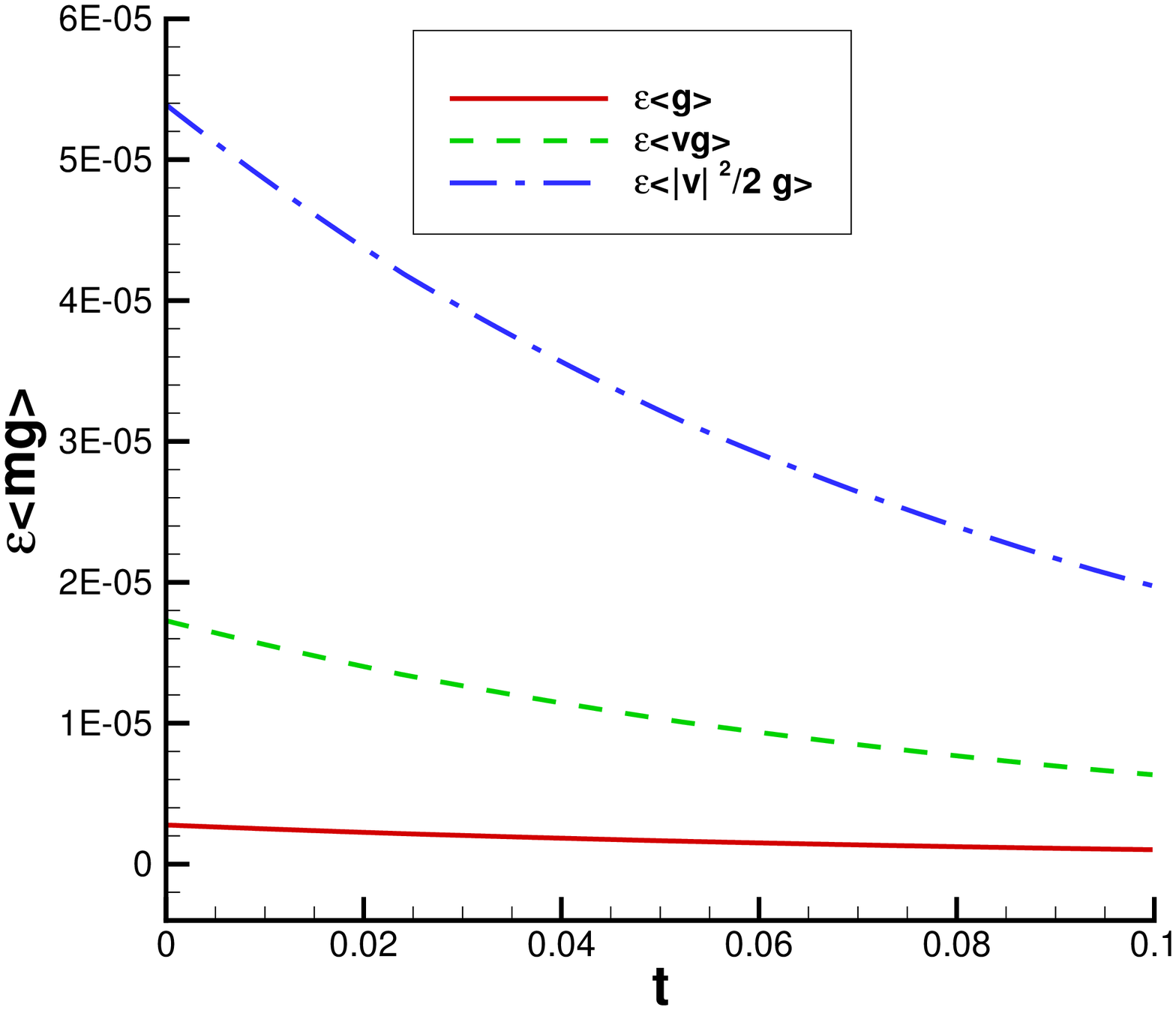},
\includegraphics[totalheight=2.5in]{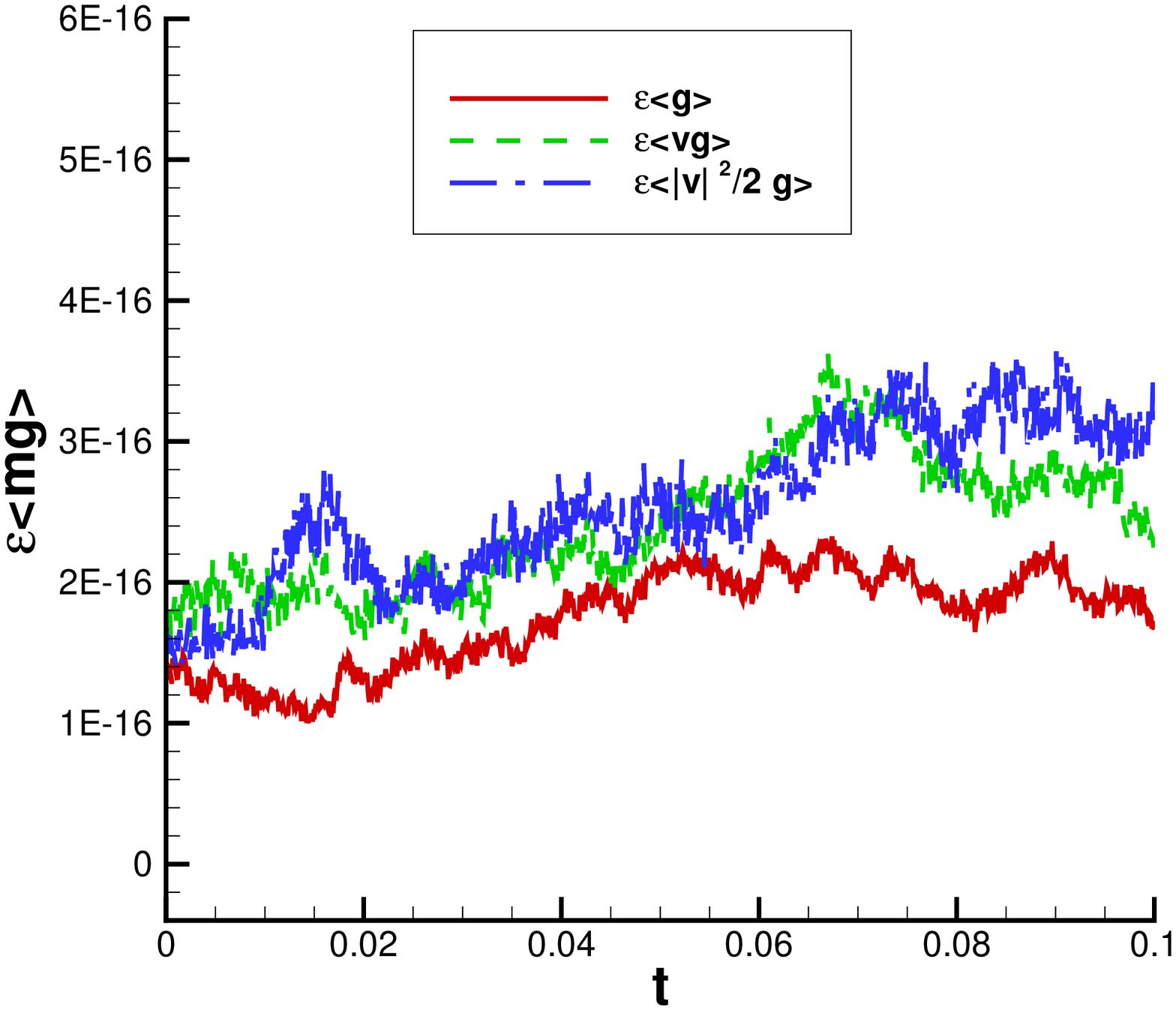}\\
\includegraphics[totalheight=2.5in]{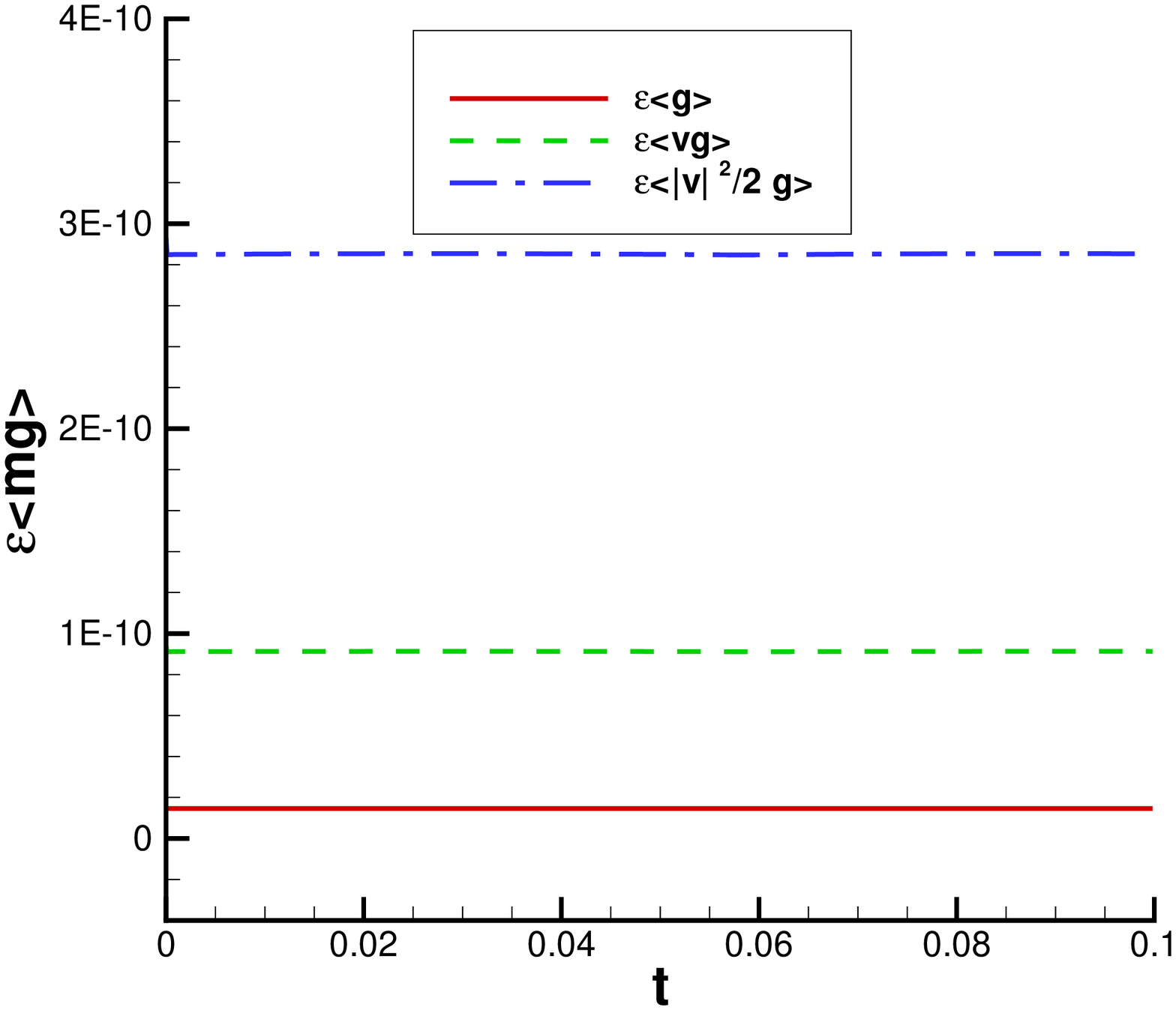},
\includegraphics[totalheight=2.5in]{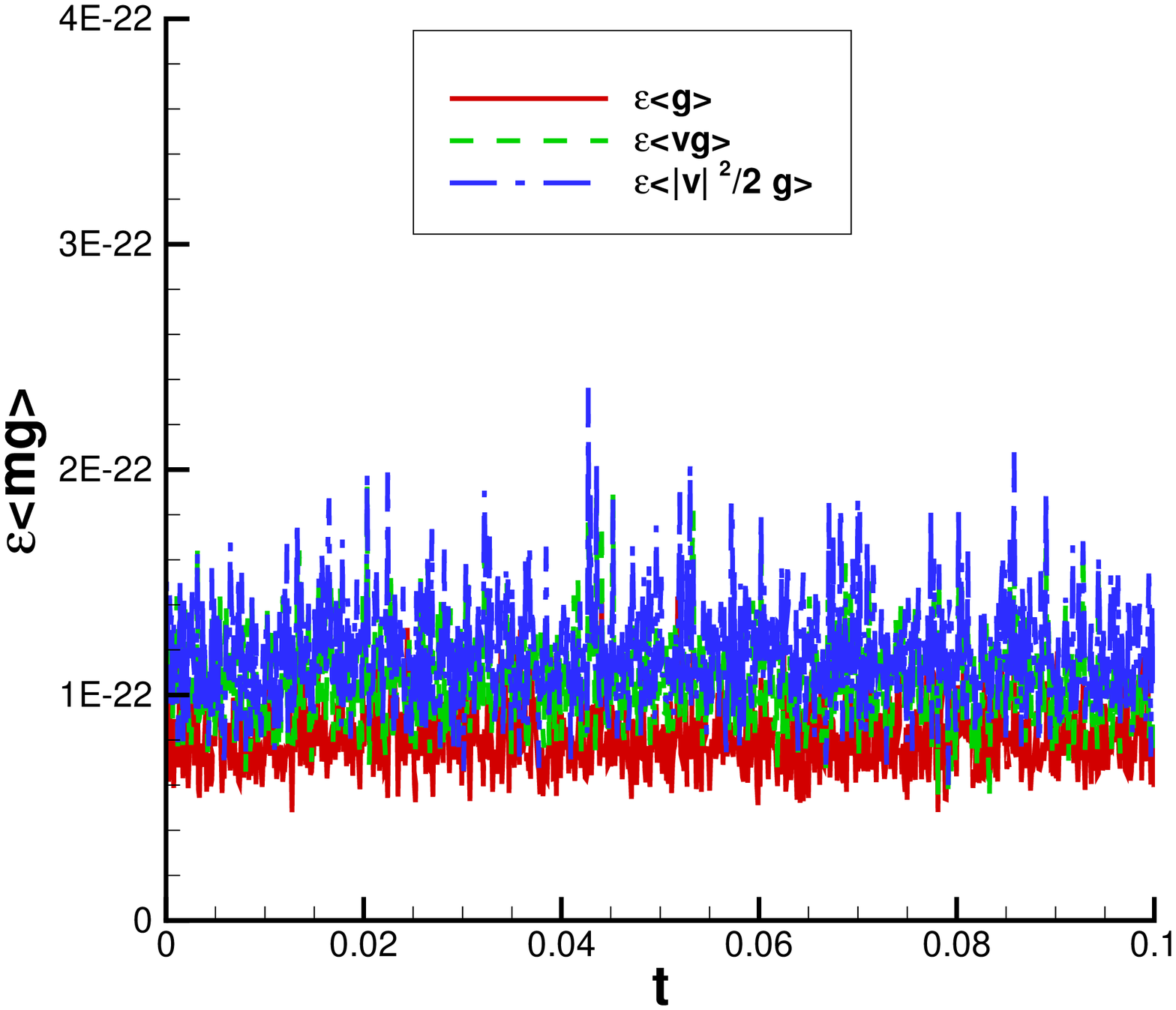}
\caption{$\eps \langle mg \rangle$ for Example \ref{ex41} with NDG3 on the spatial domain $[-\pi,\pi]$. Measured as the maximal value over $x$. $N_x=100$ and $N_v=100$.
Top: $\eps=1$; Bottom: $\eps=10^{-6}$. Left: $V_c=6$; Right: $V_c=12$. }
\label{fig0}
\end{figure}

\end{exa}

\begin{exa}
Next we consider the Sod shock tube problem with the initial conditions for the density,
the mean velocity and the pressure to be
\begin{eqnarray}
(\rho, u, p)=
\begin{cases}
(1, 0, 1), & \quad 0\le x \le 0.5, \\
(0.125, 0, 0.1), & \quad 0.5< x \le 1.
\end{cases}
\label{sod}
\end{eqnarray}
The initial distribution function $f$ is the Maxwellian \eqref{maxwellian} with $g(x,v,0)=0$.
The spatial domain $[-0.2, 1.2]$ is discretized with $N_x=50$ grid points and the velocity domain  $[-4.5, 4.5]$ is discretized with $N_v=100$ uniform points. The boundary conditions are taken to be the initial constant values at both ends in the $x$ direction.
We compute the solutions with NDG3 up to time $t=0.14$. In Fig. \ref{fig1}, we show the distribution function $f$ at $x=0.5$. It can be seen that when $\eps$ approaches $0$, the computed $f$ becomes close to a Maxwellian.
The density, the mean velocity, the temperature and the rescaled heat flux $Q_{\eps}$ are presented in Fig. \ref{fig2}, which
are similar to the results in \cite{bennoune2008uniformly}. Here the rescaled heat flux is defined as $Q_{\eps}=\frac{Q}{\eps}=\left \langle \frac{|v-u|^2}{2}(v-u)g \right \rangle$, and it
approximates $-q$ in \eqref{eq: sigma_q}. The conserved properties are demonstrated by plotting $\eps\langle mg \rangle$ in Fig. \ref{fig3}, with $m=(1, v, |v|^2/2)^t$ and for $\eps=1$ and $\eps=10^{-6}$.
For this problem with discontinuous solutions, the conservation errors of $\eps\langle mg \rangle$ can
also be greatly improved by doubling the domain $\Omega_v$.

We also take the Sod problem as a representative example to show that the proposed NDG3-IMEX method for the BGK equation is asymptotically equivalent,  up to $\mathcal{O}(\eps^2)$, to the 3rd order local DG method as described in Section \ref{sec:3.3} with $K=3$ for directly solving \eqref{cns-bgk-1d} when $0<\eps \ll 1$. In Fig. \ref{fig41}, the relative differences between the density, velocity and temperature versus $\eps$ are plotted in the logarithmic scale. When $\eps \ll 1$, the expected second order difference with respect to $\eps$ is observed. The relative difference is computed as
\beq
\text{relative difference of } w = \frac{\sum_{i,k}\omega_k|w_1(x_i^k)-w_2(x_i^k)|}{\sum_{i,k}\omega_k|w_1(x_i^k)|},
\eeq
where $w_1$ is the numerical solution of NDG3 solving the BGK equation, and $w_2$ is the numerical solution of the 3rd order local DG method for the compressible Navier-Stokes system,
with $w$ to be $\rho$, $u$ and $T$, $x_i^k$ and $\omega_k$ are the Gaussian quadrature point and corresponding weight in cell $I_i$. Here we take the domain $[-0.2, 1.2]\times[-9, 9]$ and $t=0.01$. The mesh is $N_x=50$ and $N_v=100$. For both schemes, the TVB limiter is not used. Similar results hold for other examples with constant $\eps$.
\begin{figure}[ht]
\centering
\includegraphics[totalheight=2.5in]{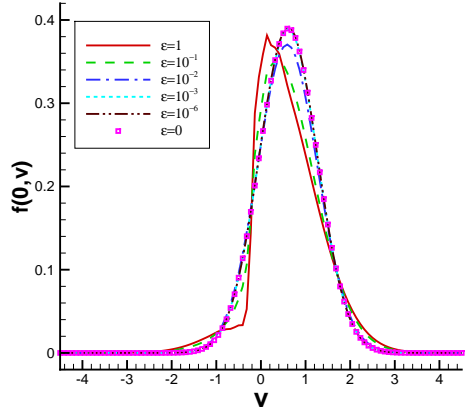}
\caption{Distribution function $f$ for the Sod problem (\ref{sod}) at $x=0.5$ as a function of $v\in[-4.5, 4.5]$.
$N_x=50$ and $N_v=100$ for NDG3 at $t=0.14$. As $\eps$ goes to $0$, $f$ becomes close to a Maxwellian.
With TVB limiter and $M_{tvb}=20$.}
\label{fig1}
\end{figure}

\begin{figure}[ht]
\centering
\includegraphics[totalheight=2.5in]{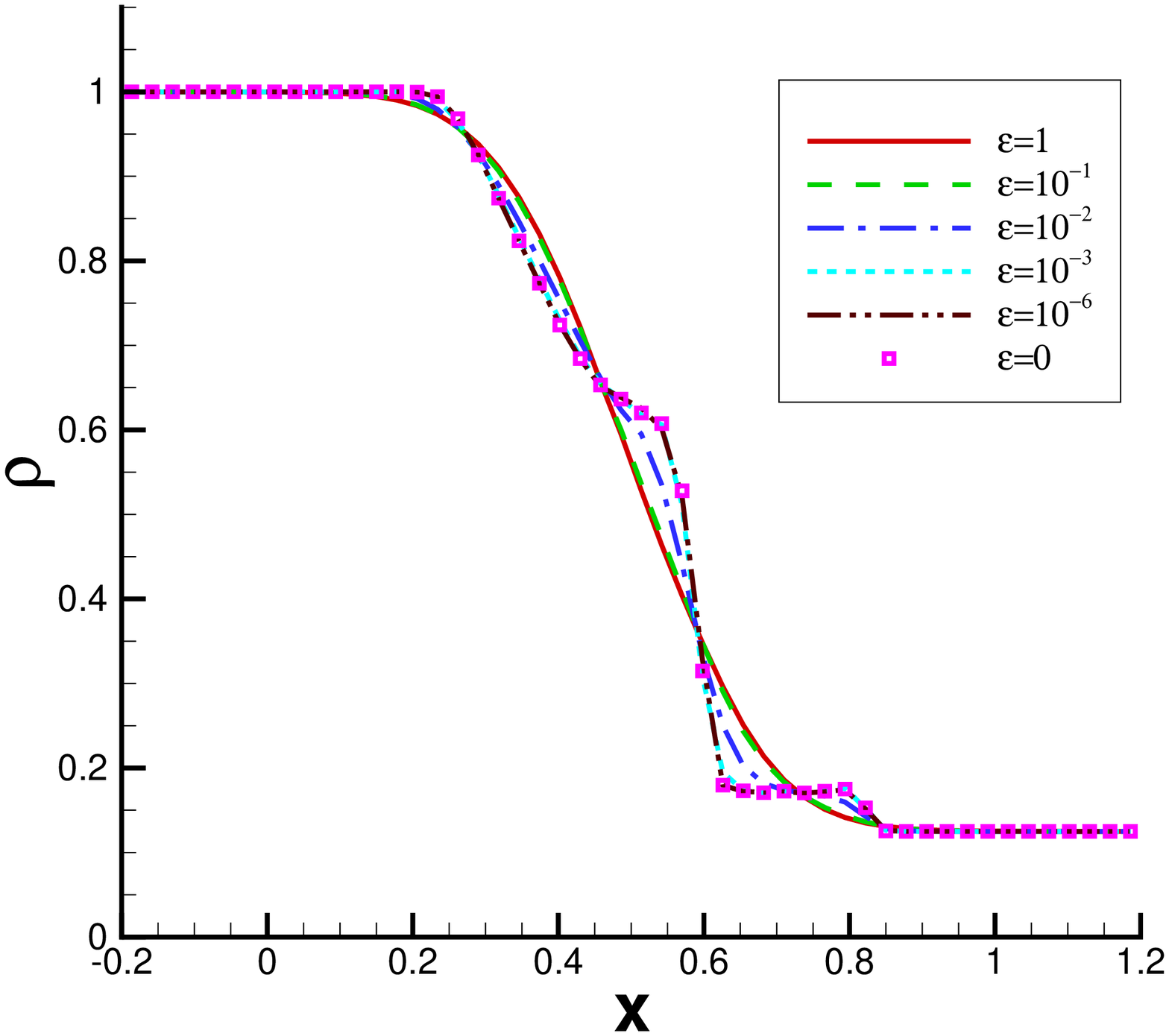},
\includegraphics[totalheight=2.5in]{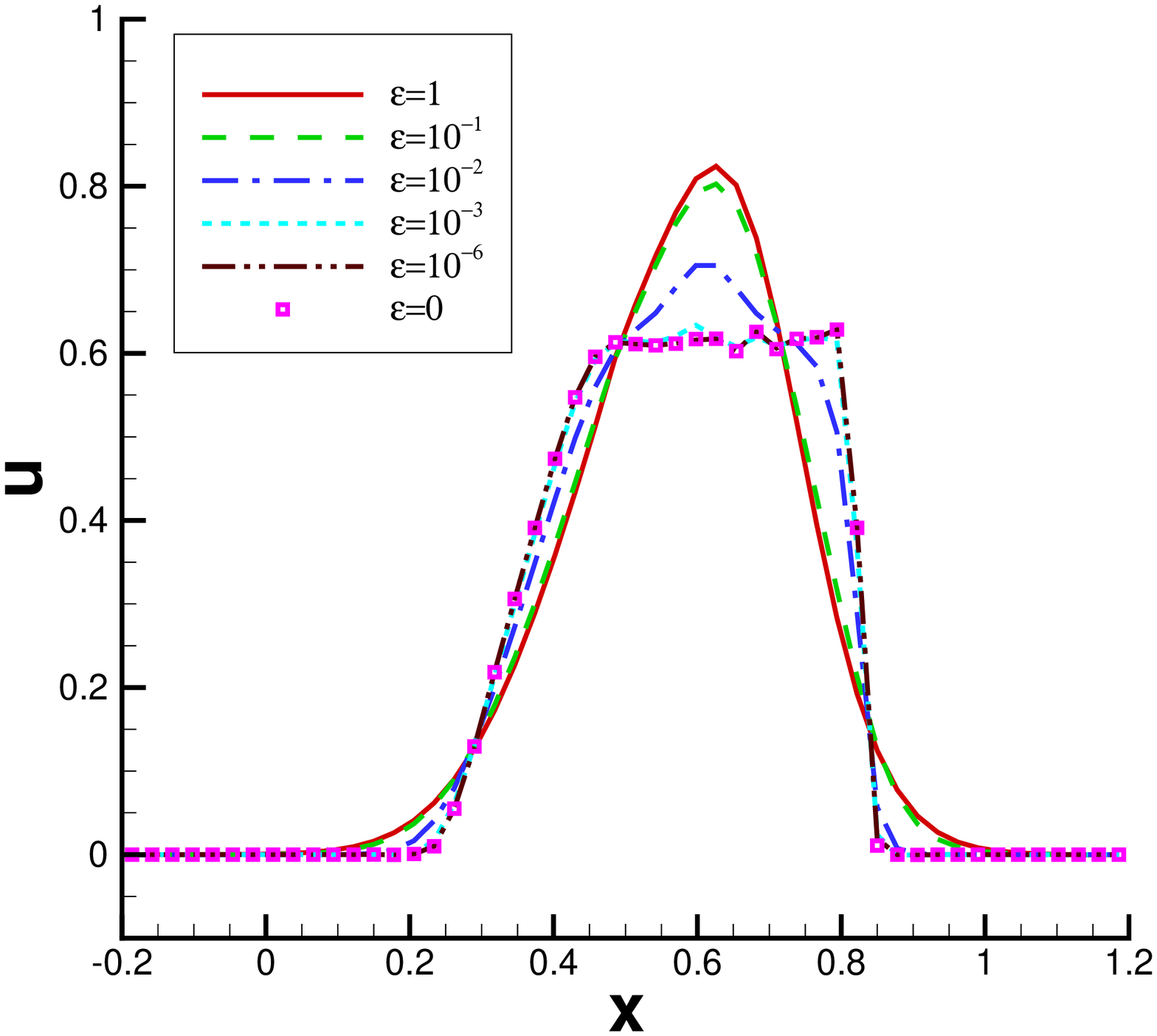}\\
\includegraphics[totalheight=2.5in]{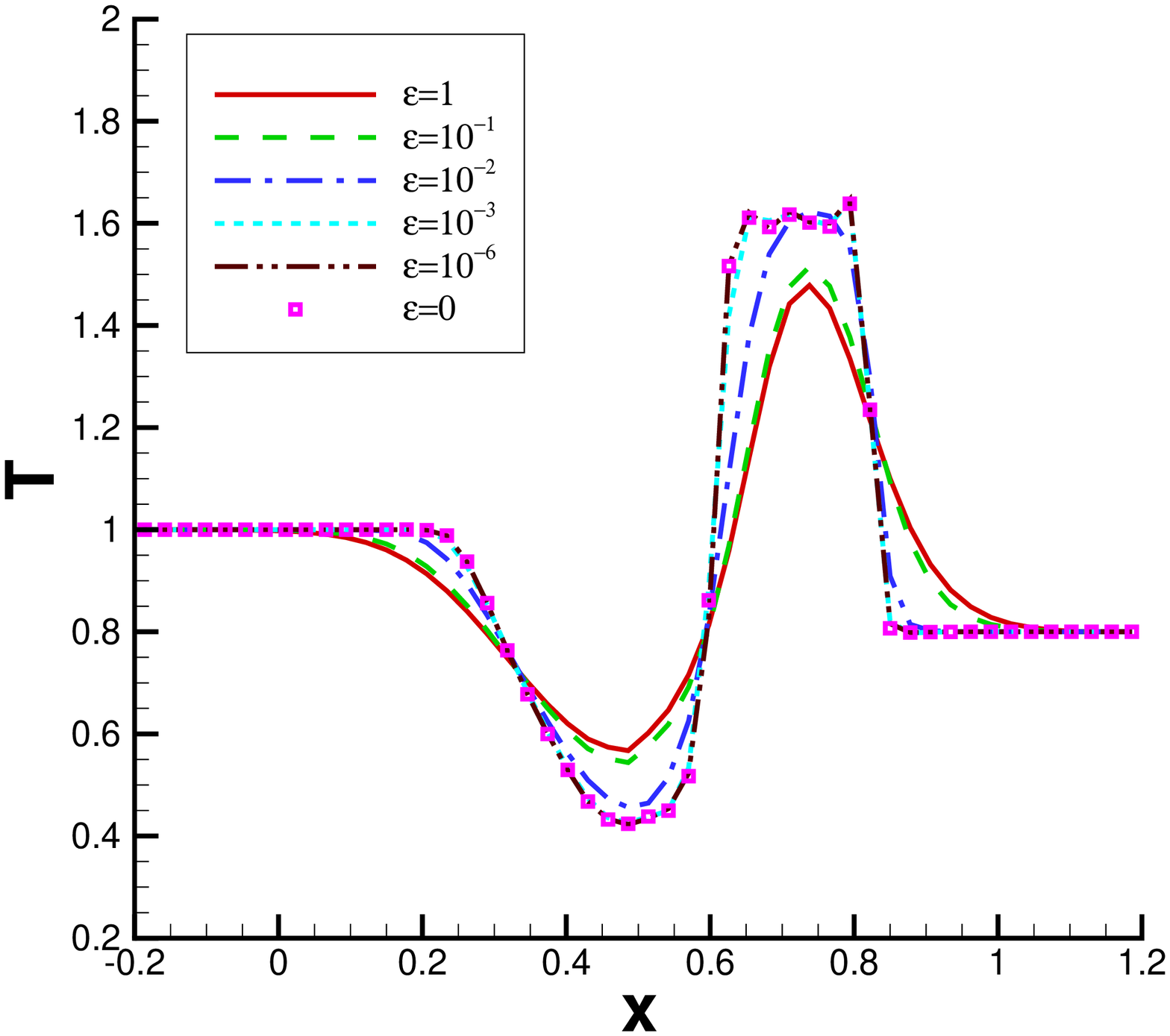},
\includegraphics[totalheight=2.5in]{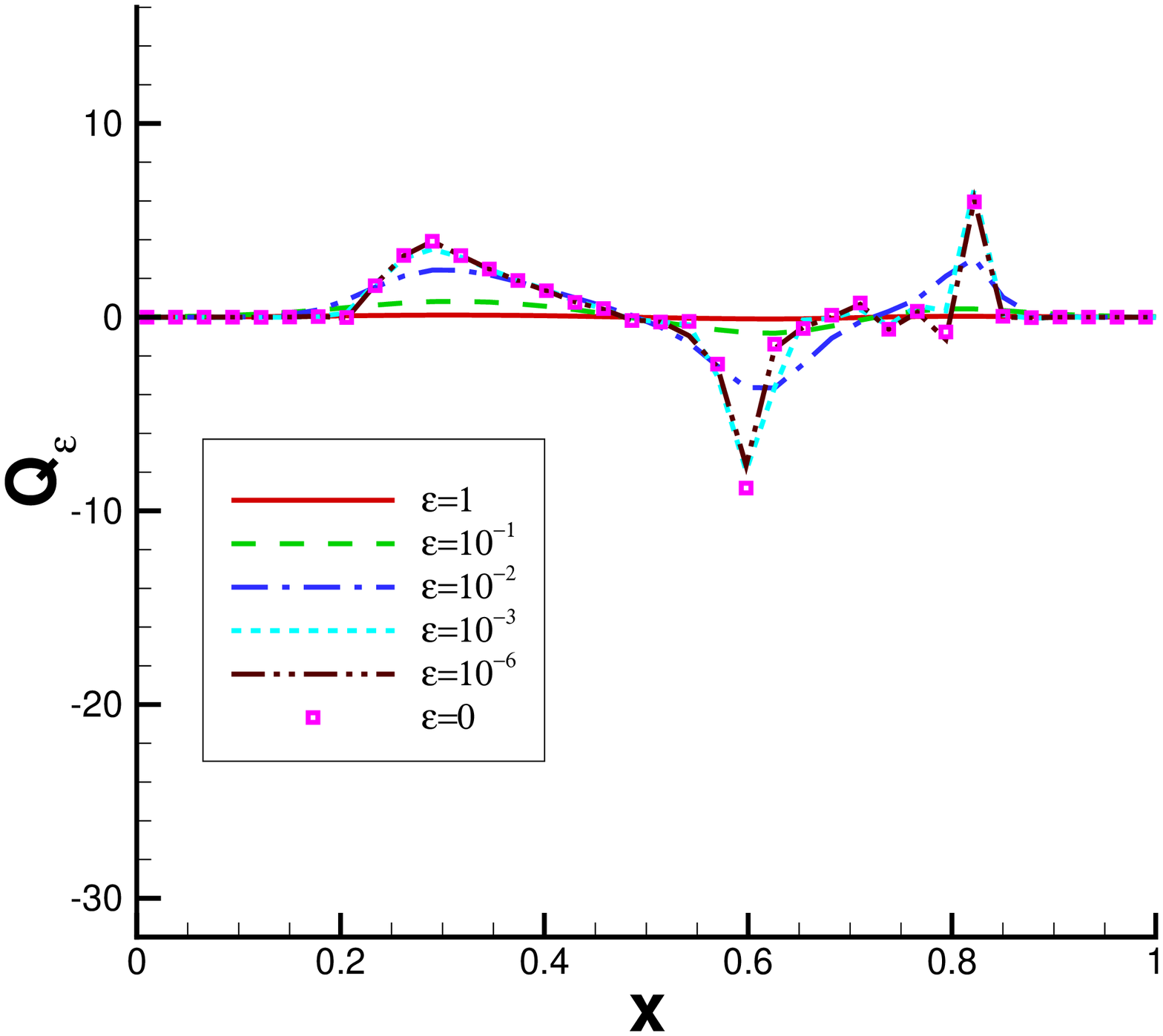}
\caption{Numerical solutions for the Sod problem (\ref{sod}) at $t=0.14$ with NDG3.
$N_x=50$ and $N_v=100$ on the domain $[-0.2, 1.2]\times [-4.5, 4.5]$. Top: density, mean velocity;
Bottom: temperature, rescaled heat flux. With TVB limiter and $M_{tvb}=20$.}
\label{fig2}
\end{figure}

\begin{figure}[ht]
\centering
\includegraphics[totalheight=2.5in]{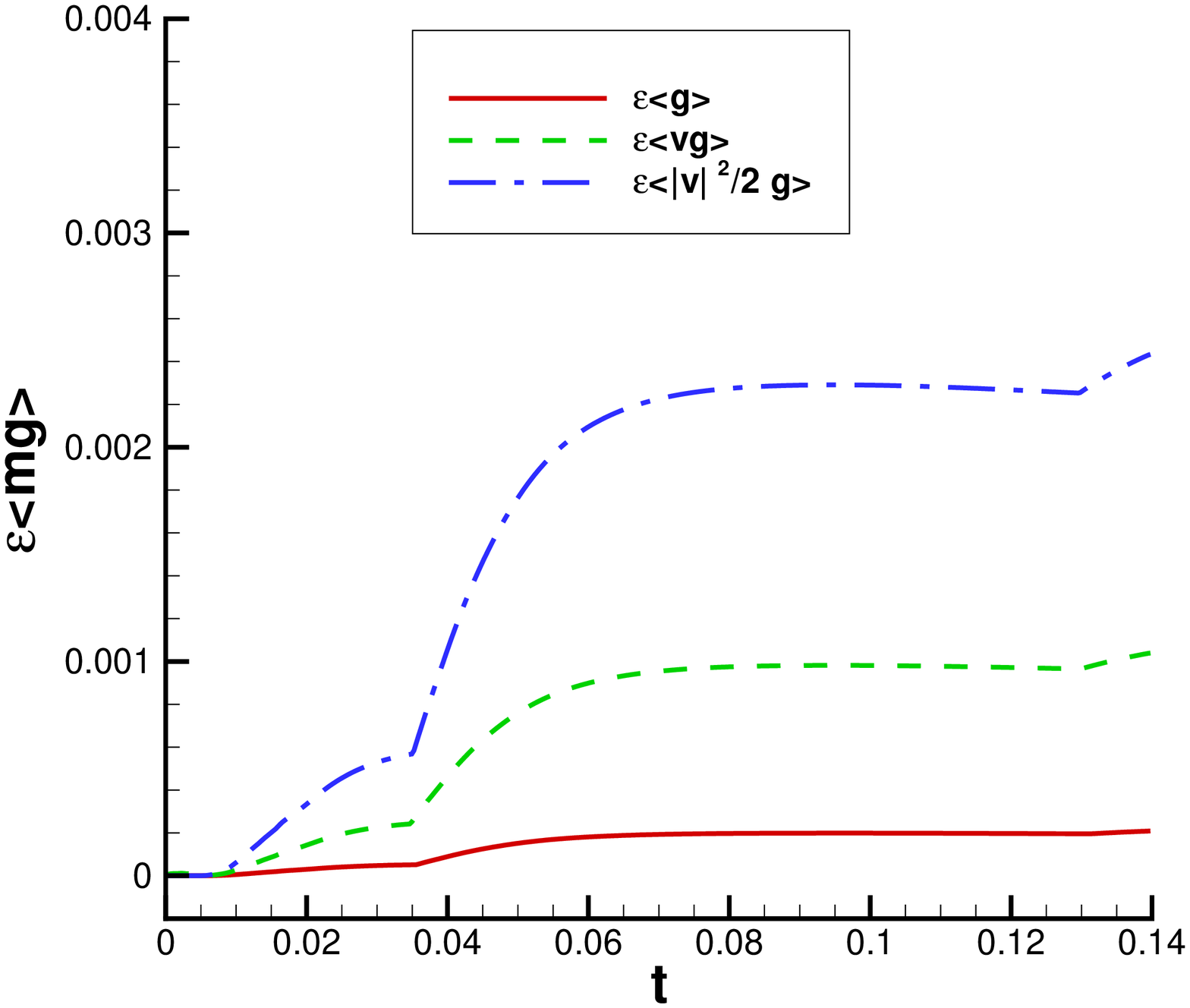},
\includegraphics[totalheight=2.5in]{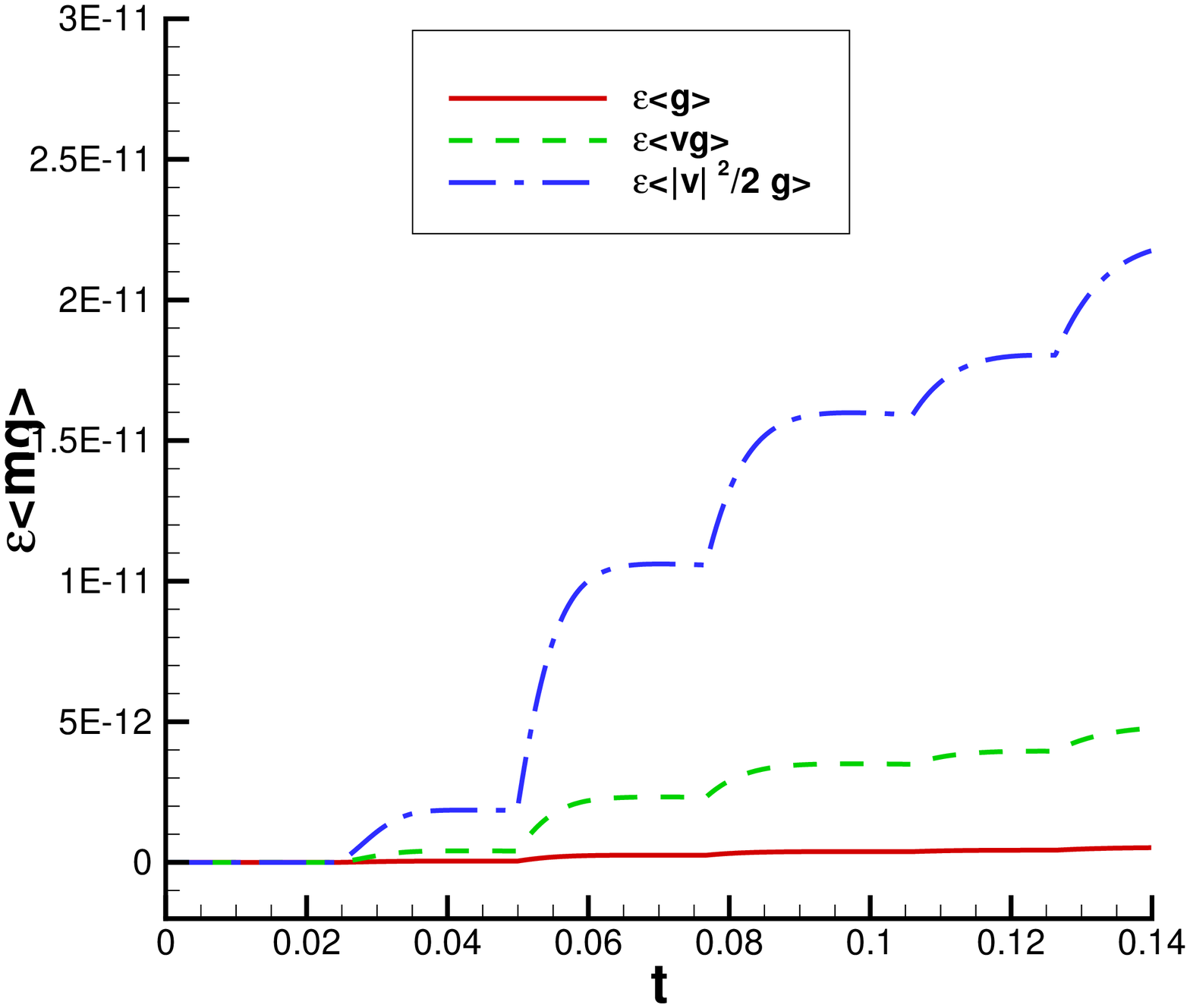}\\
\includegraphics[totalheight=2.5in]{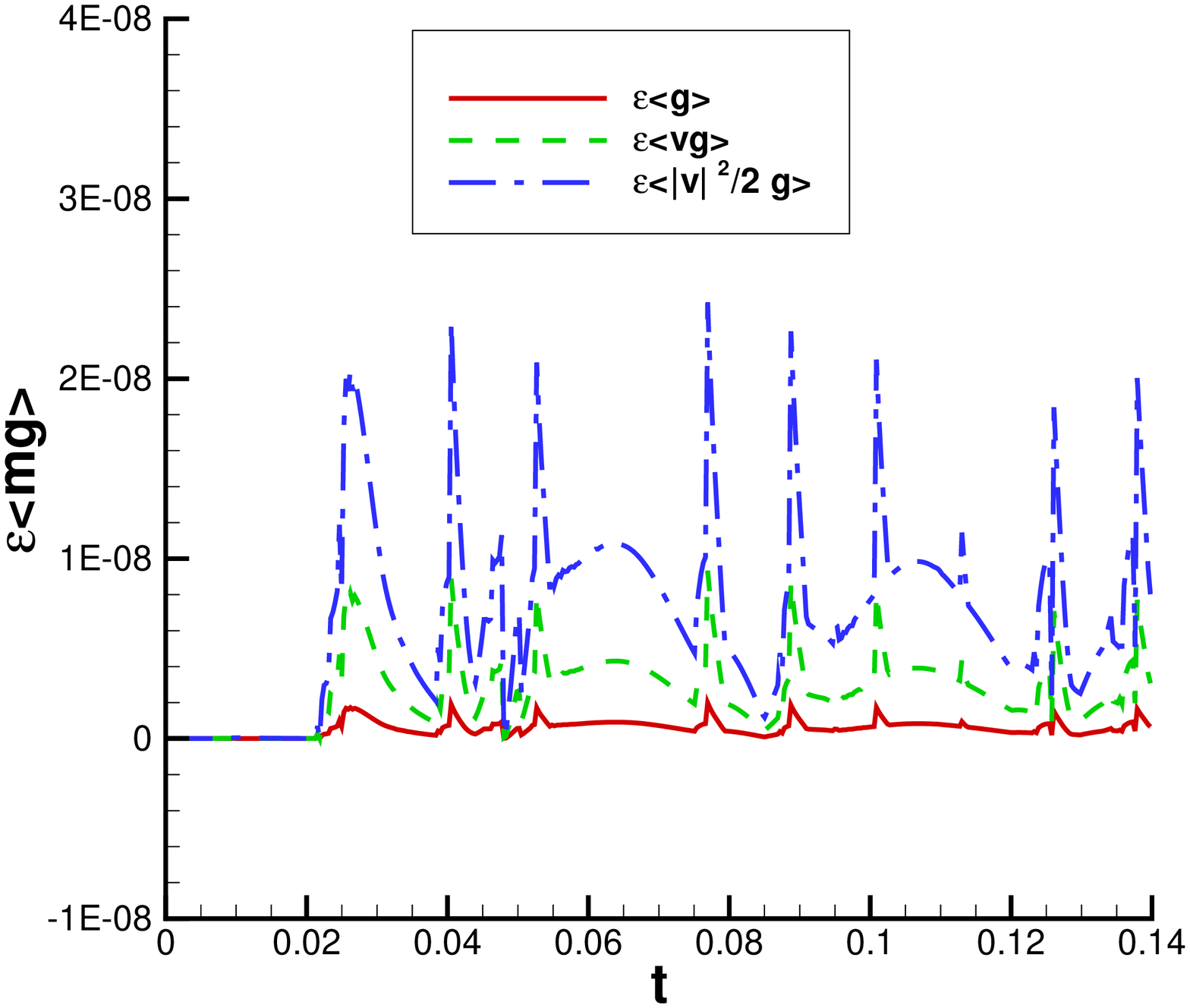},
\includegraphics[totalheight=2.5in]{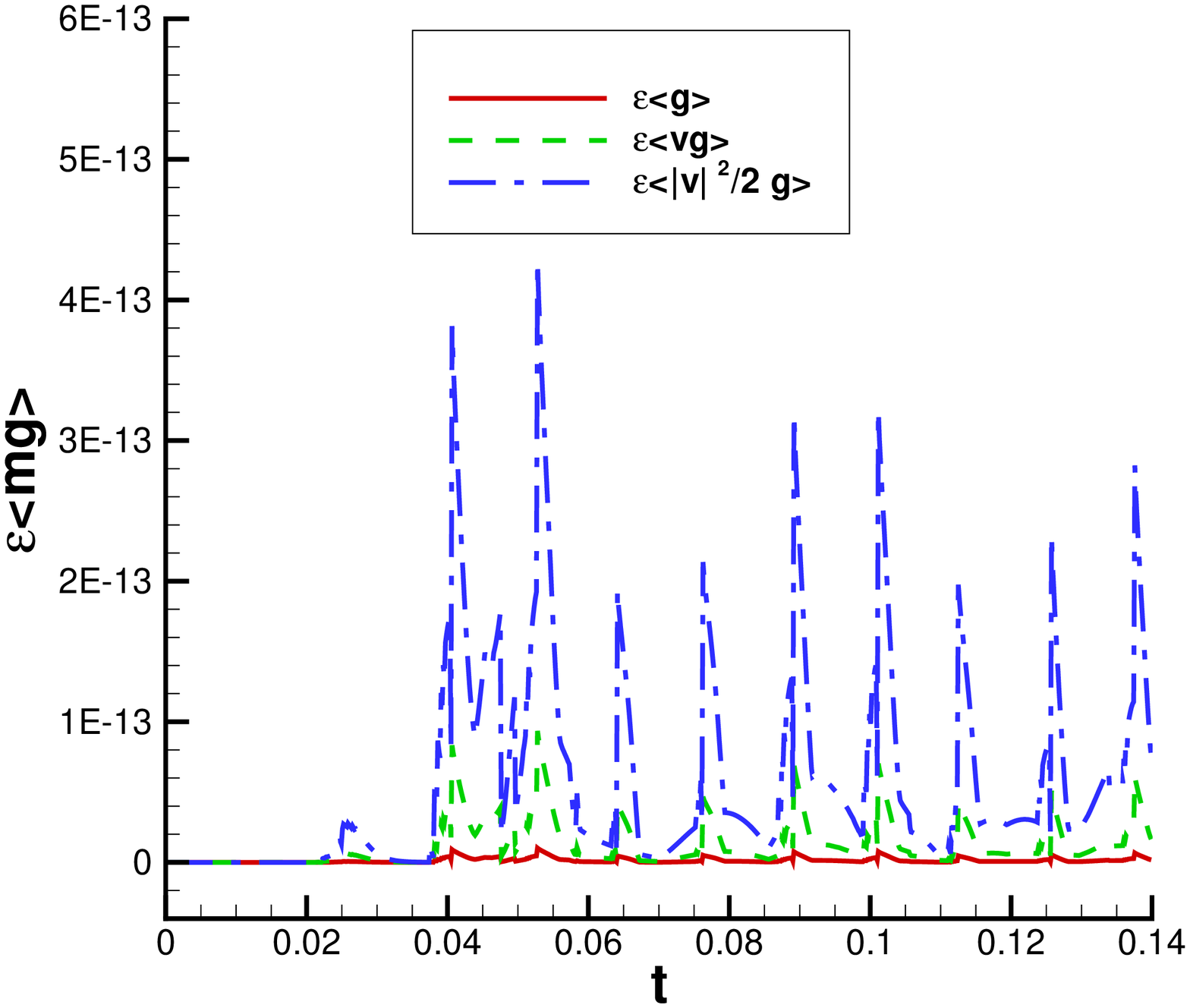}
\caption{$\eps \langle mg \rangle$ for the Sod problem (\ref{sod}) with NDG3 on the spatial domain
$[-0.2, 1.2]$. Measured as the maximal value over $x$. $N_x=50$ and $N_v=100$. Top: $\eps=1$; Bottom: $\eps=10^{-6}$. Left: $V_c=4.5$; Right: $V_c=9$. With TVB limiter and $M_{tvb}=20$.}
\label{fig3}
\end{figure}

\begin{figure}[ht]
\centering
\includegraphics[totalheight=2.5in]{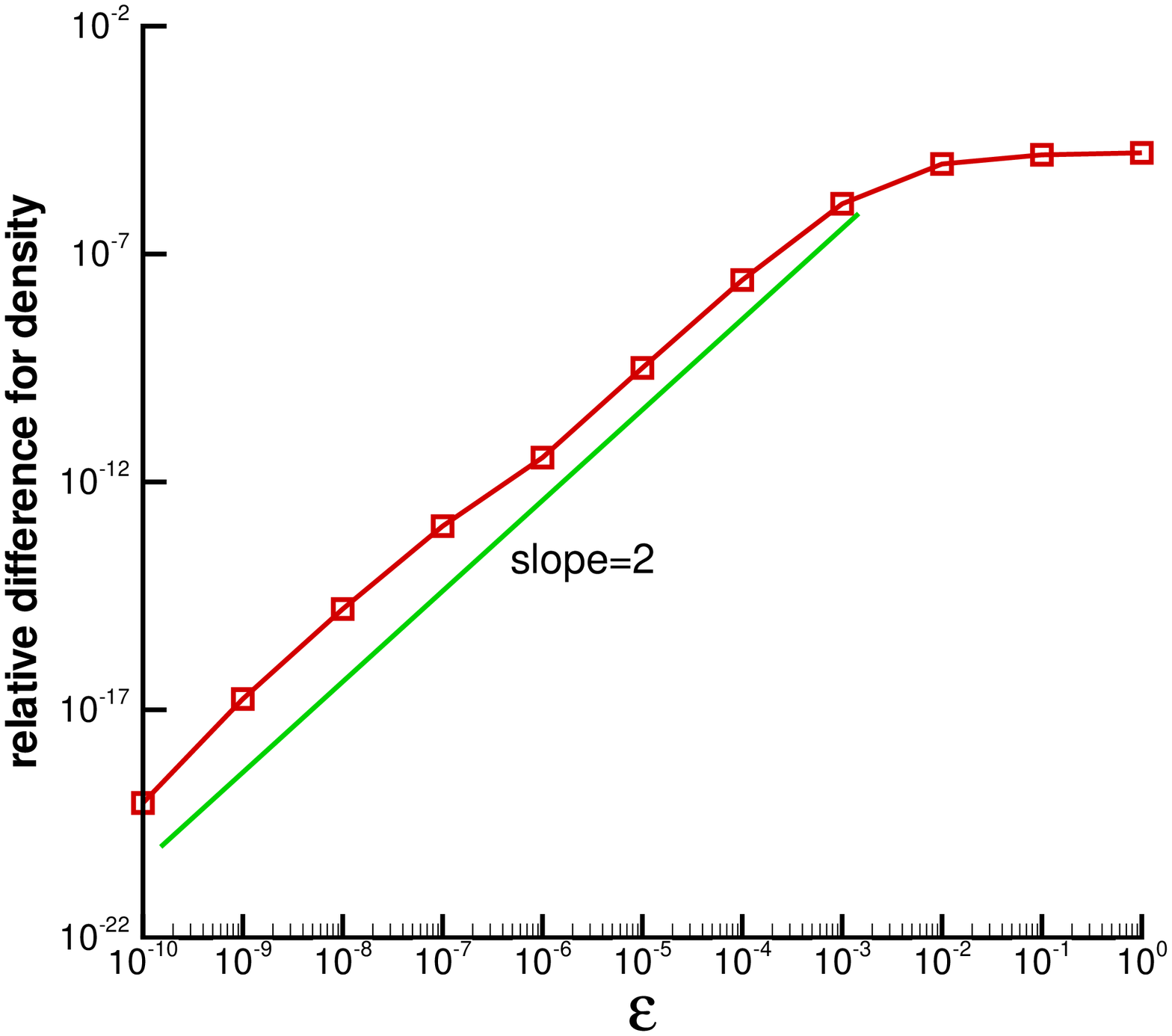},
\includegraphics[totalheight=2.5in]{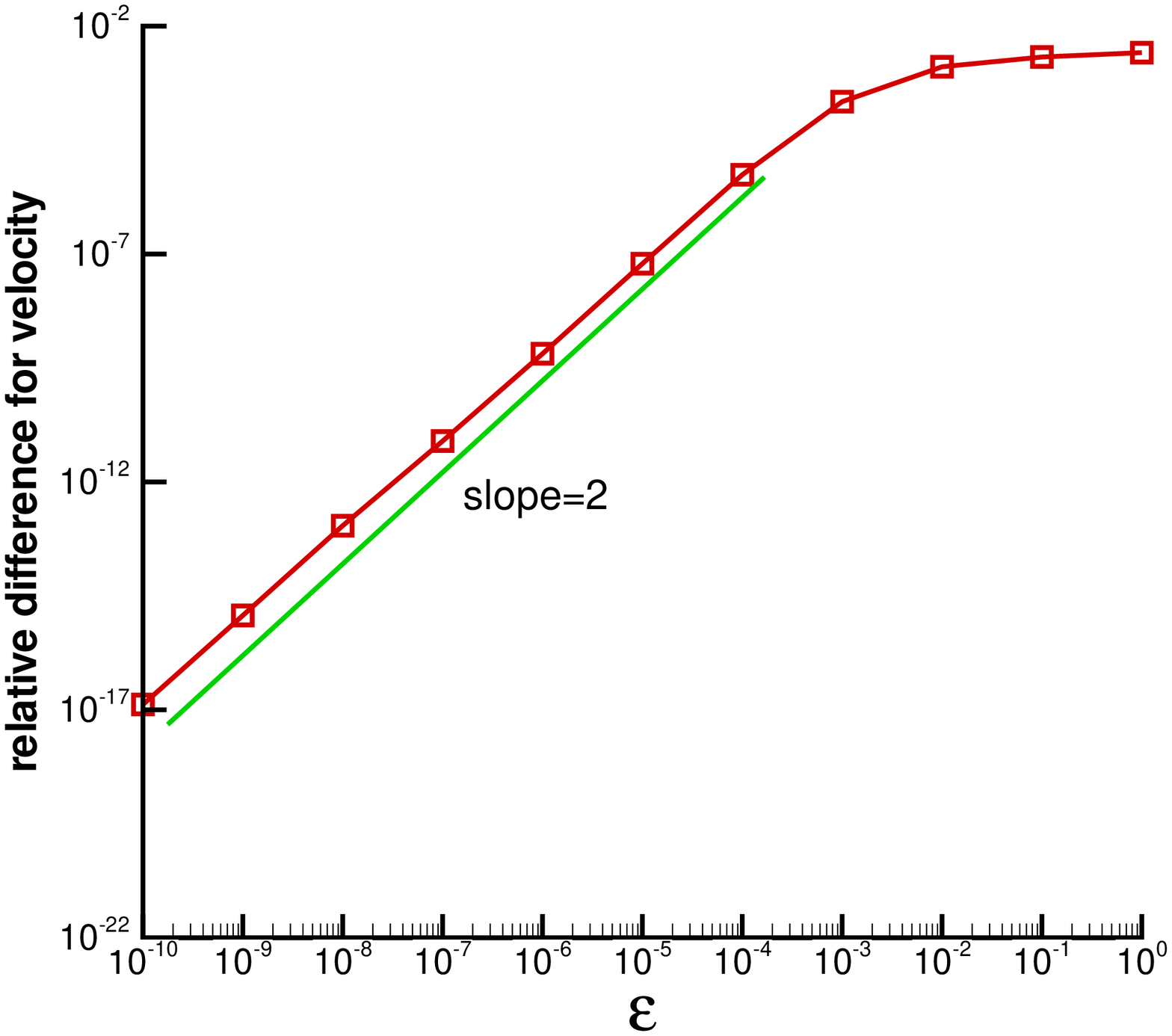} \\
\includegraphics[totalheight=2.5in]{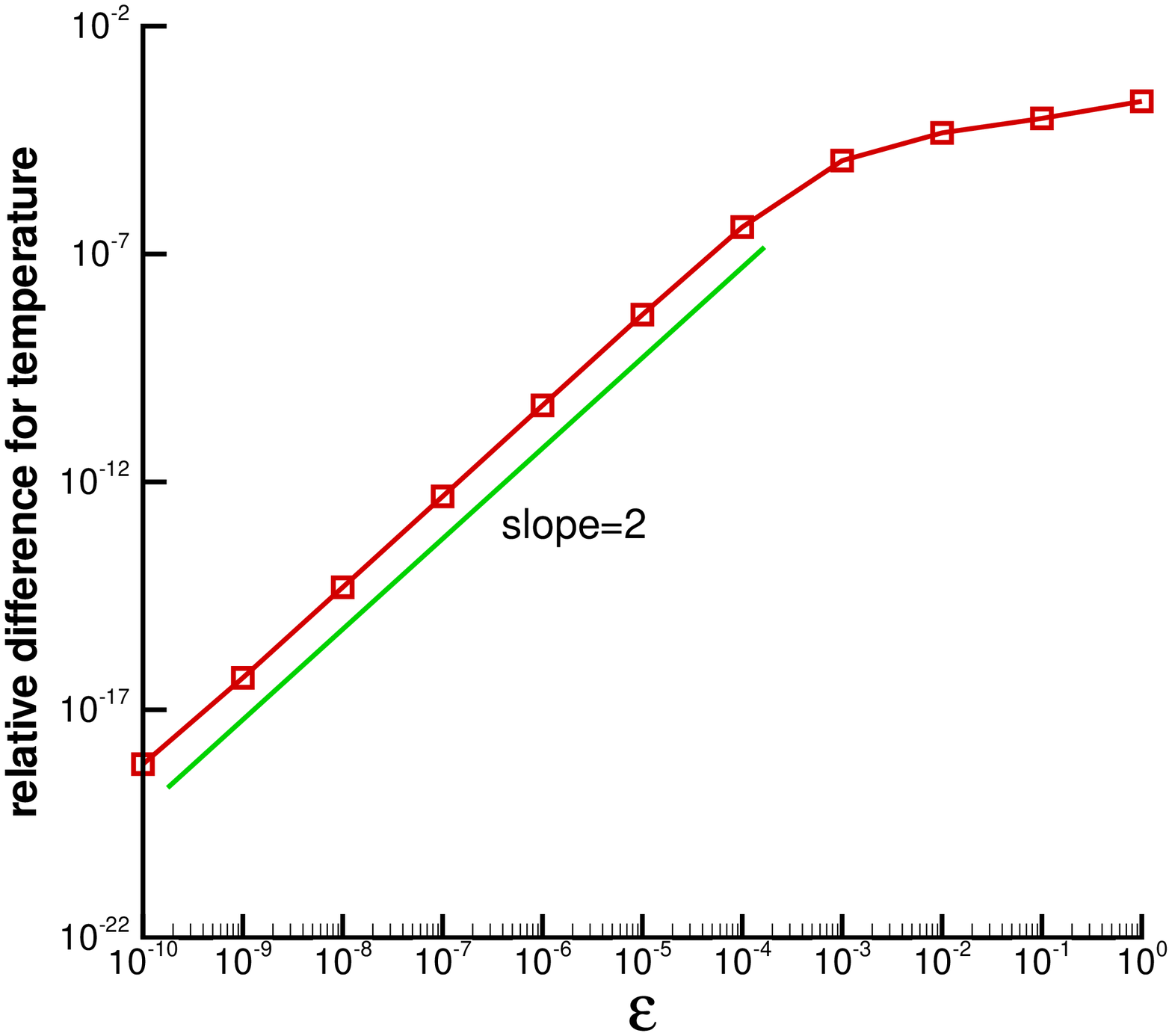}
\caption{The relative differences (in the logarithmic scale) for the Sod problem (\ref{sod}) between the solutions by NDG3 solving the BGK equation and those by the third order local DG method in \eqref{eq: compact_cns_rk0}-\eqref{eq: q_rk} for the compressible Navier-Stokes system. Top left: density. Top right: velocity. Bottom: temperature. $N_x=50$ and $N_v=100$ on the domain $[-0.2, 1.2]\times [-9, 9]$. Without TVB limiter.}
\label{fig41}
\end{figure}
\end{exa}

\begin{exa}
In this example we take the initial conditions for the density, the mean velocity and the pressure to be the same as the classic Lax shock tube problem \cite{harten1987uniformly,shu1989efficient}, which are
\begin{eqnarray}
(\rho, u, p)=
\begin{cases}
(0.445, 0.698, 3.528), & \quad 0\le x \le 0.5, \\
(0.5, 0, 0.571), & \quad 0.5< x \le 1.
\end{cases}
\label{lax}
\end{eqnarray}
We compute the solutions with NDG3  up to time $t=0.1$. The initial distribution function $f$ is the Maxwellian \eqref{maxwellian} with $g(x,v,0)=0$. The spatial domain $[-0.5, 1.5]$ is discretized with $N_x=100$ grid points, and the velocity domain $[-8, 8]$ is discretized with $N_v=100$ uniform points. The boundary conditions are taken to be the initial constant values at both ends in the $x$ direction. We show the distribution function $f$ at $x=0.5$ in Fig. \ref{fig4}. The density, the mean velocity, the temperature and the rescaled heat flux are shown in Fig. \ref{fig5}. The conserved properties are demonstrated by plotting $\eps\langle mg \rangle$ in Fig. \ref{fig6}, with $m=(1, v, |v|^2/2)^t$ and for $\eps=1$ and $\eps=10^{-6}$.

\begin{figure}[ht]
\centering
\includegraphics[totalheight=2.5in]{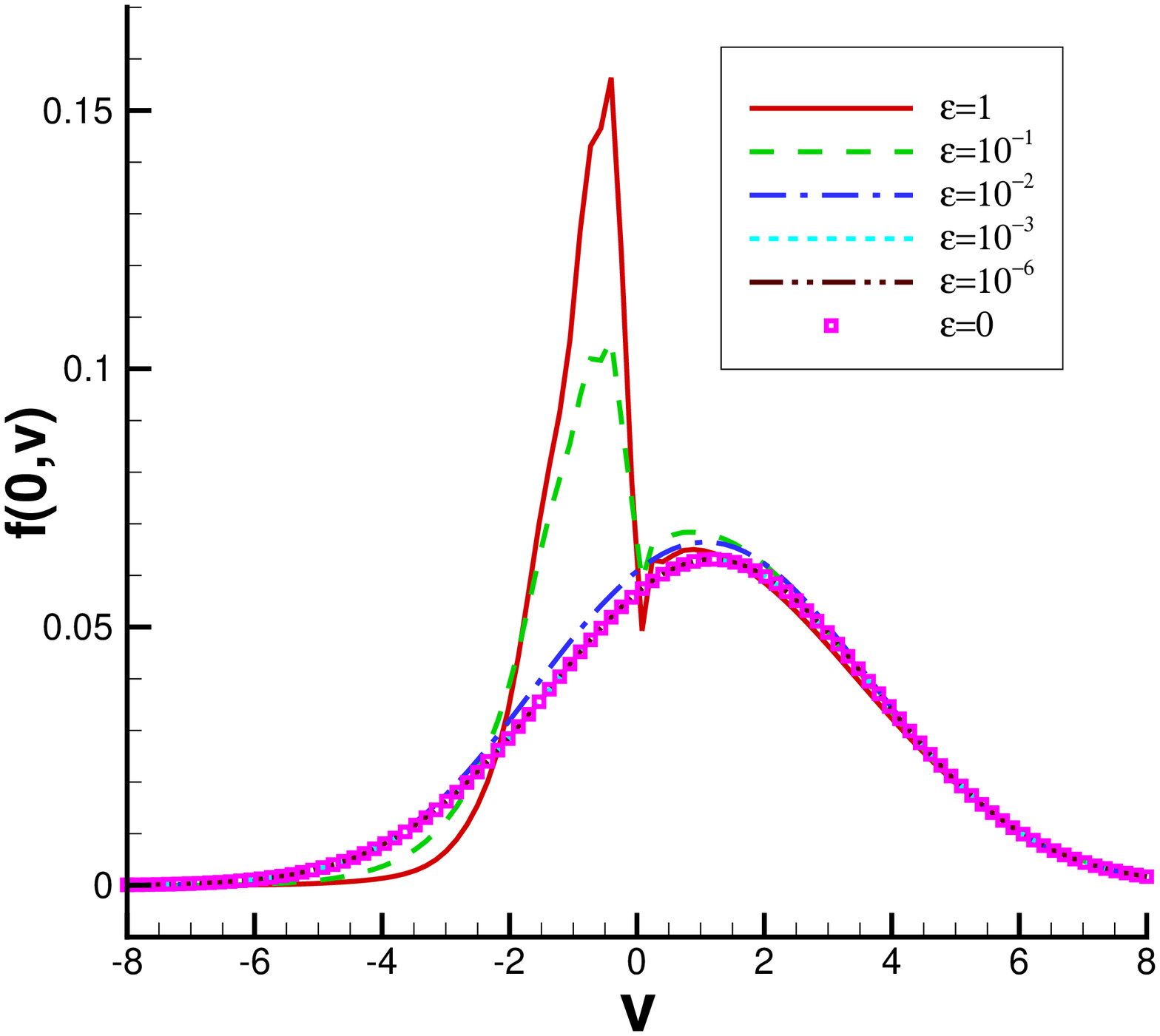}
\caption{Distribution function $f$ for the Lax problem (\ref{lax}) at $x=0.5$ as a function of $v\in[-8, 8]$.
$N_x=100$ and $N_v=100$ for NDG3 at $t=0.1$. As $\eps$ goes to $0$, $f$ becomes close to a Maxwellian.
With TVB limiter and $M_{tvb}=20$.}
\label{fig4}
\end{figure}

\begin{figure}[ht]
\centering
\includegraphics[totalheight=2.5in]{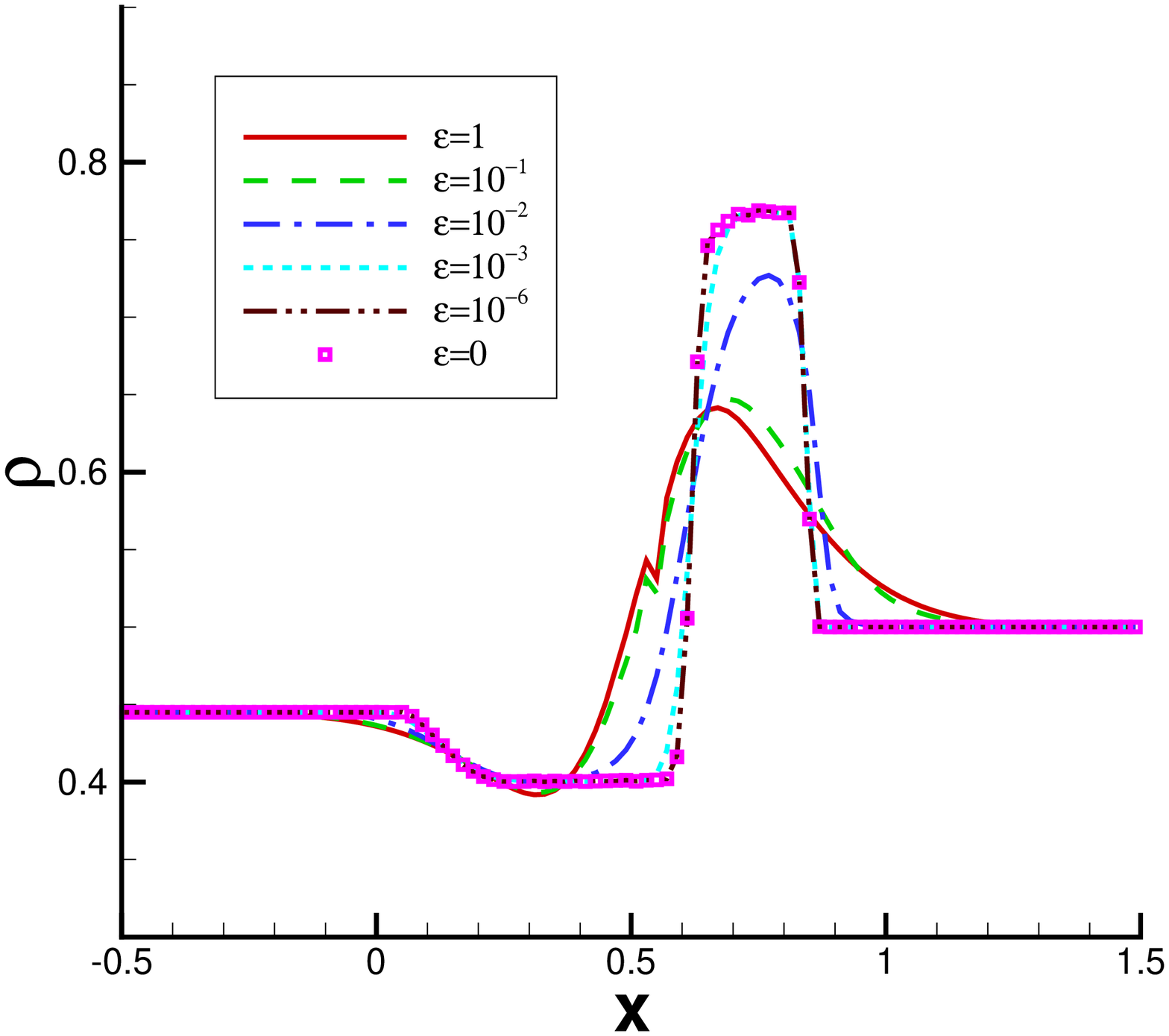},
\includegraphics[totalheight=2.5in]{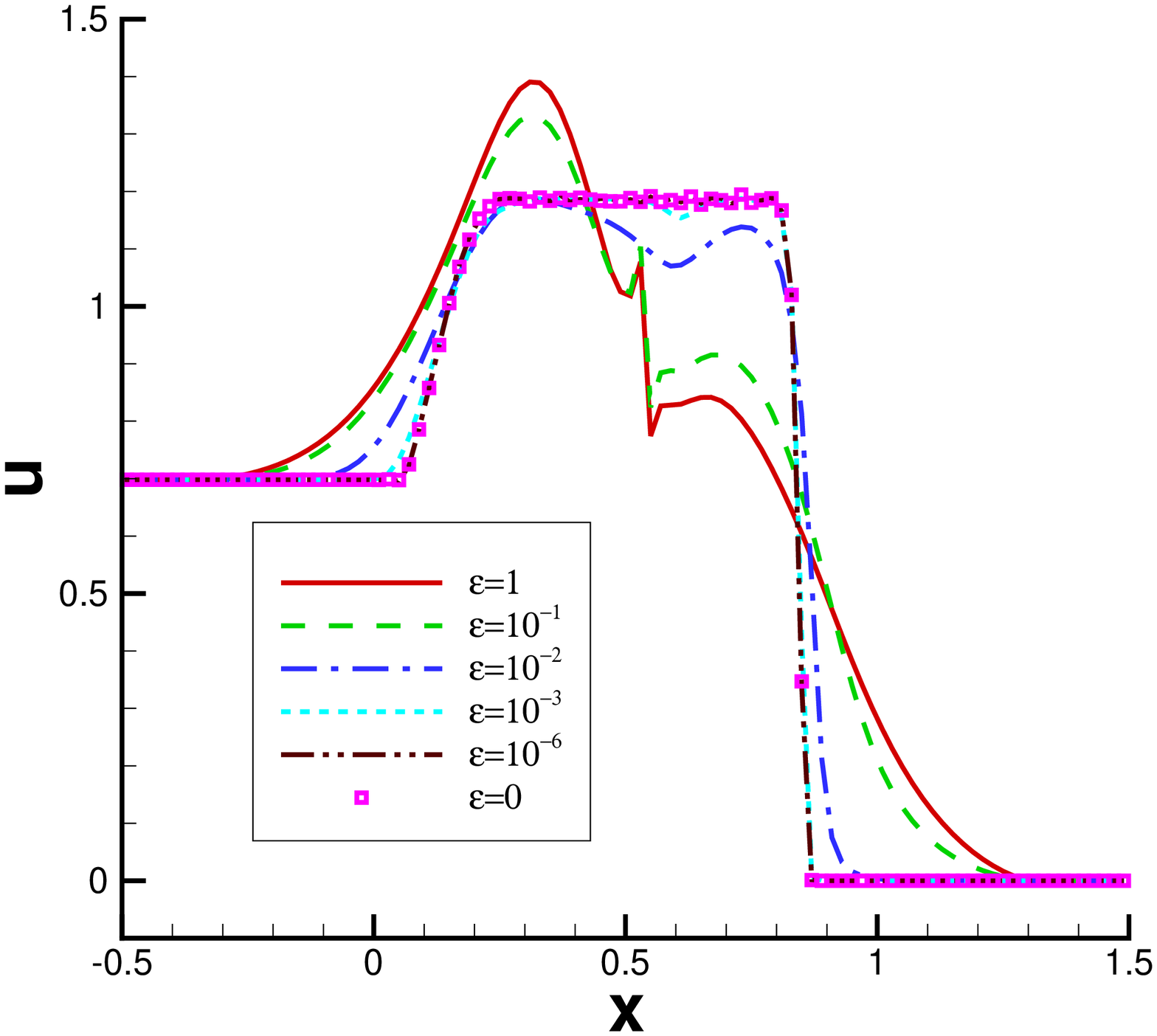}\\
\includegraphics[totalheight=2.5in]{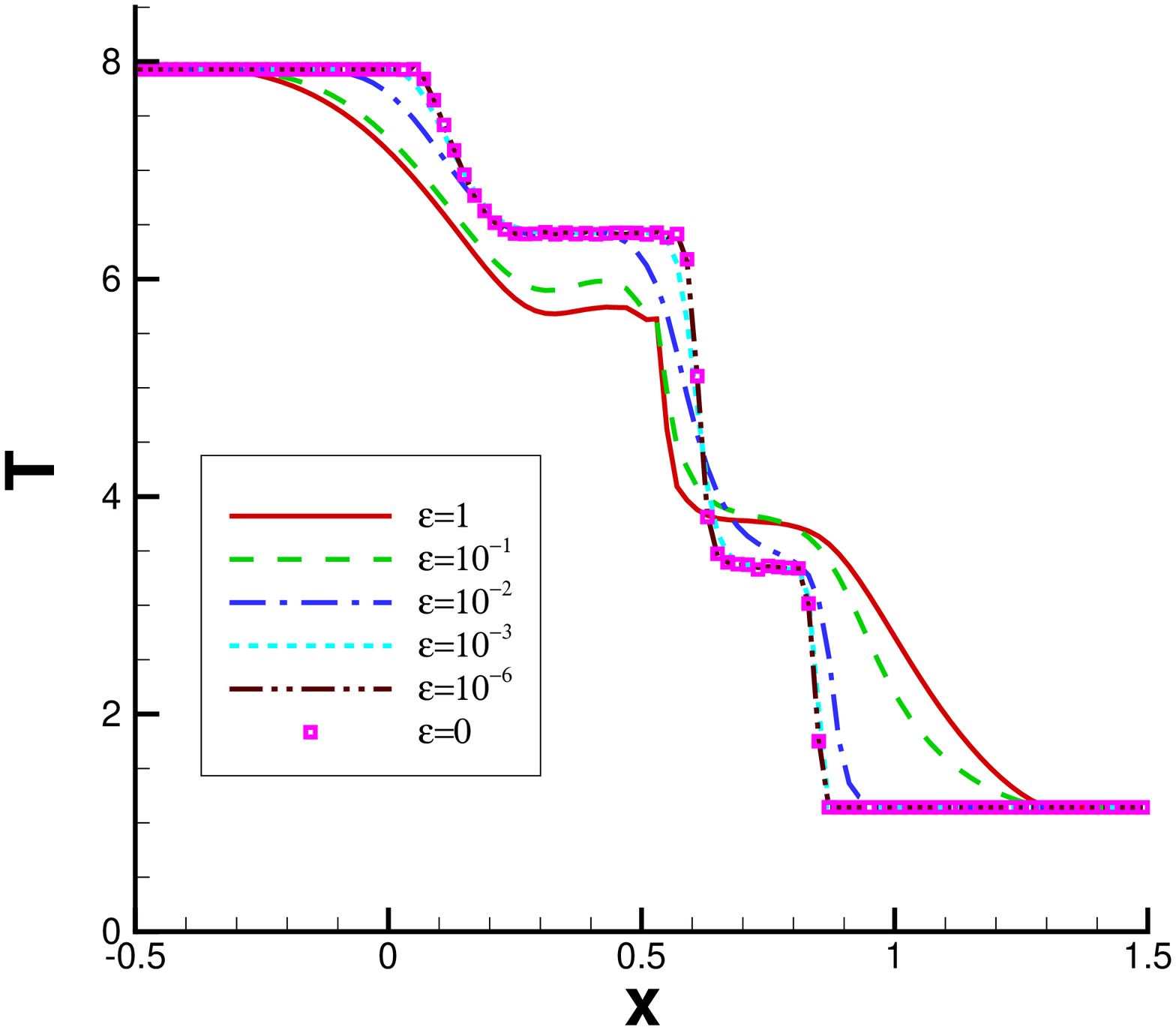},
\includegraphics[totalheight=2.5in]{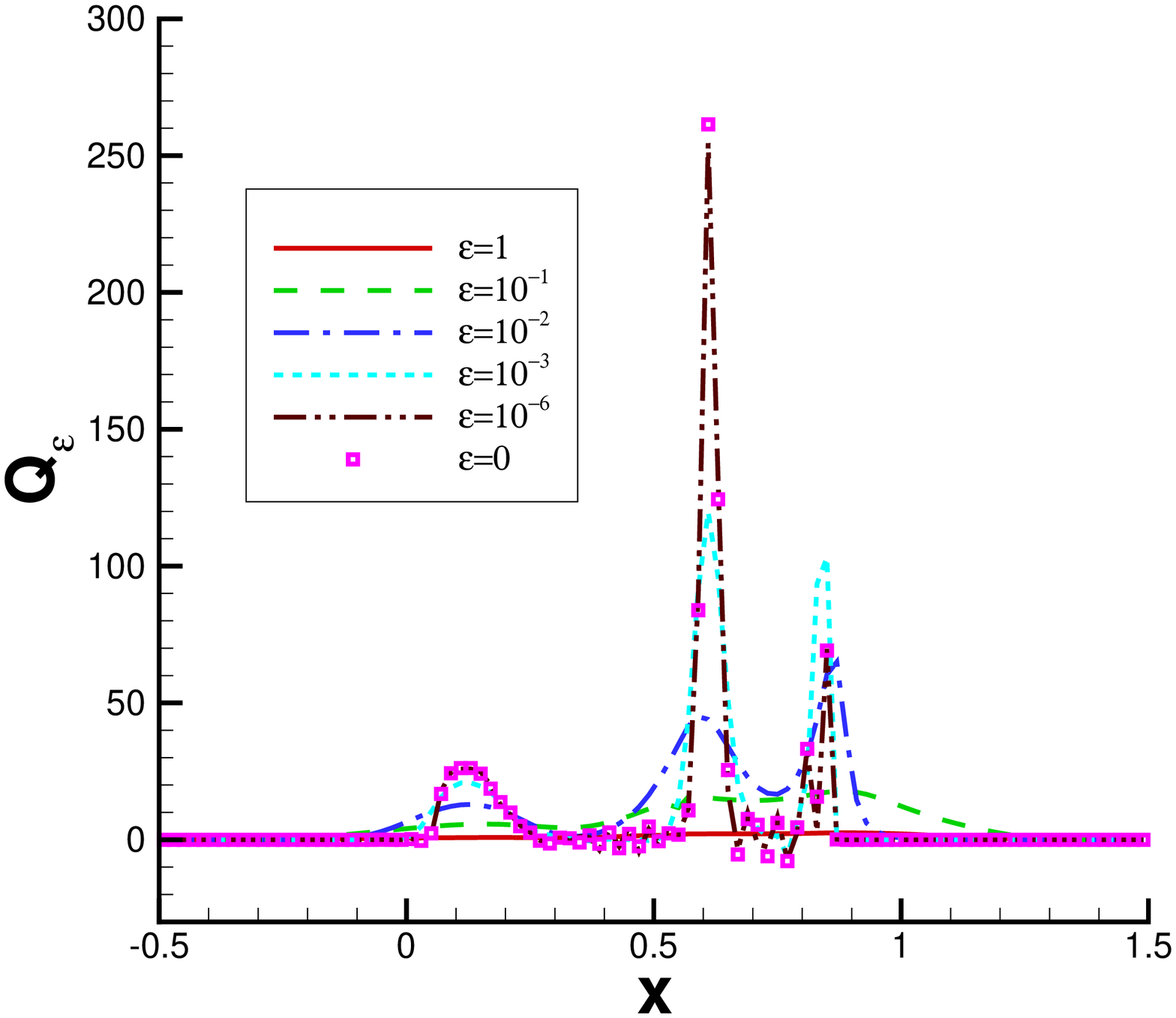}
\caption{Numerical solutions for the Lax problem (\ref{lax}) at $t=0.1$ with NDG3.
$N_x=100$ and $N_v=100$ on the domain $[-0.5, 1.5]\times [-8, 8]$. Top: density, mean velocity;
Bottom: temperature, rescaled heat flux. With TVB limiter and $M_{tvb}=20$.}
\label{fig5}
\end{figure}

\begin{figure}[ht]
\centering
\includegraphics[totalheight=2.5in]{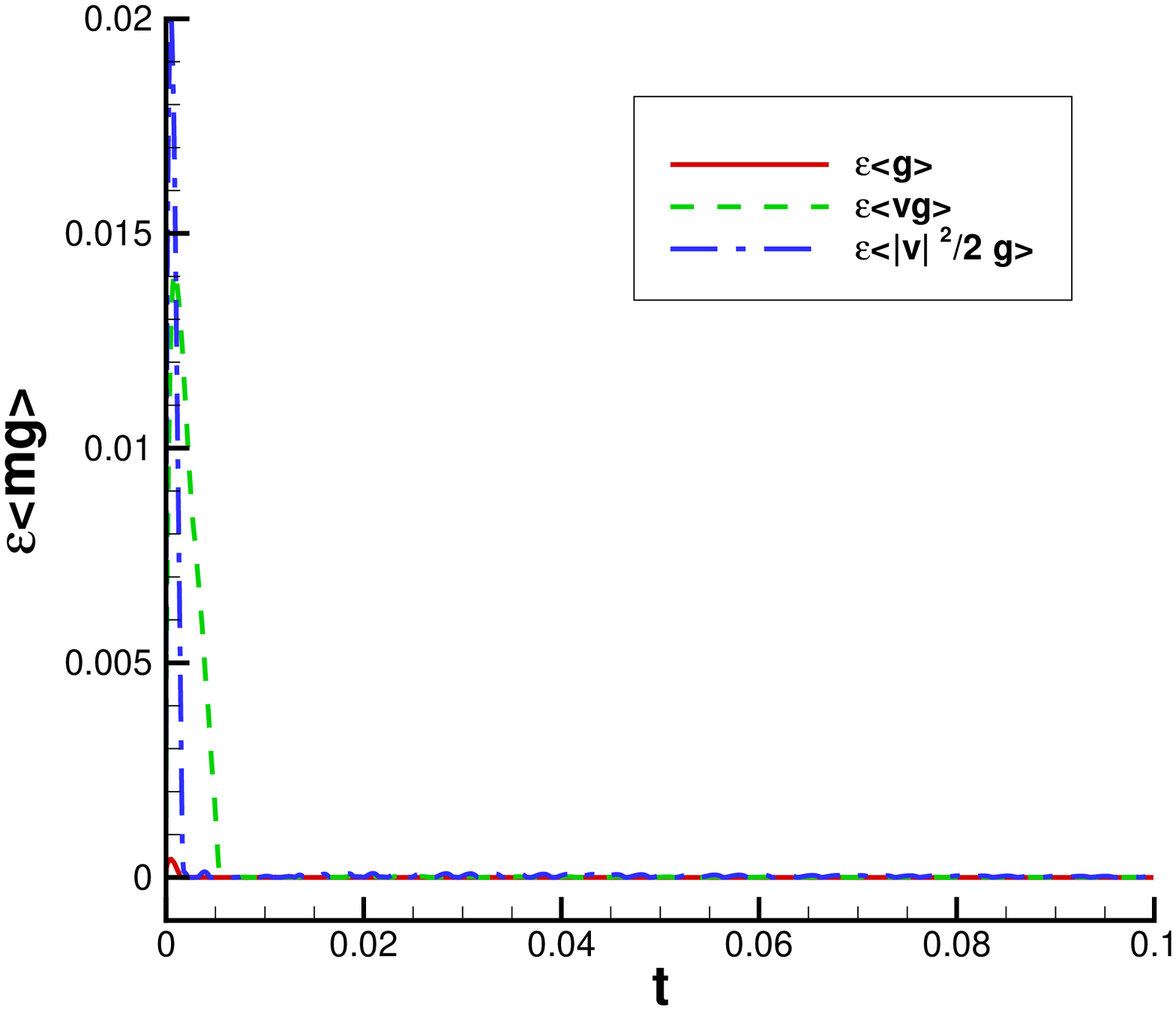},
\includegraphics[totalheight=2.5in]{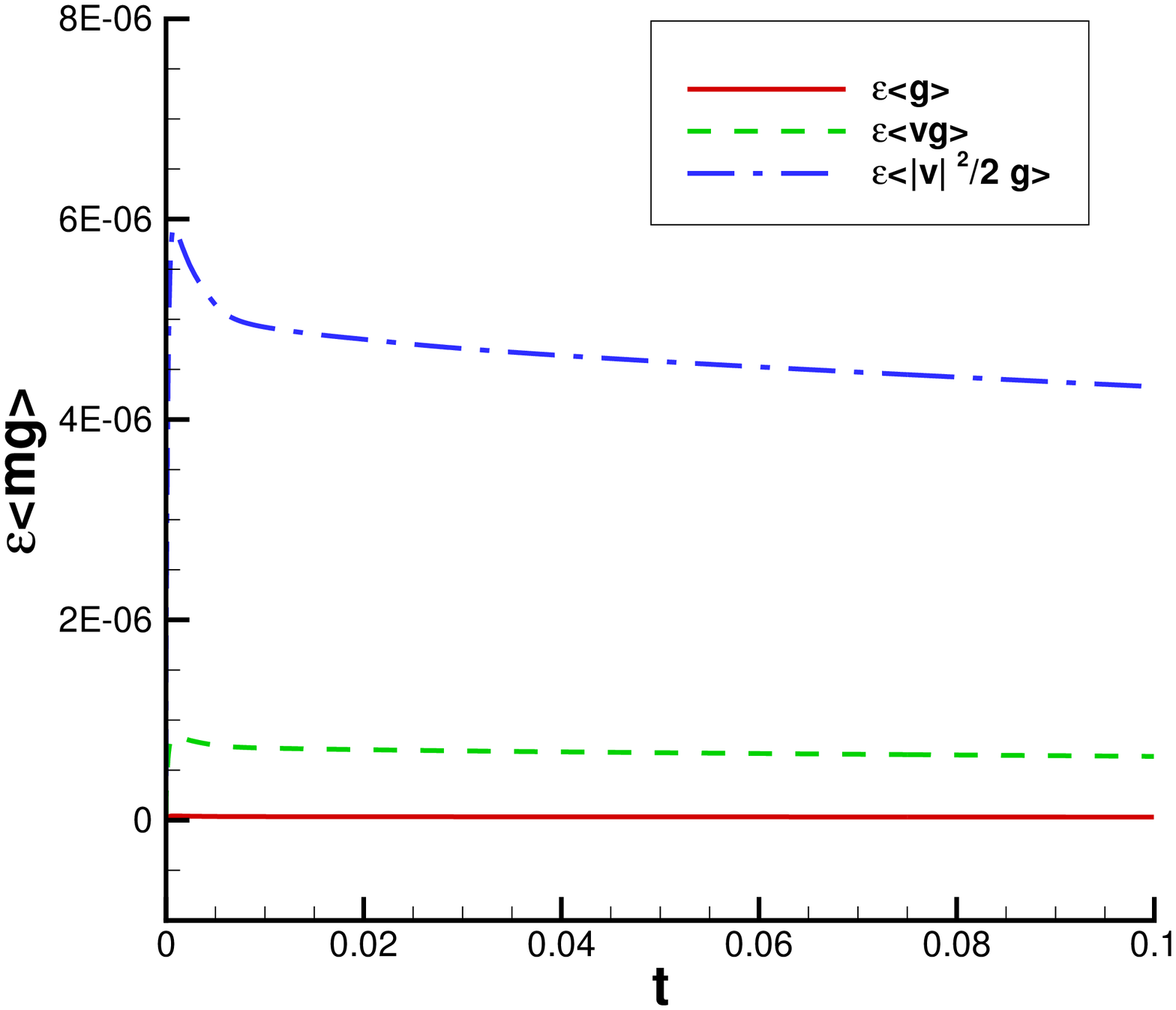} \\
\includegraphics[totalheight=2.5in]{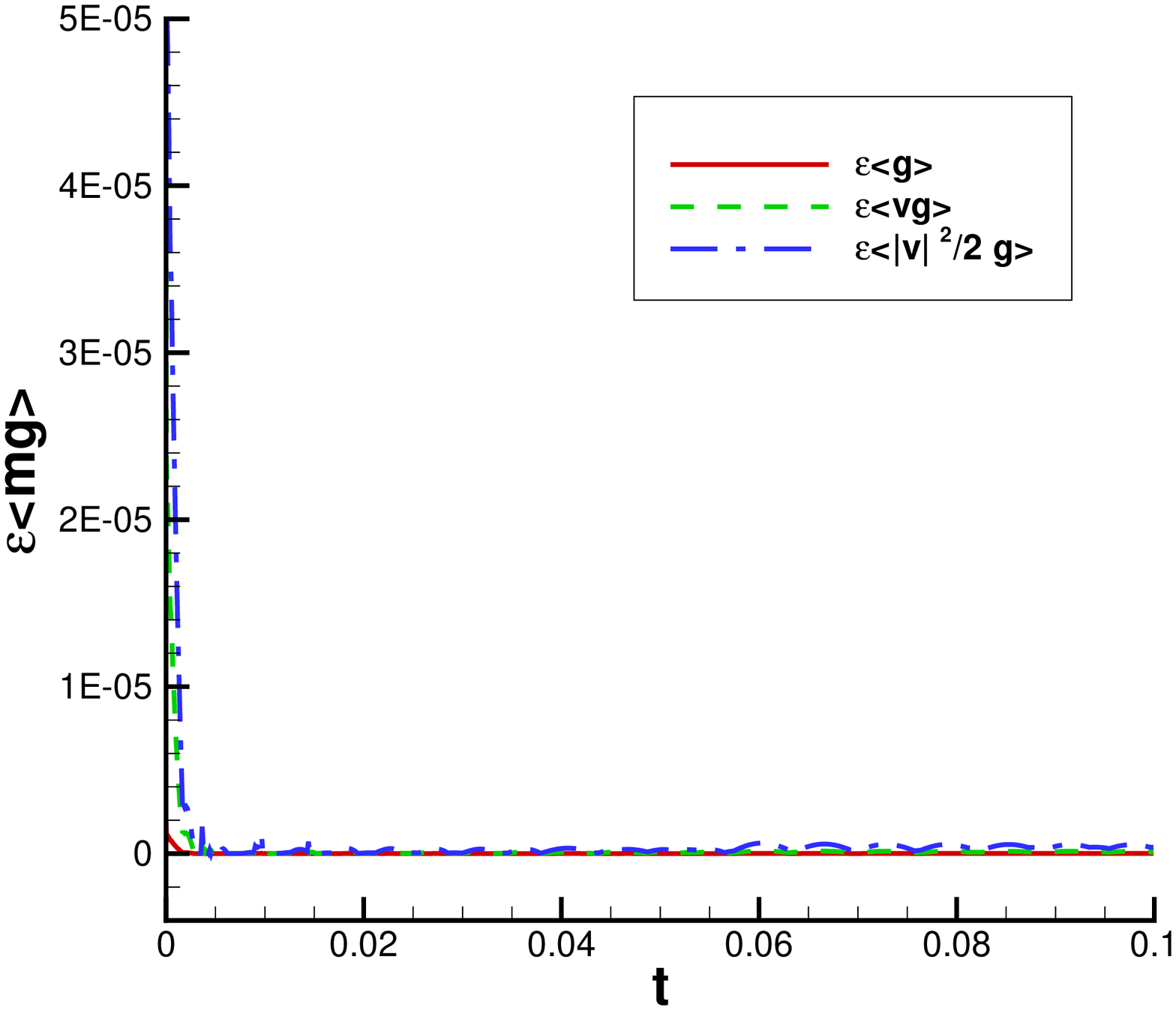},
\includegraphics[totalheight=2.5in]{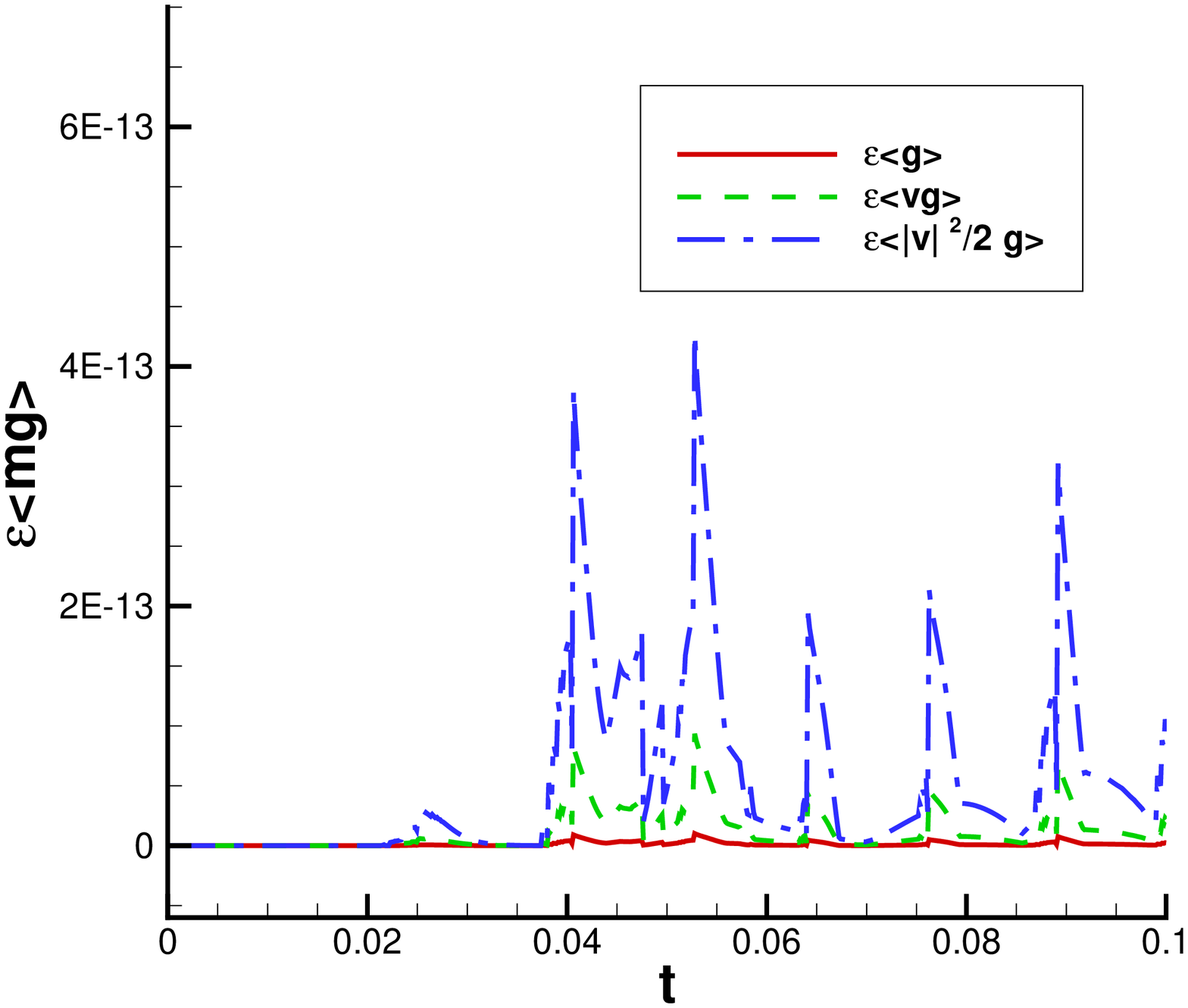}
\caption{$\eps \langle mg \rangle$ for the Lax problem (\ref{lax}) with NDG3. Measured as the maximal
value over $x$. $N_x=100$ and $N_v=100$ on the spatial domain $[-0.5, 1.5]$. Top: $\eps=1$;
Bottom: $\eps=10^{-6}$. Left: $V_c=8$; Right: $V_c=16$. With TVB limiter and $M_{tvb}=20$. }
\label{fig6}
\end{figure}
\end{exa}

\begin{exa}
The Shu-Osher problem \cite{shu1989efficient} is a pure right moving shock with a sine wave structure at the shock front. In this example, we keep the right states of the shock and the ratio of the pressures to be the same as in \cite{shu1989efficient}. The left states are obtained by the Rankine-Hugoniot shock jump condition with $\gamma=3$. Specifically, they are
\begin{eqnarray}
(\rho, u, p)=
\begin{cases}
(1.756757, 2.005122, 10.333333), & \quad  x \le -2 , \\
(1+0.1\sin(x), 0, 1), & \quad x>-2 .
\end{cases}
\label{so}
\end{eqnarray}
Initially the shock is located at $x=-2$ and $g(x,v,0)=0$. The solutions are computed with NDG3 up to time $t=1$.
We take a large enough spatial domain $[-12, 12]$ and it is discretized with $N_x=200$ grid points. The boundary values outside the computational domain in the $x$ direction are extrapolations of the initial values. The velocity domain $[-10, 10]$ is discretized with $N_v=100$ uniform points. We show the distribution function $f$ at $x=0$ in Fig. \ref{fig7}. The density, the mean velocity, the temperature and the rescaled heat flux are shown in Fig. \ref{fig8}. The conserved properties are further illustrated by plotting $\eps\langle mg \rangle$ in Fig. \ref{fig9}, with  $m=(1, v, |v|^2/2)^t$ and for $\eps=1$ and $\eps=10^{-6}$.

For this example, we also compare the performance of the proposed methods with different accuracy orders on a relatively coarser mesh $N_x=50$ and $N_v=100$ in Fig. \ref{fig16}. The reference solution is NDG3 on a mesh of $N_x=800$ and $N_v=200$.
Here we take the TVB parameter $M_{tvb}=1$ and we can see that the higher order schemes NDG2 and NDG3 perform much better than NDG1.

\begin{figure}[ht]
\centering
\includegraphics[totalheight=2.5in]{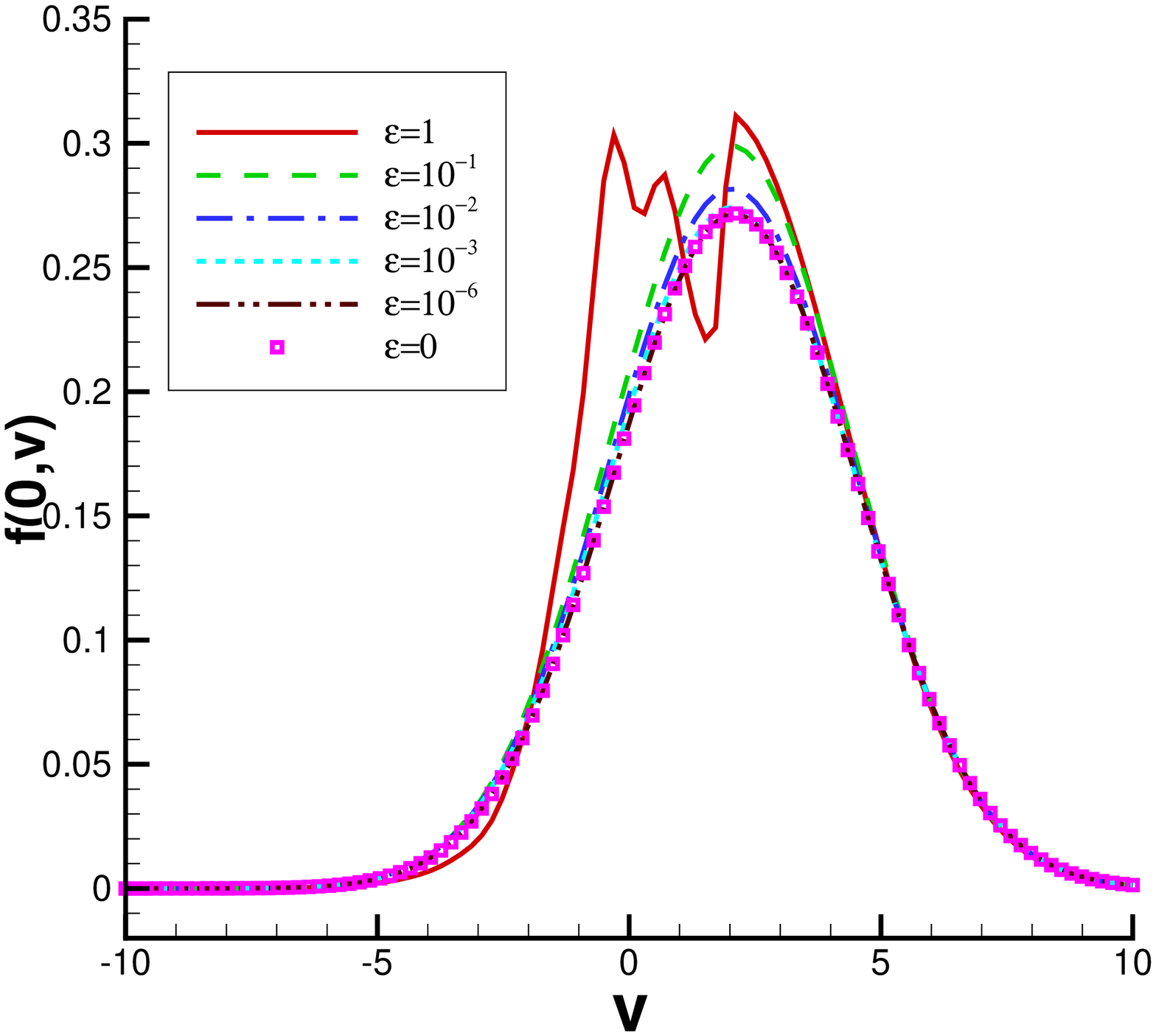}
\caption{Distribution function $f$ for the Shu-Osher problem (\ref{so}) at $x=0$ as a function of
$v\in[-10, 10]$. $N_x=200$ and $N_v=100$ for NDG3  at $t=1$. As $\eps$ goes to $0$, $f$ becomes close to a Maxwellian. With TVB limiter and $M_{tvb}=20$.}
\label{fig7}
\end{figure}

\begin{figure}[ht]
\centering
\includegraphics[totalheight=2.5in]{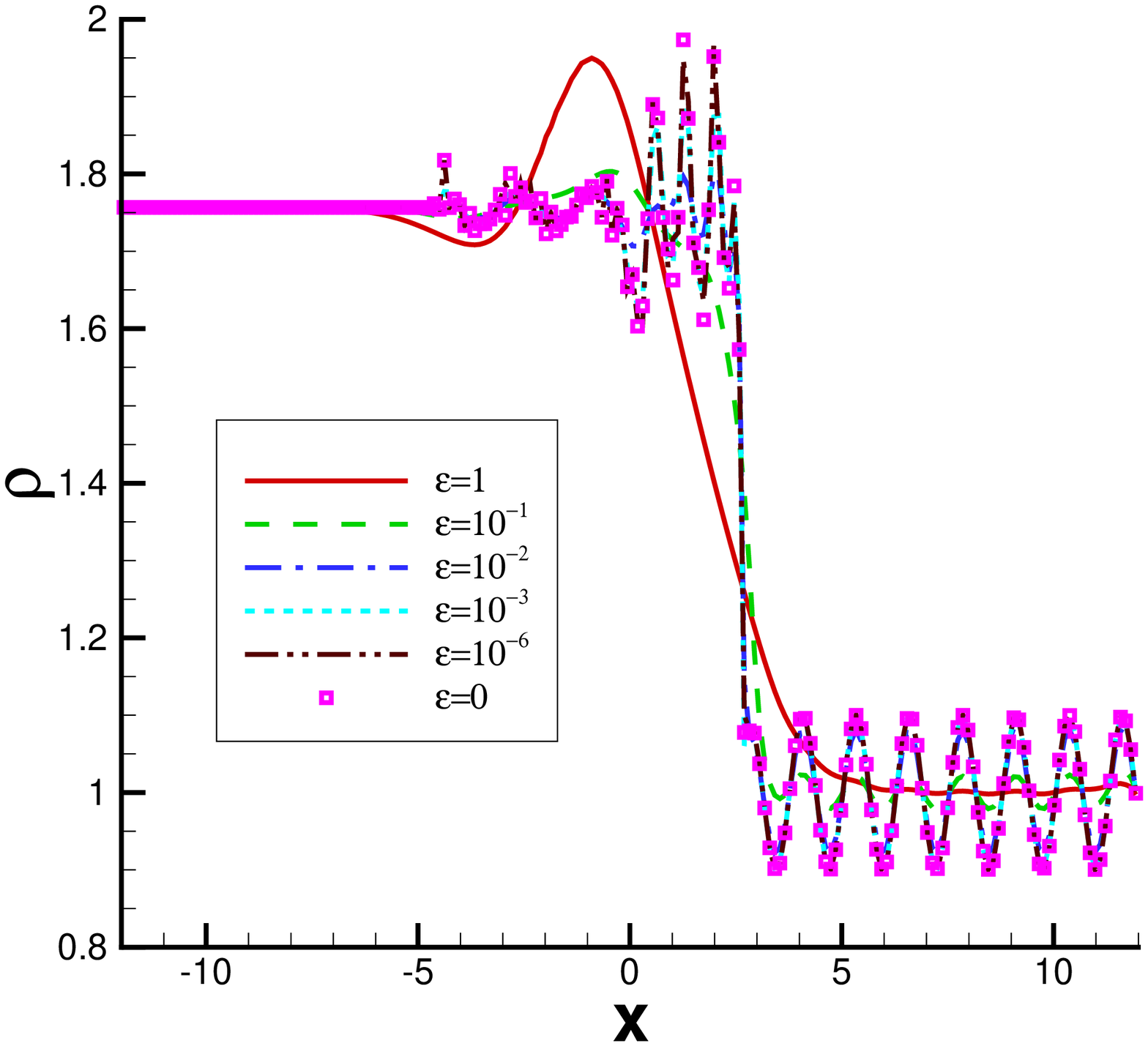},
\includegraphics[totalheight=2.5in]{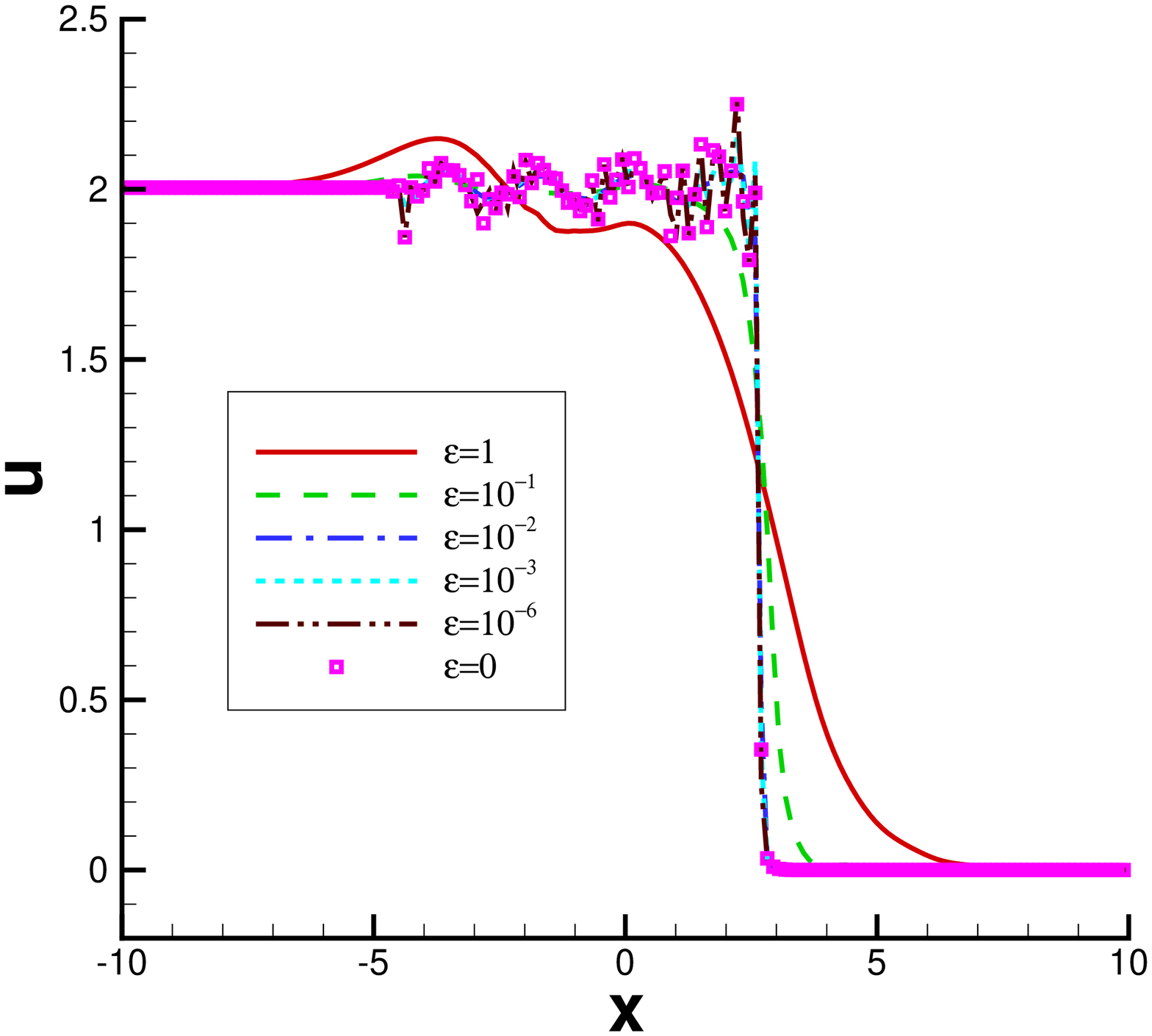}\\
\includegraphics[totalheight=2.5in]{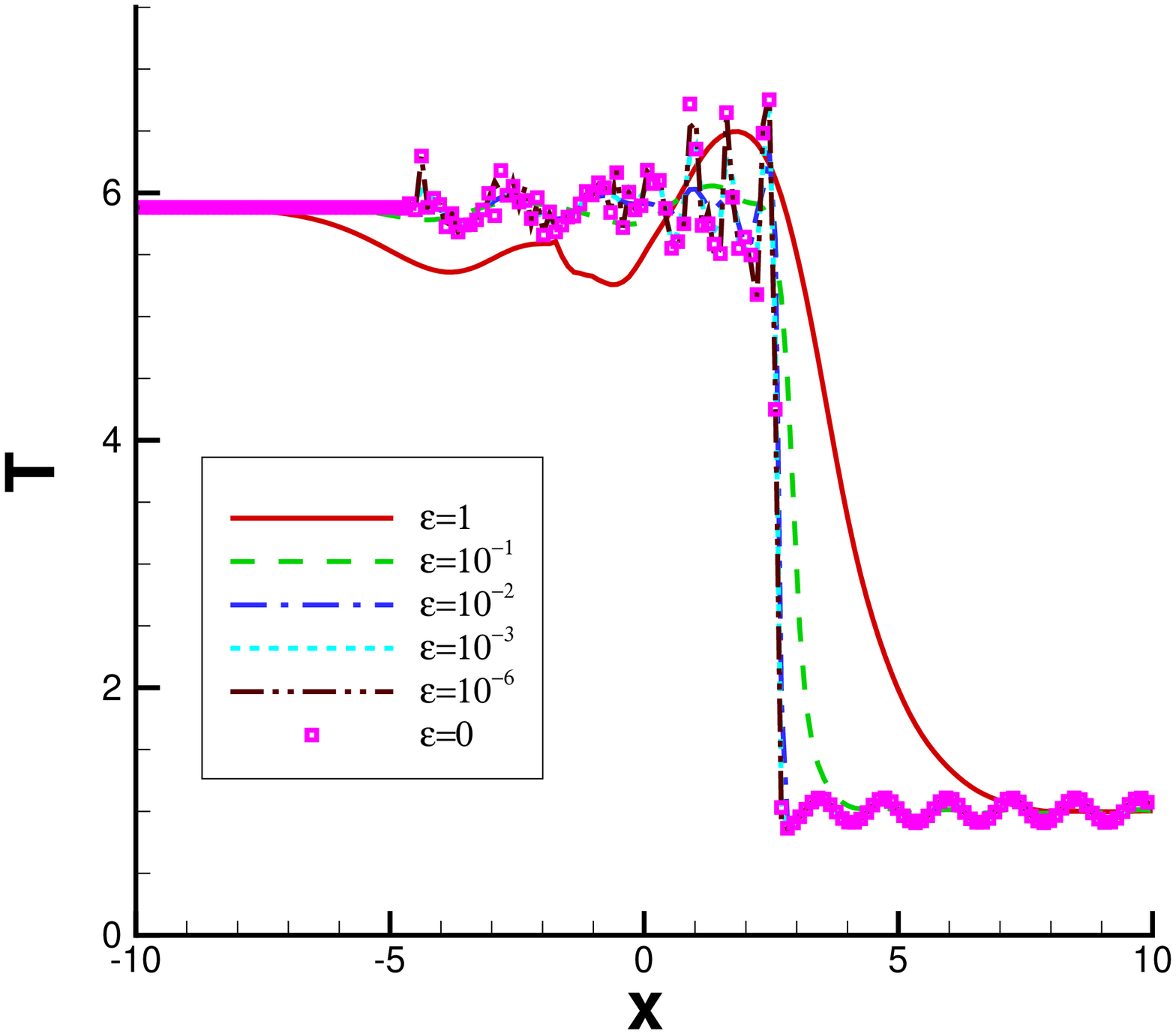},
\includegraphics[totalheight=2.5in]{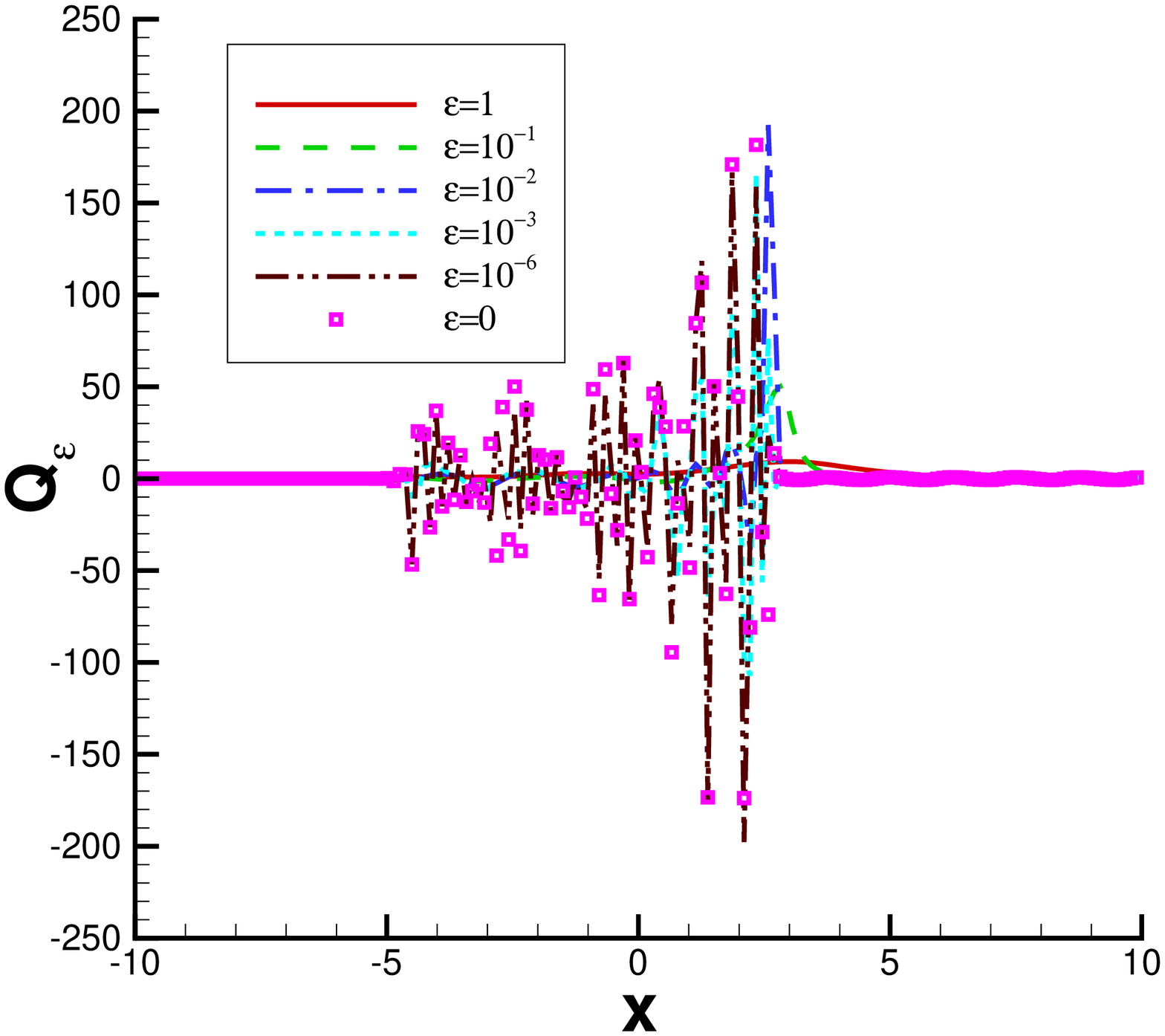}\\
\caption{Numerical solutions for the Shu-Osher problem (\ref{so}) at $t=1$ with NDG3.
$N_x=200$ and $N_v=100$ on the domain $[-12, 12]\times [-10, 10]$. Top: density, mean velocity;
Bottom: temperature, rescaled heat flux. With TVB limiter and $M_{tvb}=20$.}
\label{fig8}
\end{figure}

\begin{figure}[ht]
\centering
\includegraphics[totalheight=2.5in]{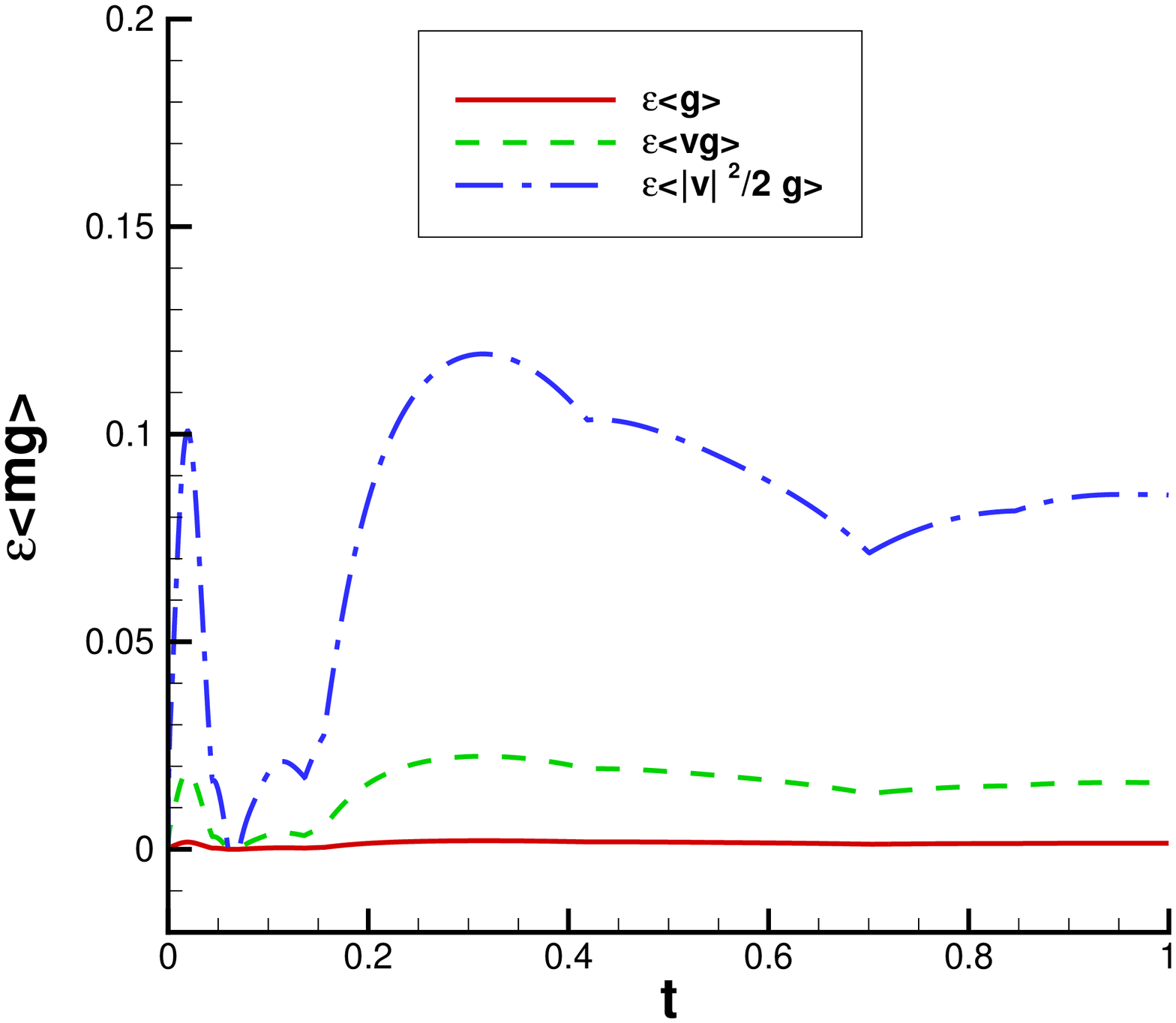},
\includegraphics[totalheight=2.5in]{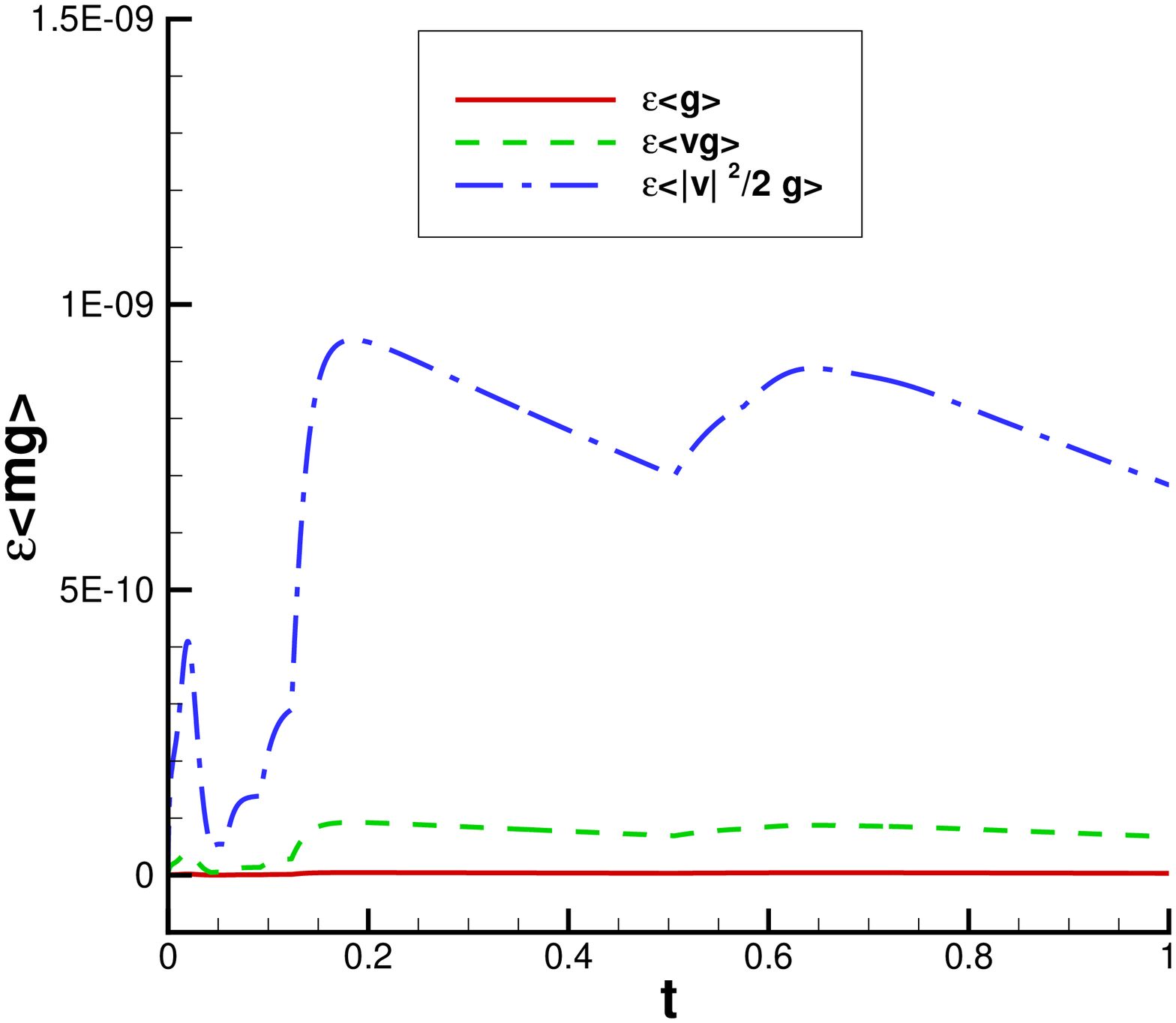} \\
\includegraphics[totalheight=2.5in]{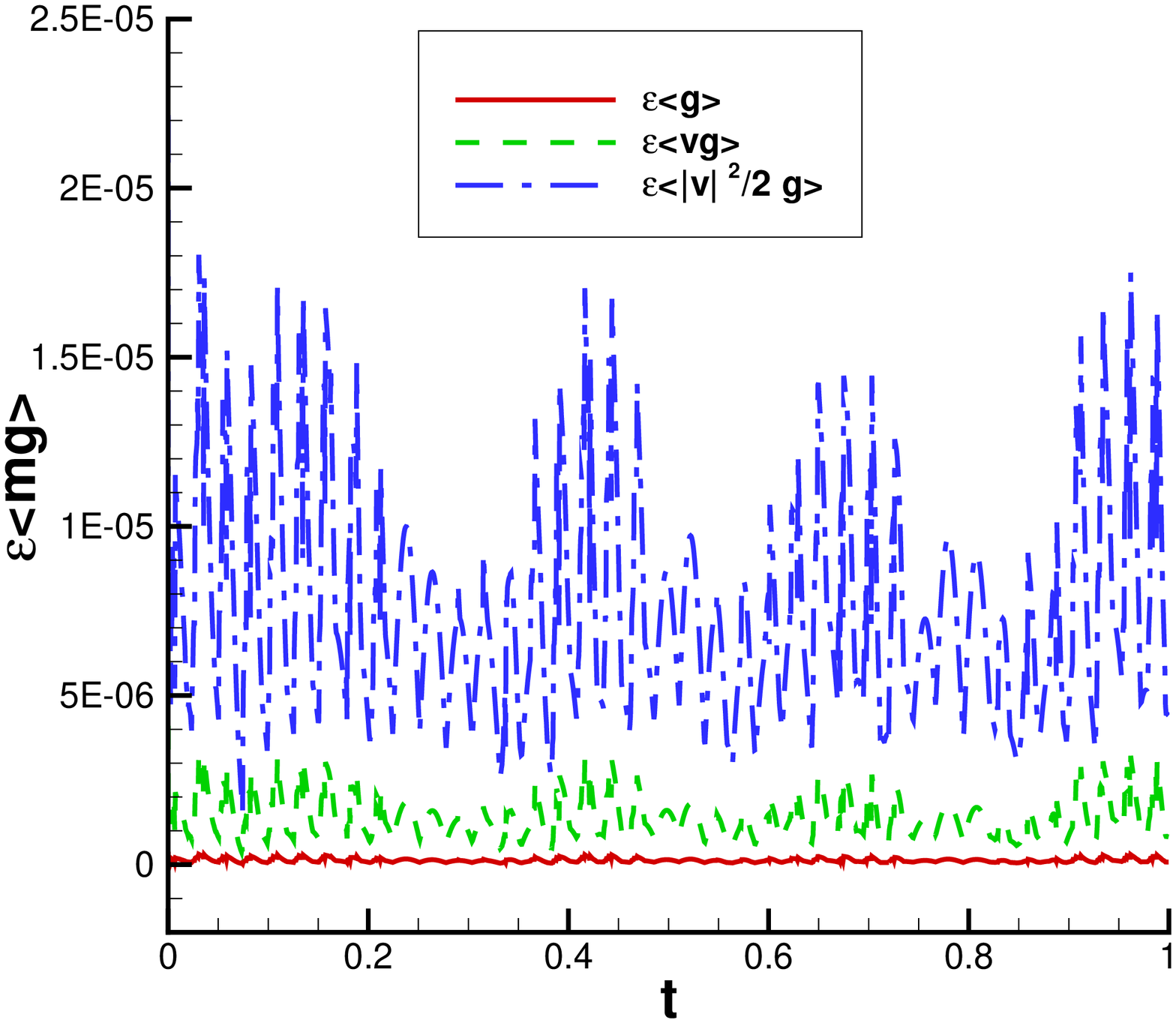},
\includegraphics[totalheight=2.5in]{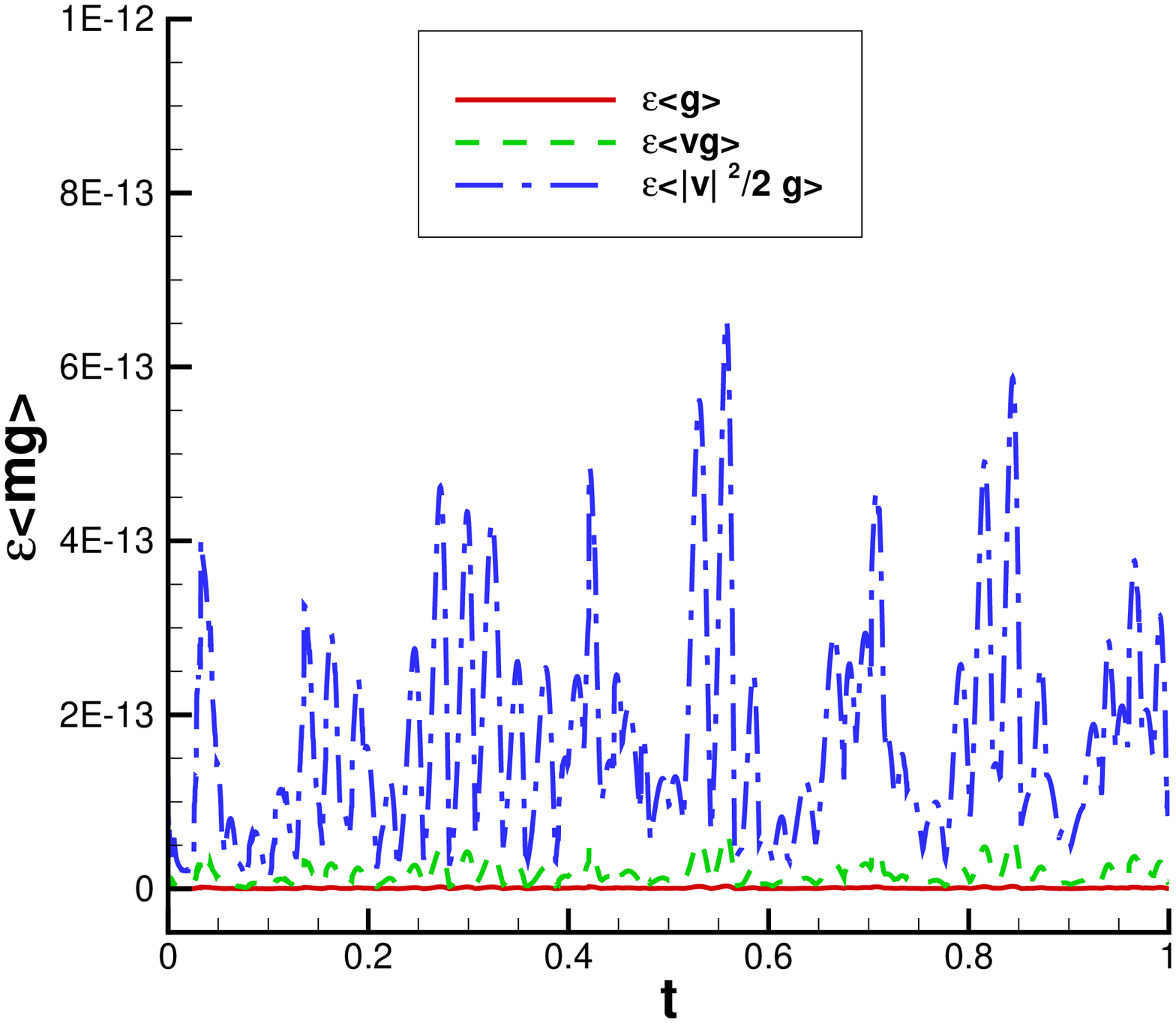}
\caption{$\eps\langle mg \rangle$ for the Shu-Osher problem (\ref{so}) with NDG3. Measured as the
maximal value over $x$. $N_x=200$ and $N_v=100$ on the spatial domain $[-12, 12]$. Top: $\eps=1$;
Bottom: $\eps=10^{-6}$. Left: $V_c=10$; Right: $V_c=20$. With TVB limiter and $M_{tvb}=20$. }
\label{fig9}
\end{figure}

\begin{figure}[ht]
\centering
\subfigure[$\eps=1$]{
\includegraphics[totalheight=2.5in]{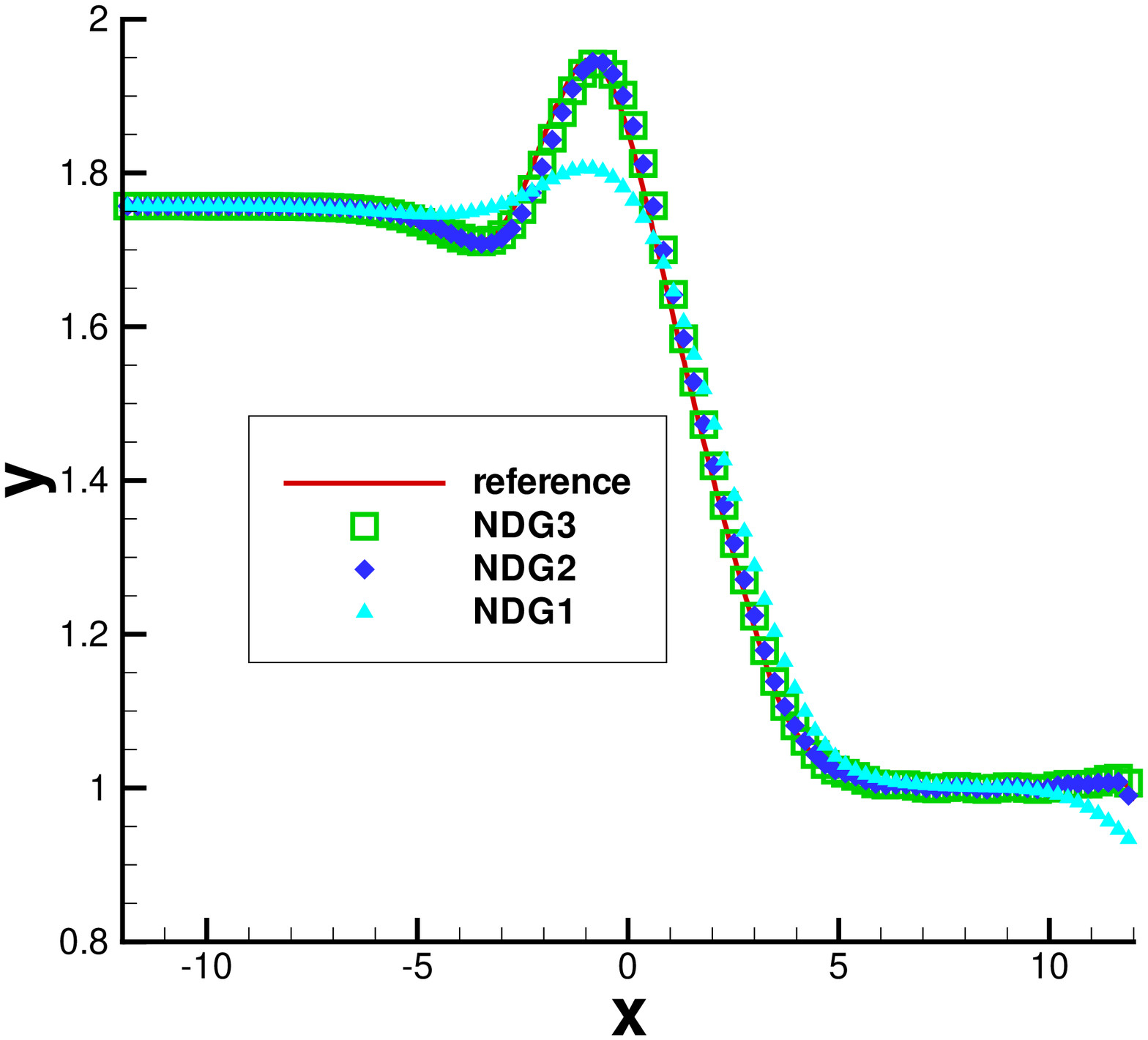}},
\subfigure[$\eps=10^{-1}$]{
\includegraphics[totalheight=2.5in]{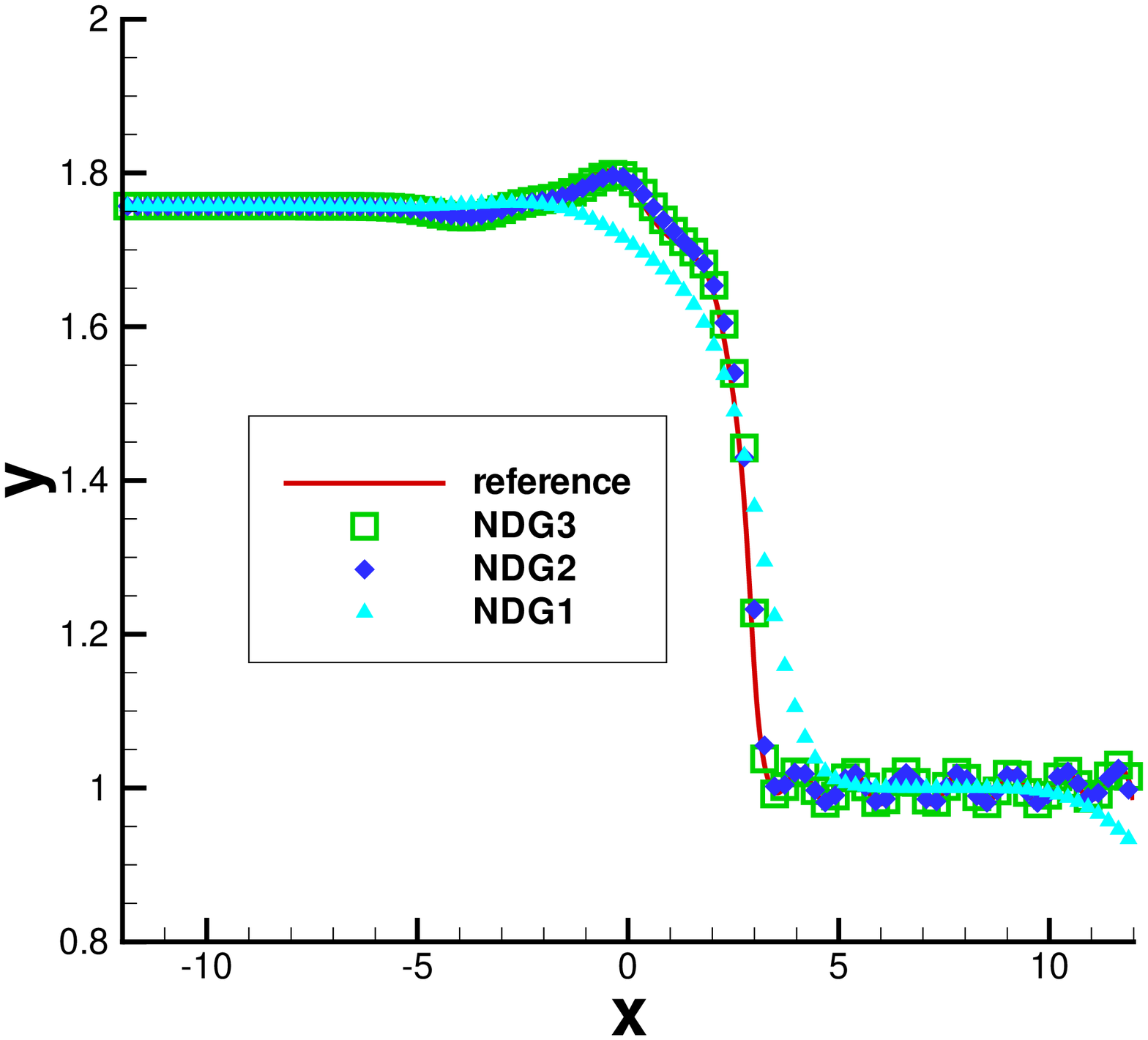}}\\
\subfigure[$\eps=10^{-2}$]{
\includegraphics[totalheight=2.5in]{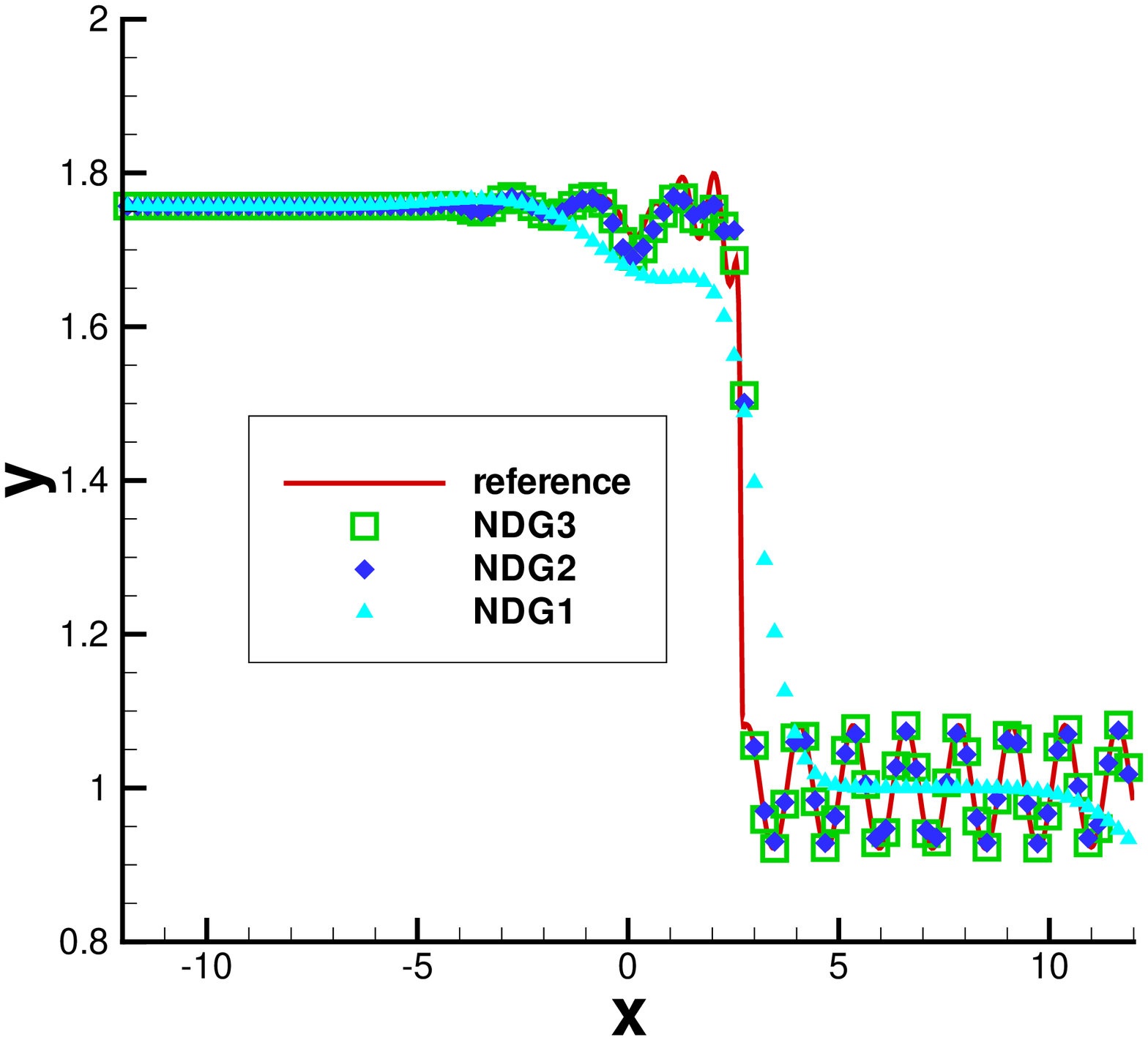}},
\subfigure[$\eps=10^{-6}$]{
\includegraphics[totalheight=2.5in]{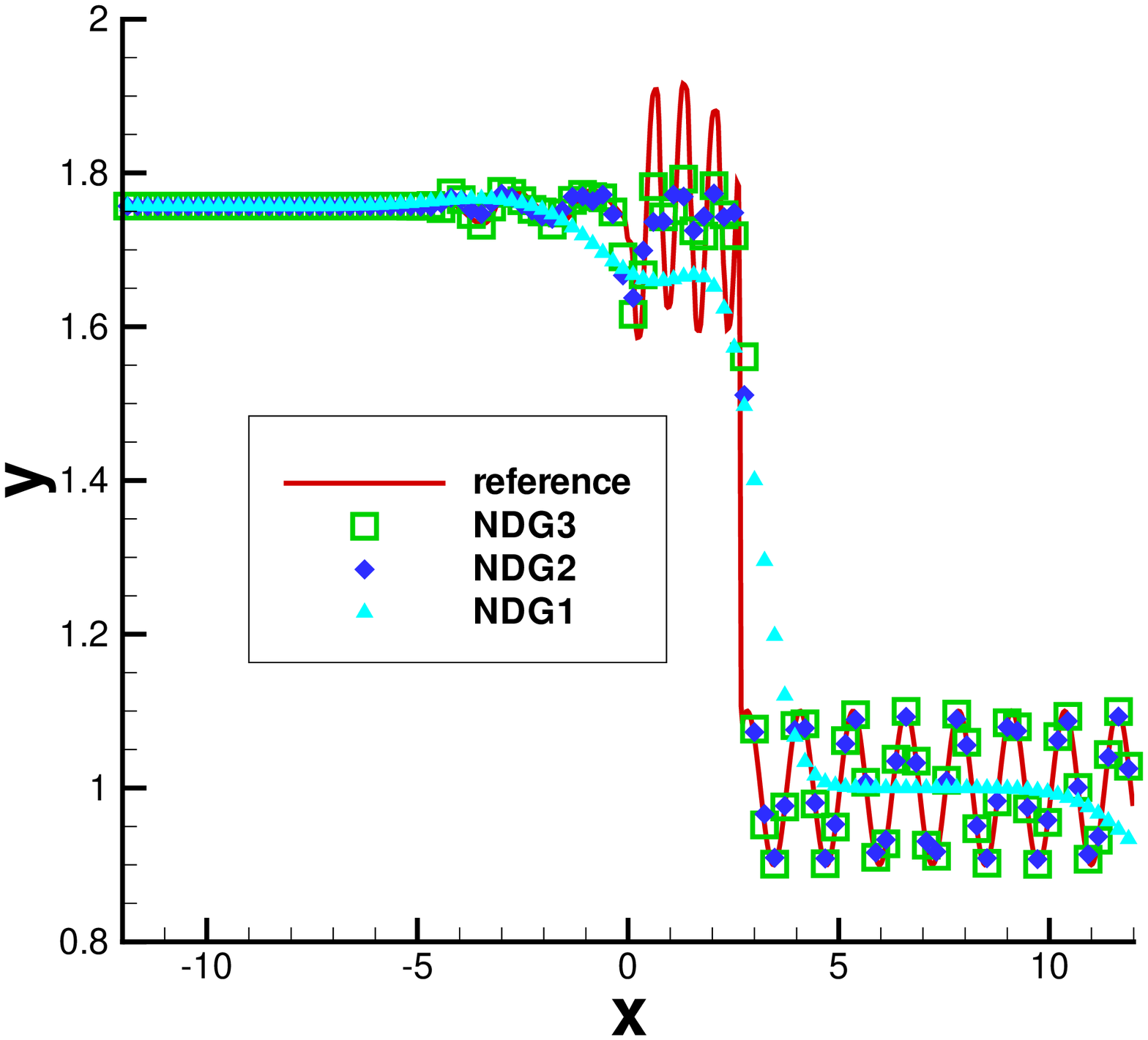}}\\
\caption{Numerical solution of density for the Shu-Osher problem (\ref{so}) at $t=1$
on the domain $[-12, 12]\times [-10, 10]$.
Solid line: reference solution of NDG3 with $N_x=800$ and $N_v=200$; Symbols with $N_x=100$
and $N_v=100$: square NDG3, diamond NDG2, delta NDG1. With TVB limiter and $M_{tvb}=1$. }
\label{fig16}
\end{figure}

\end{exa}

\begin{exa}
Finally we consider an example with a variable $\eps(x)$,
\begin{equation}
\eps(x)=\eps_0+\frac12\Big(\tanh(1- a_0 x)+\tanh(1+a_0 x)\Big).
\label{eps}
\end{equation}
The initial distribution function $f$ is far away from the Maxwellian, which is
\begin{eqnarray}
f(x,v,0)=\frac{\tilde \rho}{2( 2 \pi \tilde T)^{1/2}}\left[\exp\left(-\frac{|v-\tilde u|^2}{2 \tilde T}\right)
+\exp\left(-\frac{|v+\tilde u|^2}{2\tilde T}\right)\right],
\end{eqnarray}
with
\begin{equation}
\tilde \rho(x)=1+0.875 \sin(w x), \quad \tilde T(x)=0.5+0.4 \sin(\omega x), \quad \tilde u(x)=0.75,
\end{equation}
on the spatial domain $x\in [-L, L]$, where $\omega=\pi/L$ and $L=0.5$. From the definition \eqref{U-vector}, the initial macroscopic variables are
\begin{equation}
\rho(x,0)=\tilde \rho(x), \quad u(x,0)=0, \quad T(x,0)=\tilde T(x)+ \tilde u(x)^2,
\end{equation}
and the initial Maxwellian distribution is
\begin{equation}
M_U(x,v,0)=\frac{\rho(x,0)}{(2\pi T(x,0))^{1/2}}\exp\left(-\frac{|v-u(x,0)|^2}{2T(x,0)}\right).
\end{equation}
Periodic boundary conditions are used for both $U$ and $g$ in the $x$ direction. This example has a similar spirit as the one in \cite{filbet2010class}. The velocity domain is taken to be $\Omega_v=[-10, 10]$.

In Fig. \ref{fig10}, we first show the distribution function $f$, the density $\rho$, the mean velocity $u$ and the temperature $T$ at time $t=0.1, 0.3, 0.45$, with a wide peak of $\eps(x)$ in \eqref{eps} where $a_0=11$ and $\eps_0=10^{-6}$ (Fig. \ref{fig14} dashed line). $N_x=40$ and $N_v=100$ are used. Discontinuities can be observed in the solutions, while overall the solution structures are simple.
For the three methods with different accuracy,
NDG2 is close to NDG3, with both matching the reference solutions of NDG3 with $N_x=200$ and $N_v=200$ very well, and they are observed to perform much better than NDG1.
%
%
For this case, we also show the $L^1$ errors and orders, which are computed as in (\ref{l1error}) and (\ref{l1order}), for NDG3 at a short time $t=0.001$ in Table \ref{tab10}. At least the 2nd order accuracy can be observed for both $\rho$ and $g$ in this mixed regime problem.

Fig. \ref{fig11} has a similar setting to those of Fig. \ref{fig10}, but with a narrower peak of $\eps(x)$ where $a_0=40$ (Fig. \ref{fig14} solid line). Discontinuities can be observed in the solutions.
We can see that on the relatively coarser mesh with $N_x=40$ and $N_v=100$, the results from NDG2 and NDG3 are comparable to the solid lines of NDG3 on a finer mesh with  $N_x=200$ and $N_v=200$. Again, NDG2 and NDG3 perform much better than NDG1.



For the case of $a_0=40$, we also take a bigger $\eps_0=10^{-3}$ and compare the results of NDG3 with the one obtained by explicitly solving the BGK model \eqref{bgk} with NDG1 spatial discretization and the first order Euler forward time discretization on the mesh of $N_x=1000$ and $N_v=100$. The results in Fig. \ref{fig13} match each other very well.

\begin{table}[!h]
\centering
\caption{$L^1$ errors and orders for $\rho$ and $g$ of the mixed regime problem with $\eps(x)$ in \eqref{eps}. $a_0=11$ and $\eps_0=10^{-6}$. NDG3. $t=0.001$. $V_c=20$. }
\vspace{0.2cm}
  \begin{tabular}{|c|c|c|c|c|c|}
    \hline
    N  &  $L^1$ error of $\rho$ & order   & $L^1$ error of $g$  & order \\\hline
    20 &     5.12E-05 &       --&     1.44E-03 &       --  \\  \hline
    40 &     7.30E-06 &     2.81&     2.62E-04 &     2.46  \\  \hline
    80 &     1.08E-06 &     2.76&     3.65E-05 &     2.84  \\  \hline
   160 &     2.06E-07 &     2.39&     9.03E-06 &     2.01  \\  \hline
  \end{tabular}
\label{tab10}
\end{table}

\begin{figure}
\centering
\includegraphics[totalheight=2.5in]{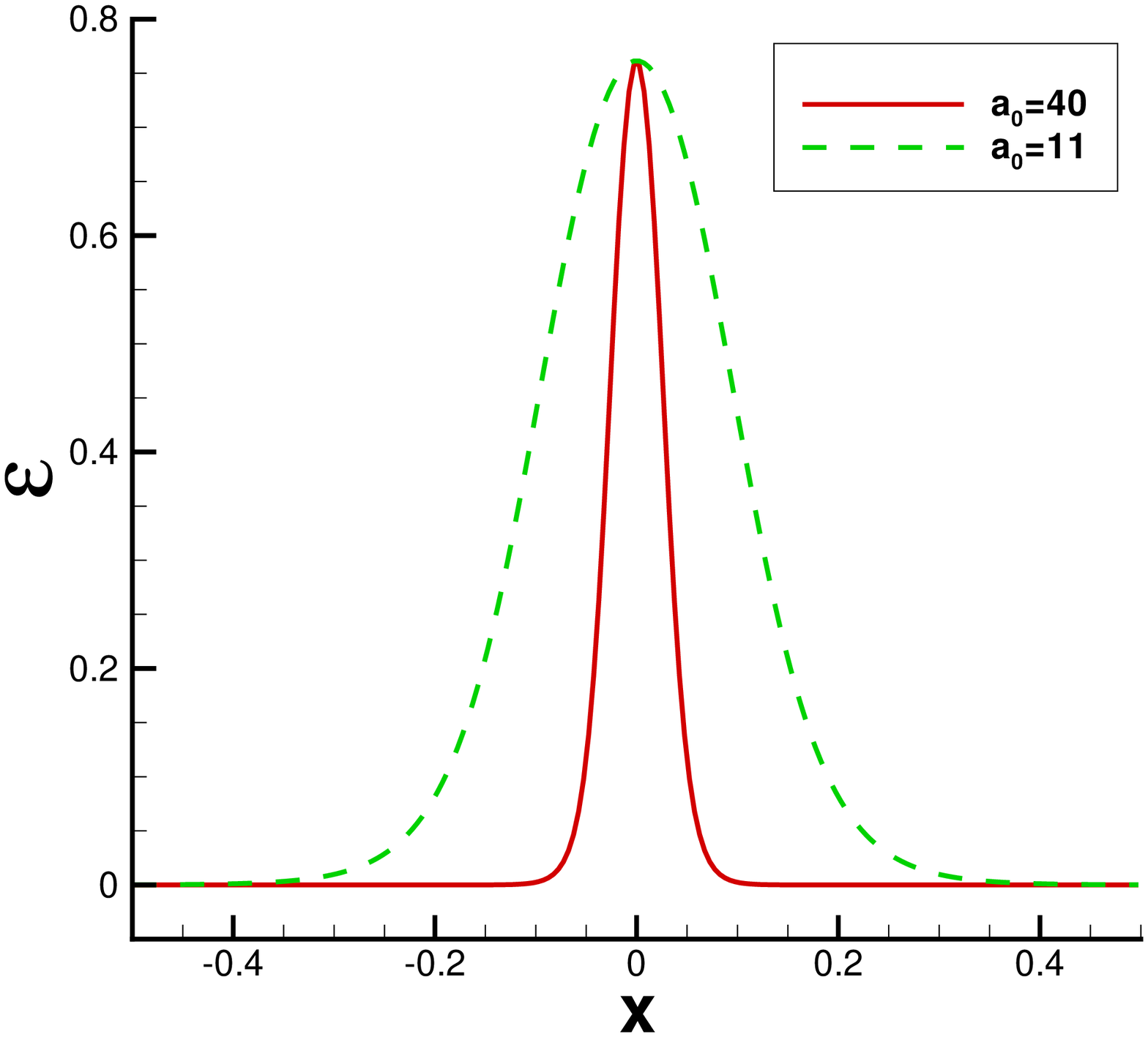}
\caption{Variable $\eps(x)$ in (\ref{eps}) with $\eps_0=10^{-6}$.}
\label{fig14}
\end{figure}

\begin{figure}[ht]
\centering
\includegraphics[totalheight=1.8in]{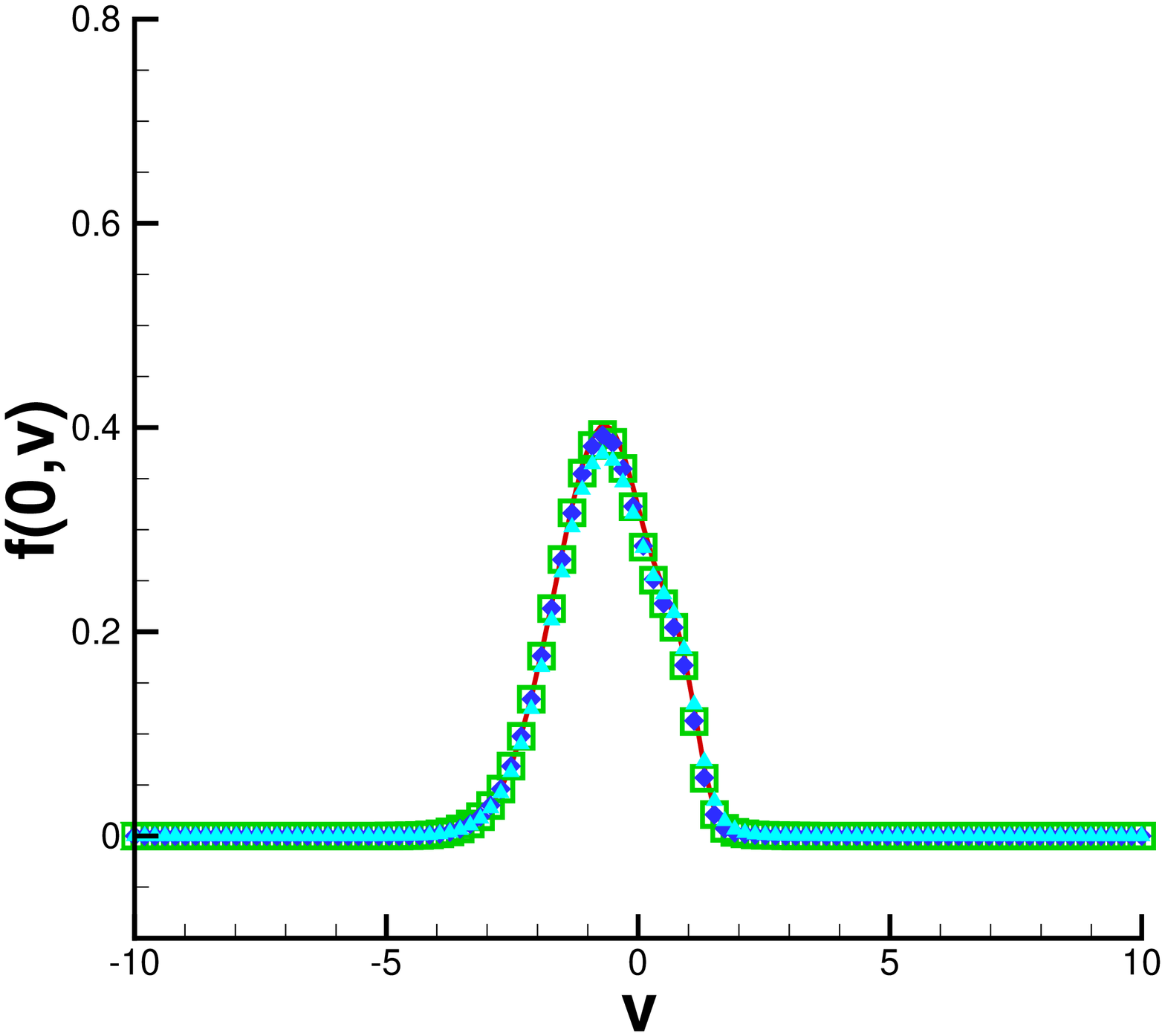},
\includegraphics[totalheight=1.8in]{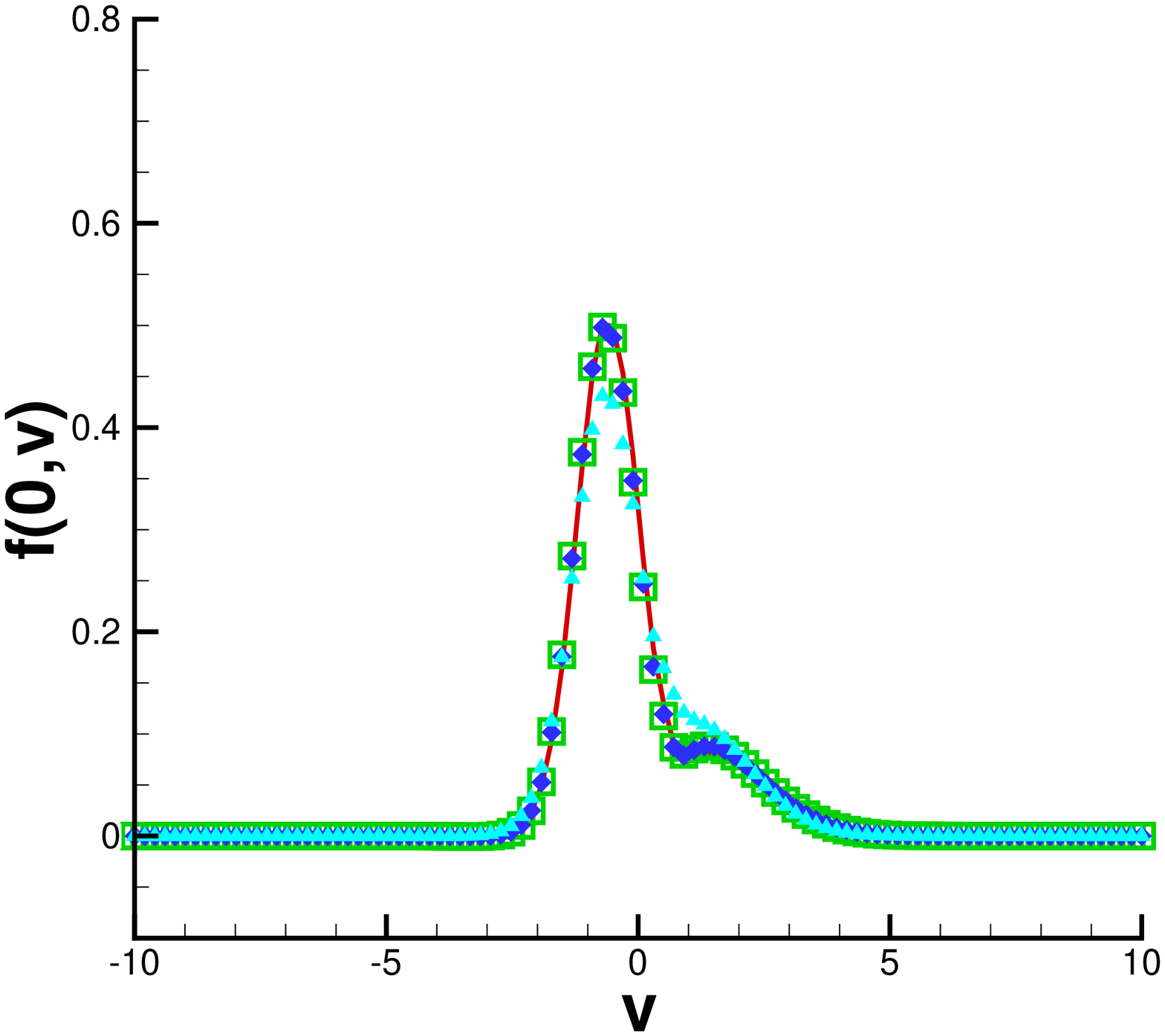},
\includegraphics[totalheight=1.8in]{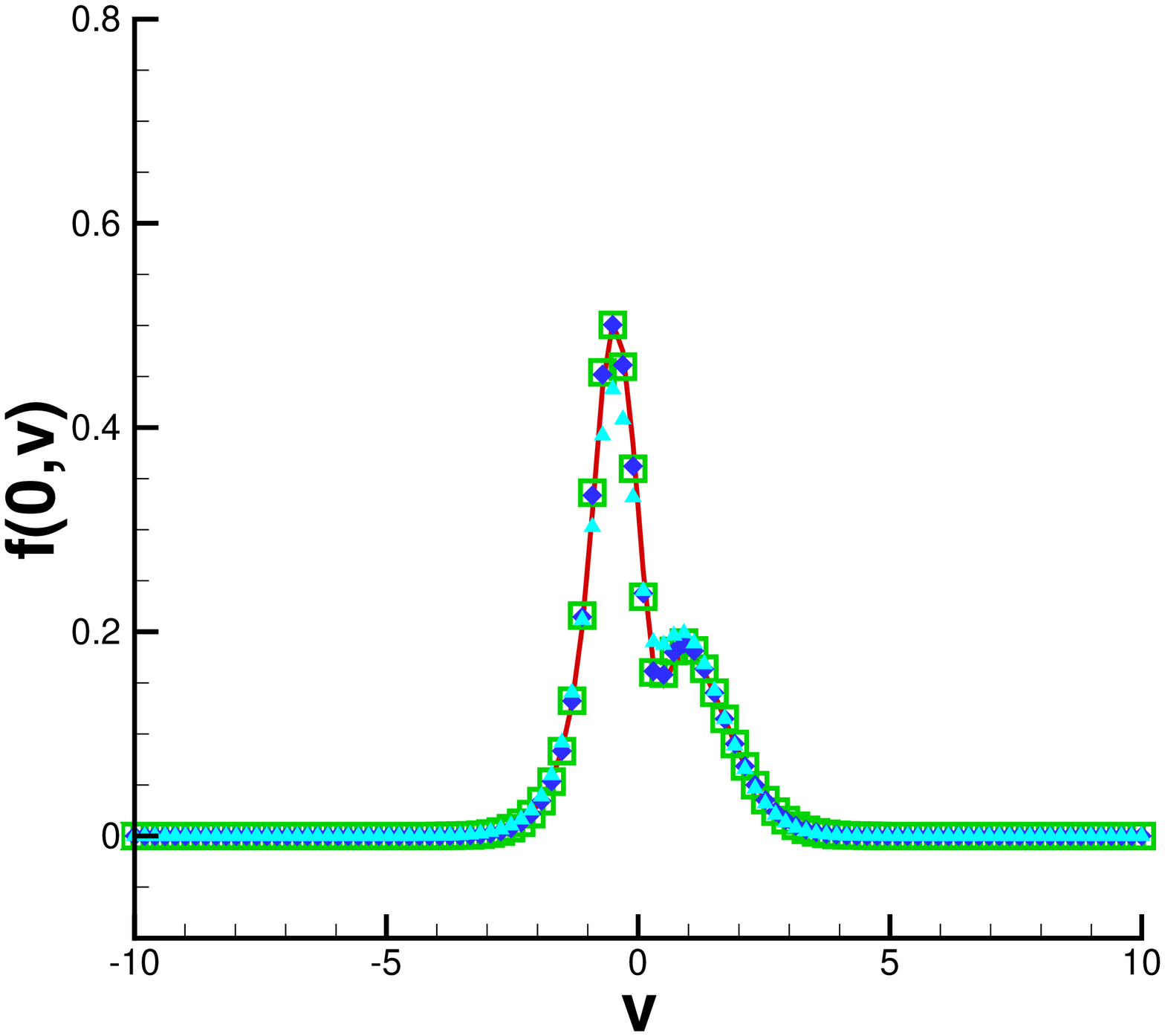} \\
\includegraphics[totalheight=1.8in]{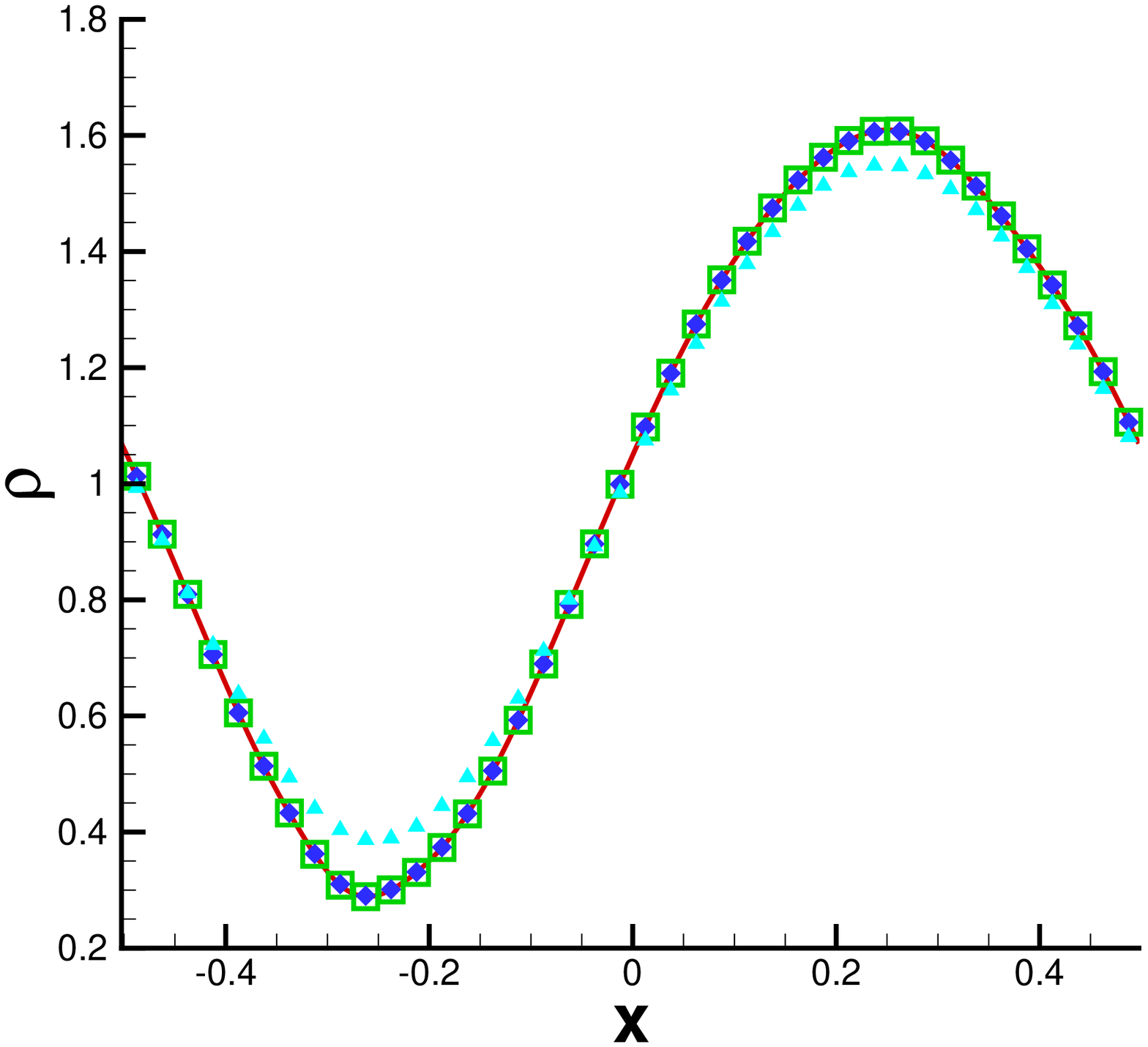},
\includegraphics[totalheight=1.8in]{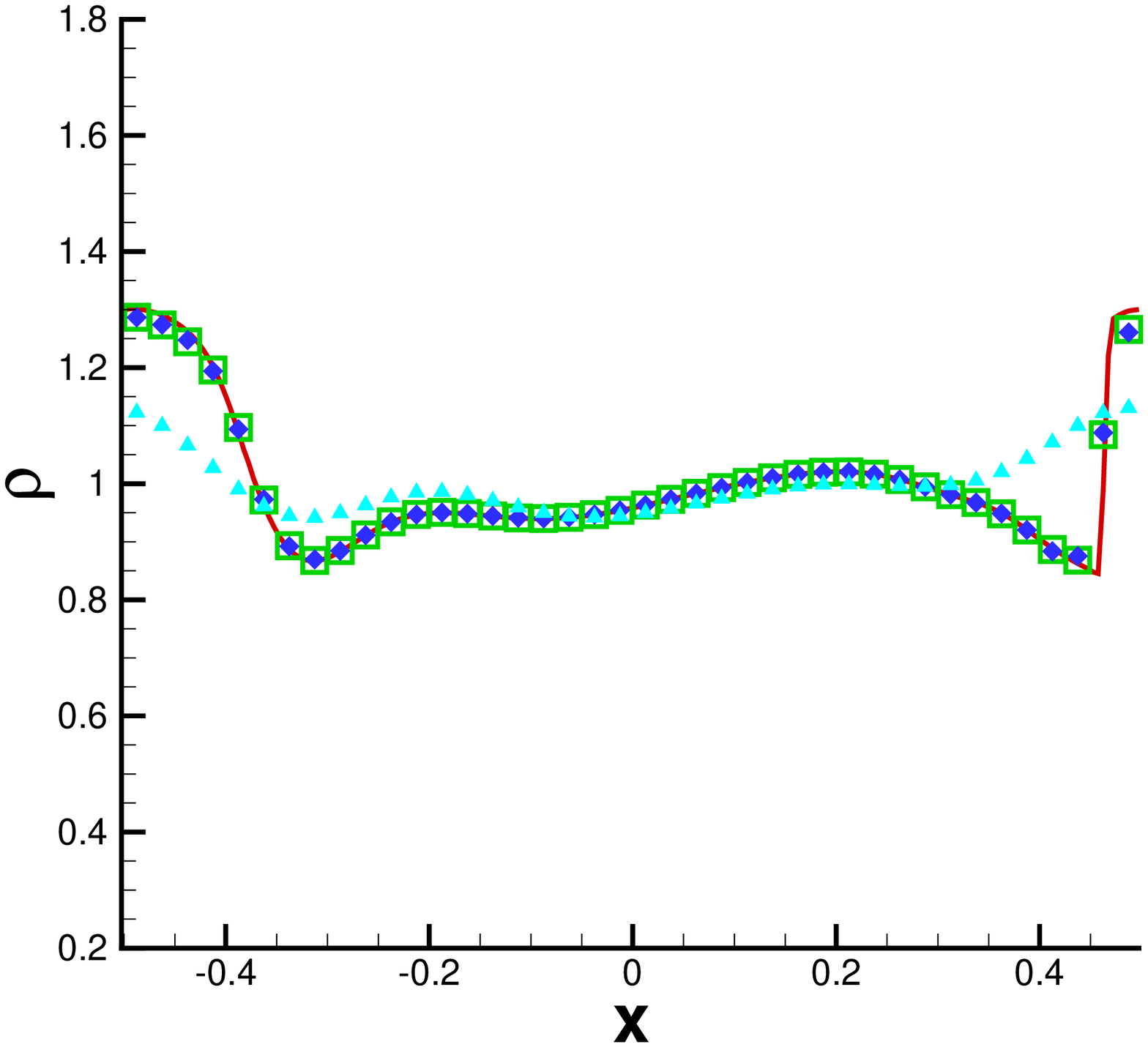},
\includegraphics[totalheight=1.8in]{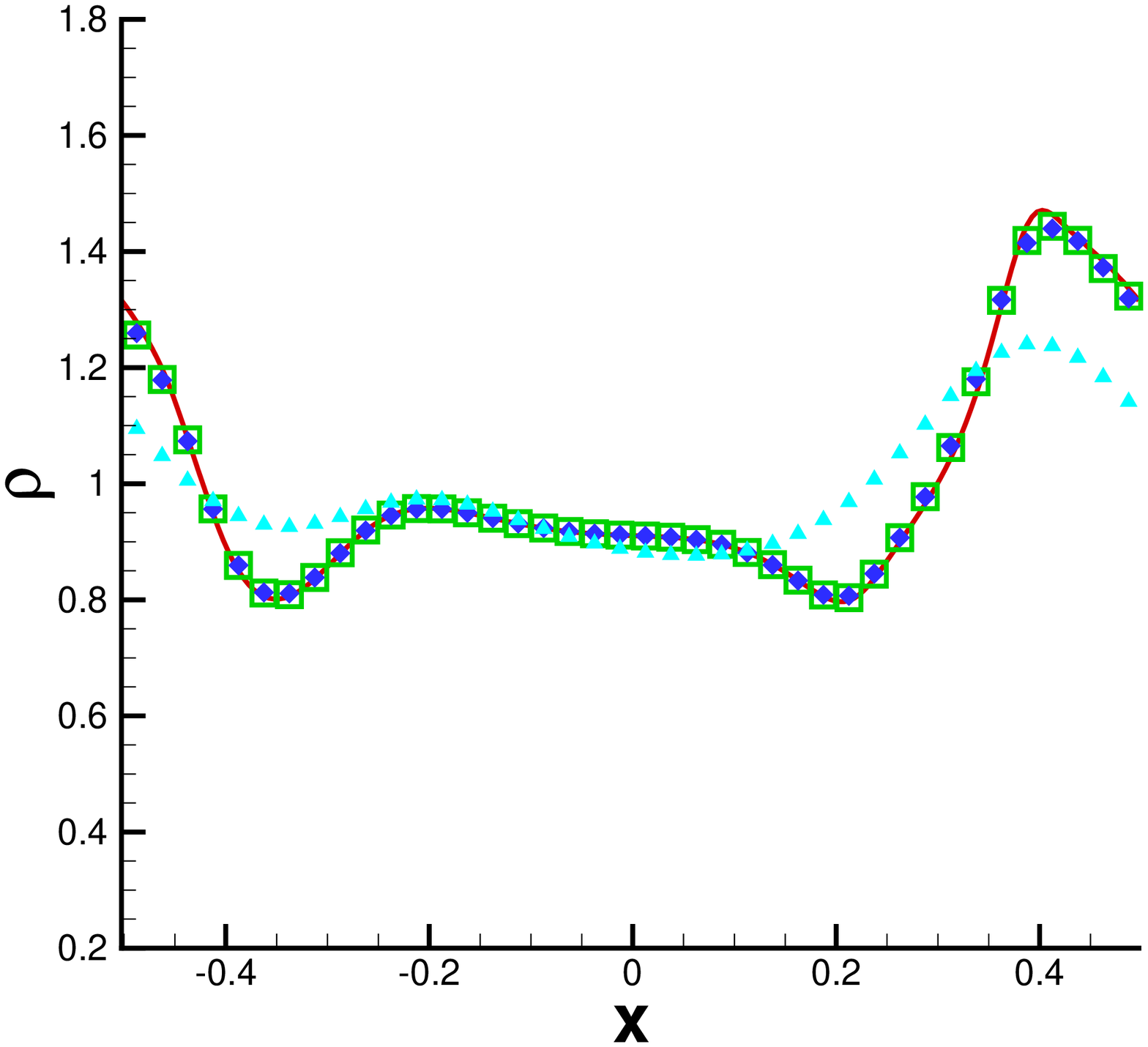} \\
\includegraphics[totalheight=1.8in]{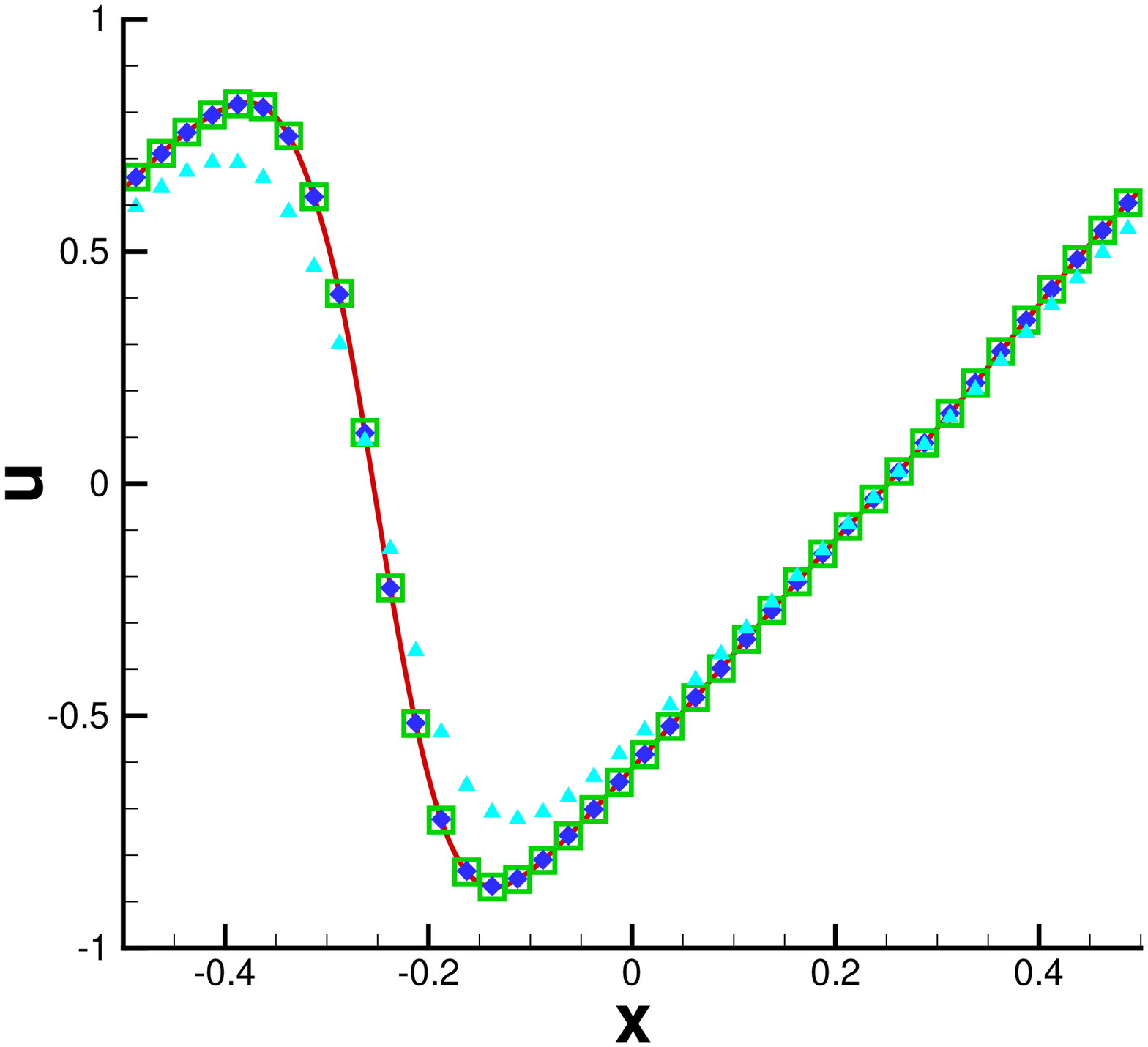},
\includegraphics[totalheight=1.8in]{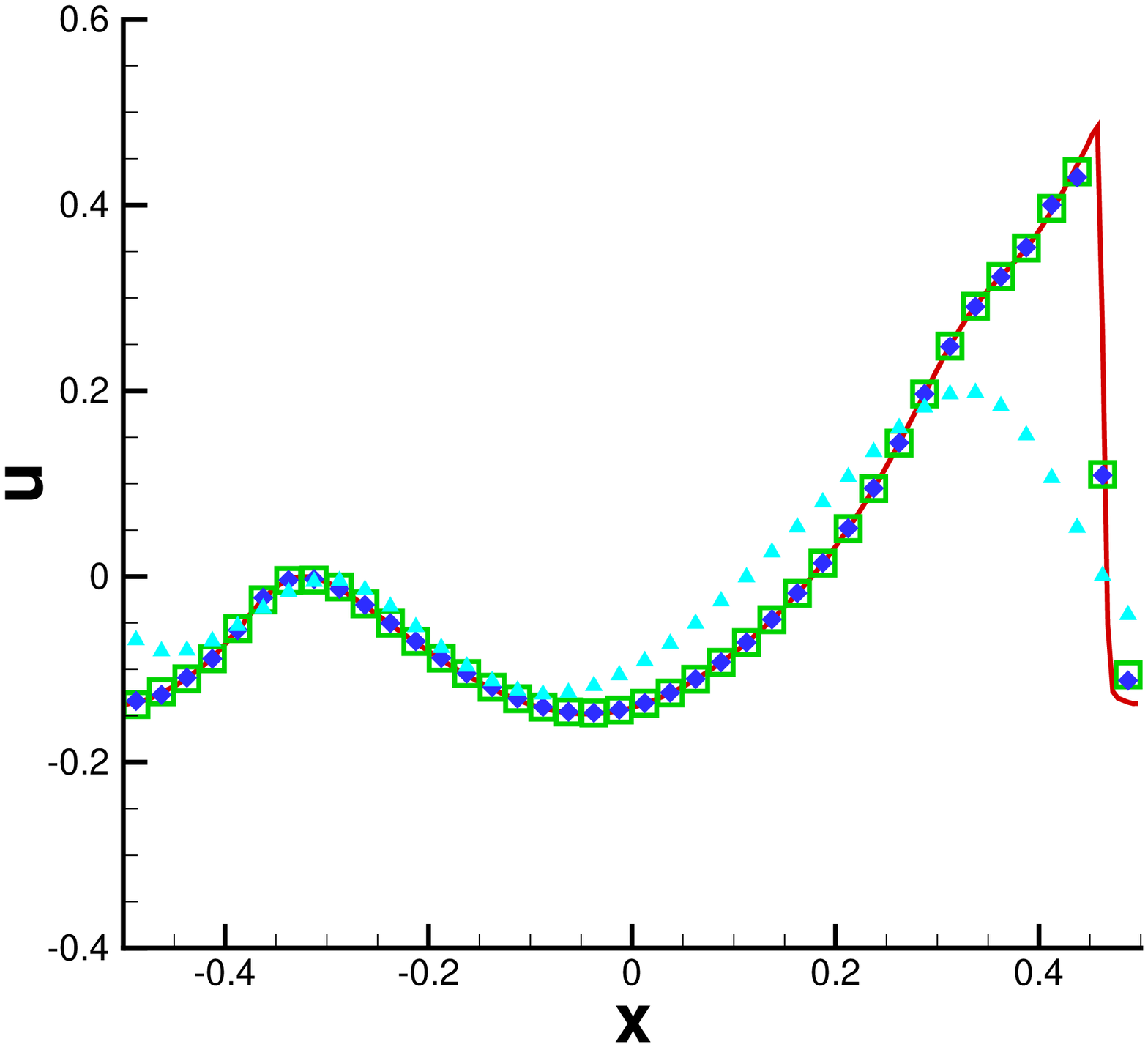},
\includegraphics[totalheight=1.8in]{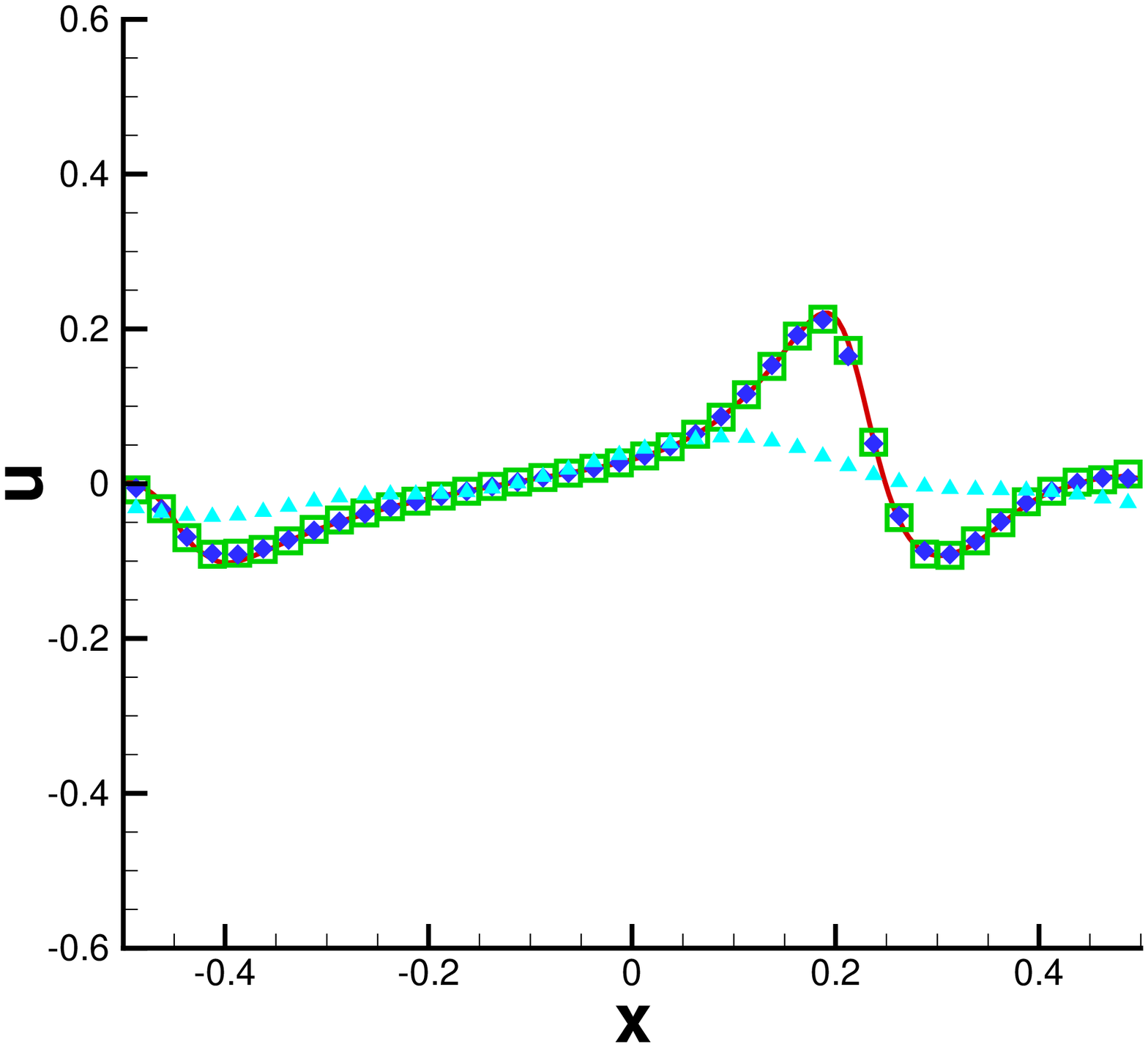} \\
\includegraphics[totalheight=1.8in]{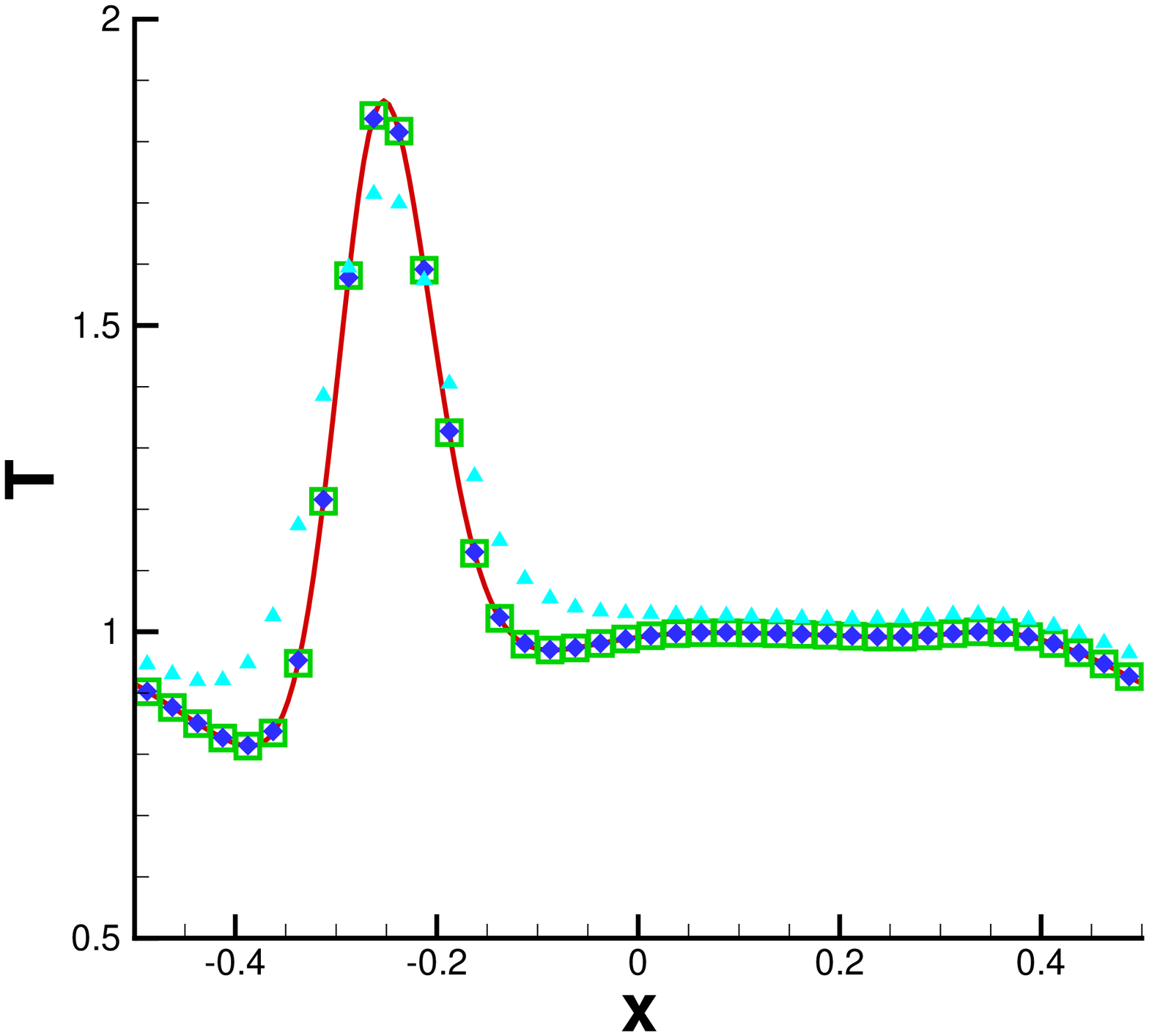},
\includegraphics[totalheight=1.8in]{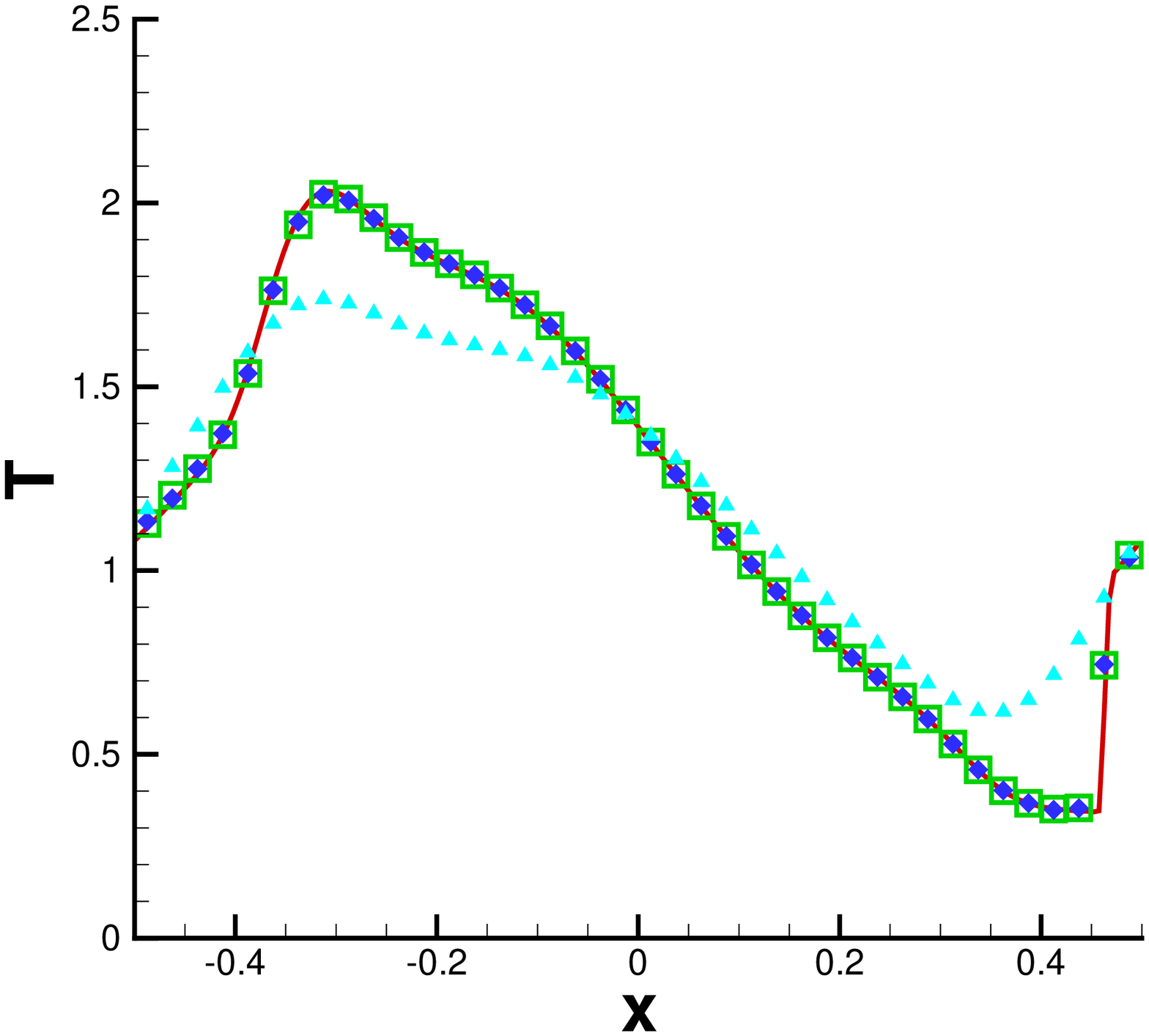},
\includegraphics[totalheight=1.8in]{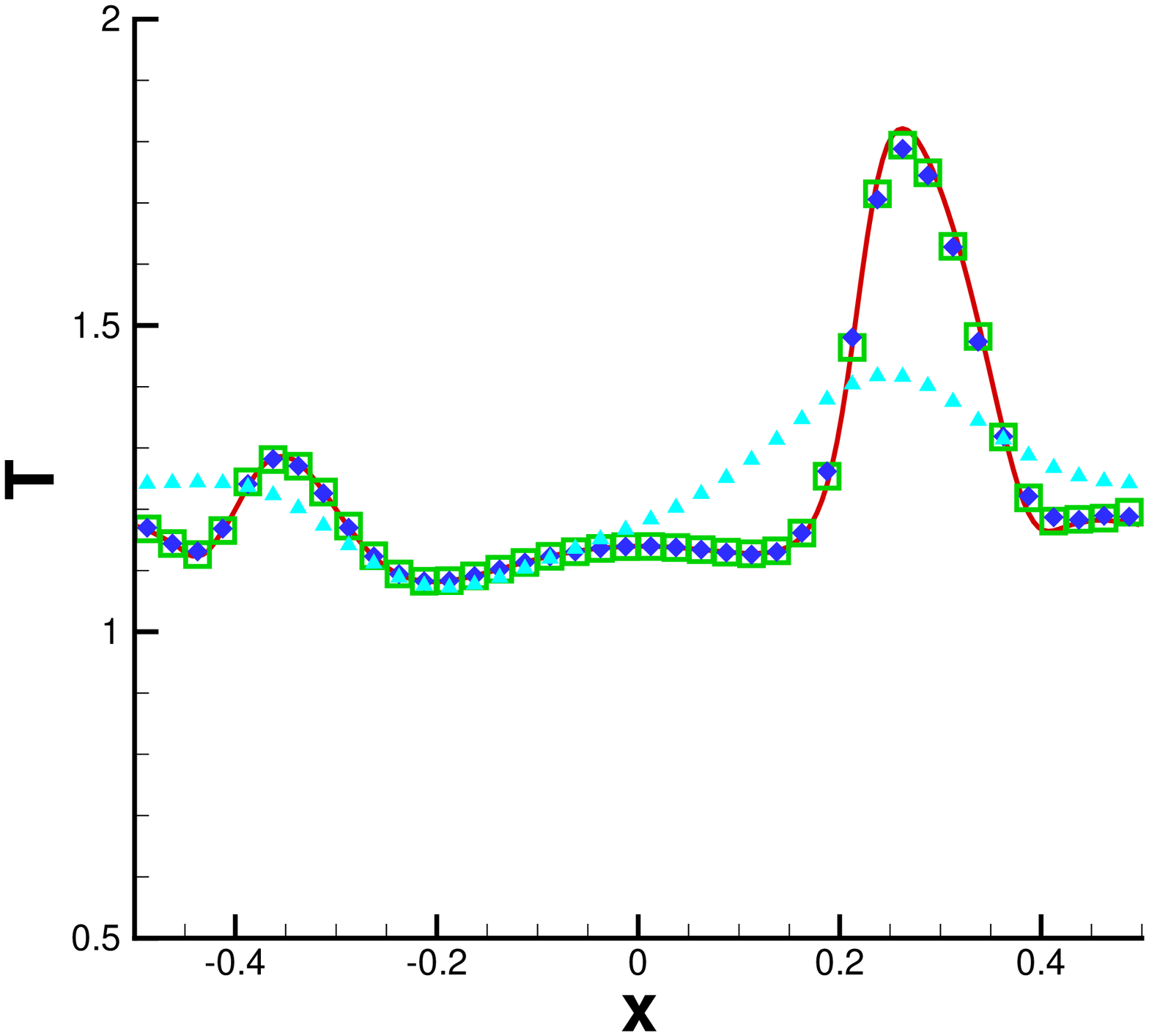}
\caption{Mixed regime problem with $\eps(x)$ in \eqref{eps} with $a_0=11$ and $\eps_0=10^{-6}$ on the domain $[-0.5,0.5]\times[-10,10]$. $N_x=40$ and $N_v=100$. Symbols: square NDG3, diamond NDG2, delta NDG1. Solid line: reference solution of NDG3 with $N_x=200$ and $N_v=200$. From left to right: time $t=0.1, 0.3, 0.45$. From top to bottom, the distribution function $f$ at $x=0$ along $v$ direction, the density $\rho$, the mean velocity $u$ and the temperature $T$. With TVB limiter and $M_{tvb}=20$.}
\label{fig10}
\end{figure}


\begin{figure}[ht]
\centering
\includegraphics[totalheight=1.8in]{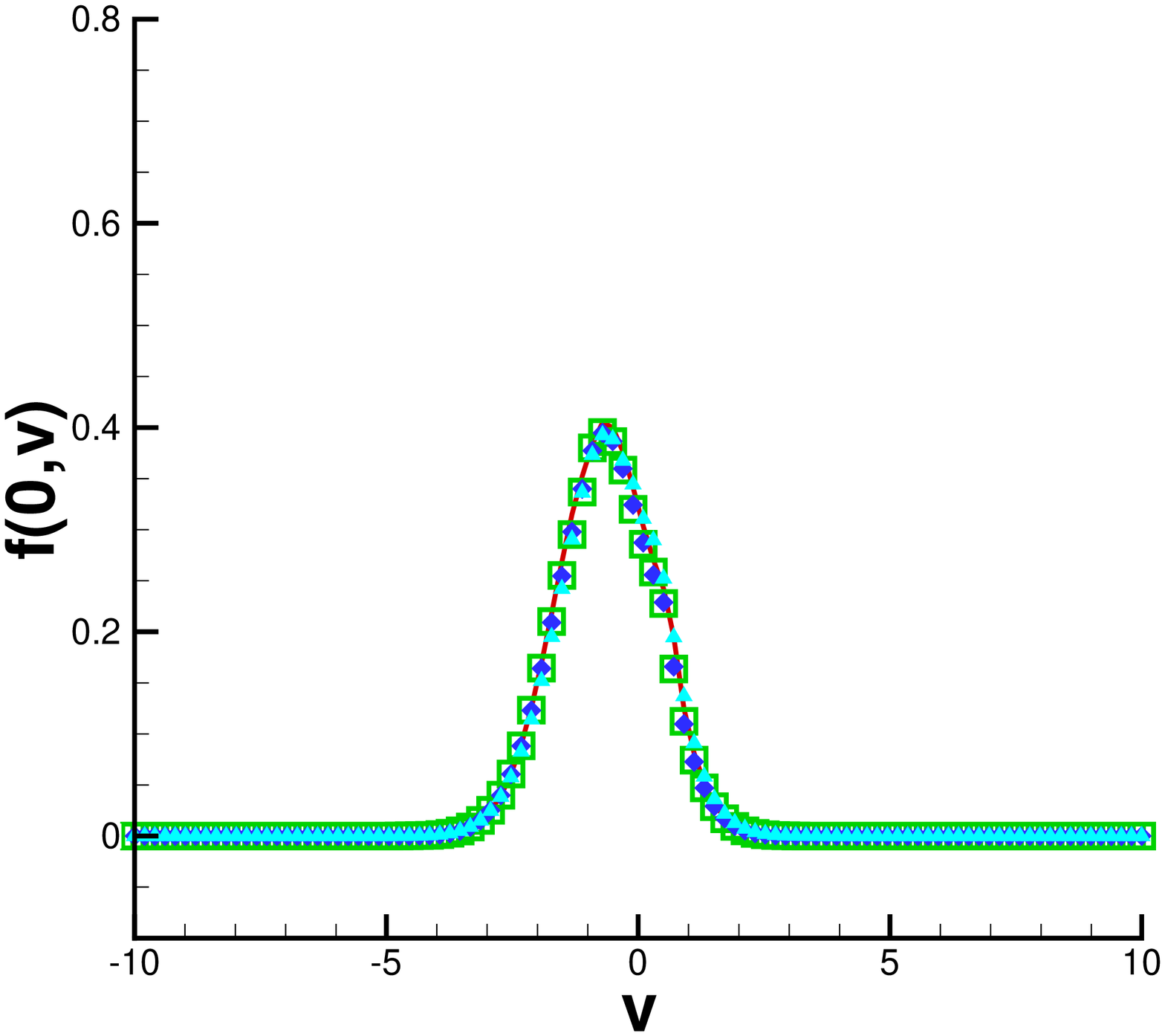},
\includegraphics[totalheight=1.8in]{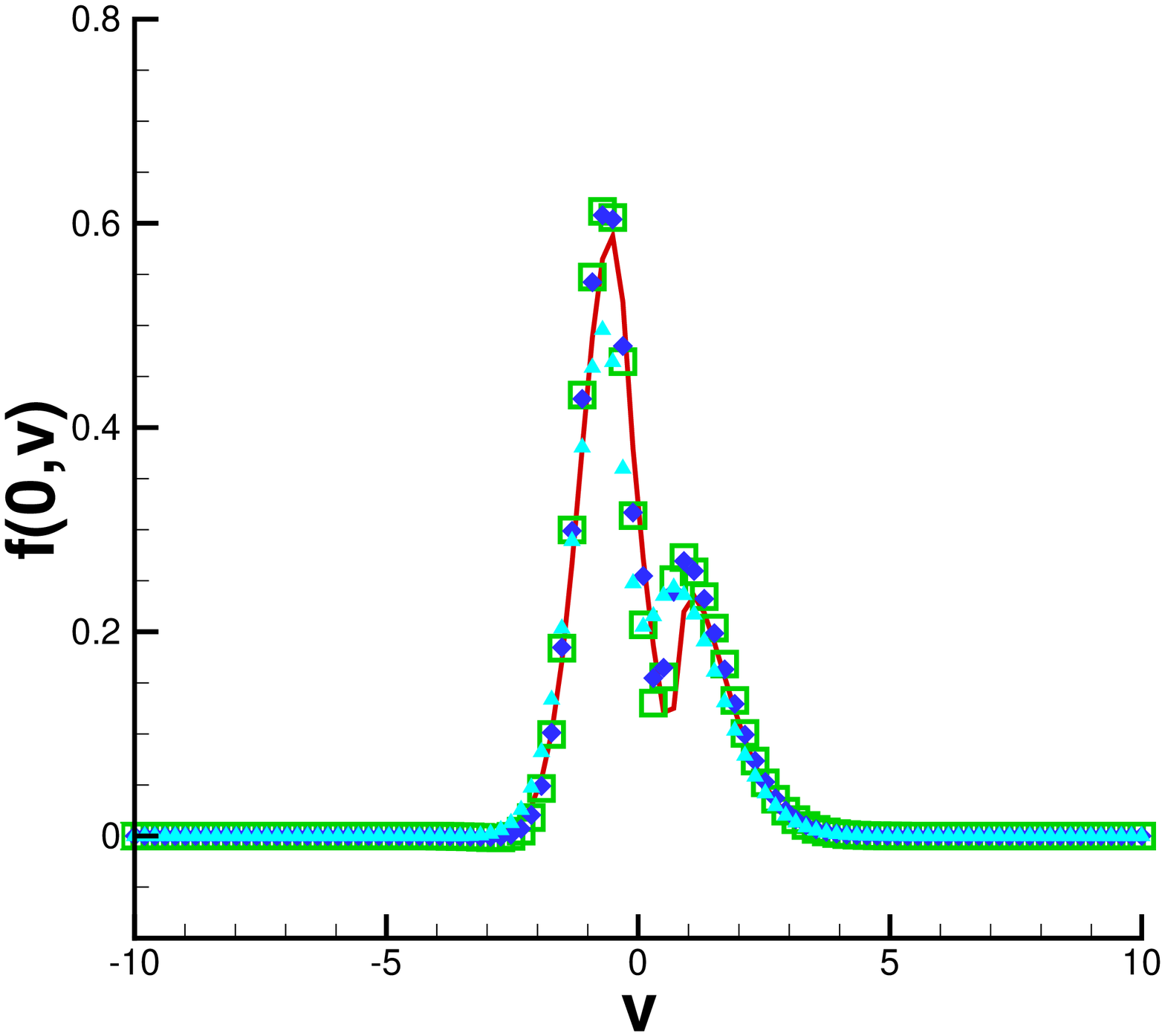},
\includegraphics[totalheight=1.8in]{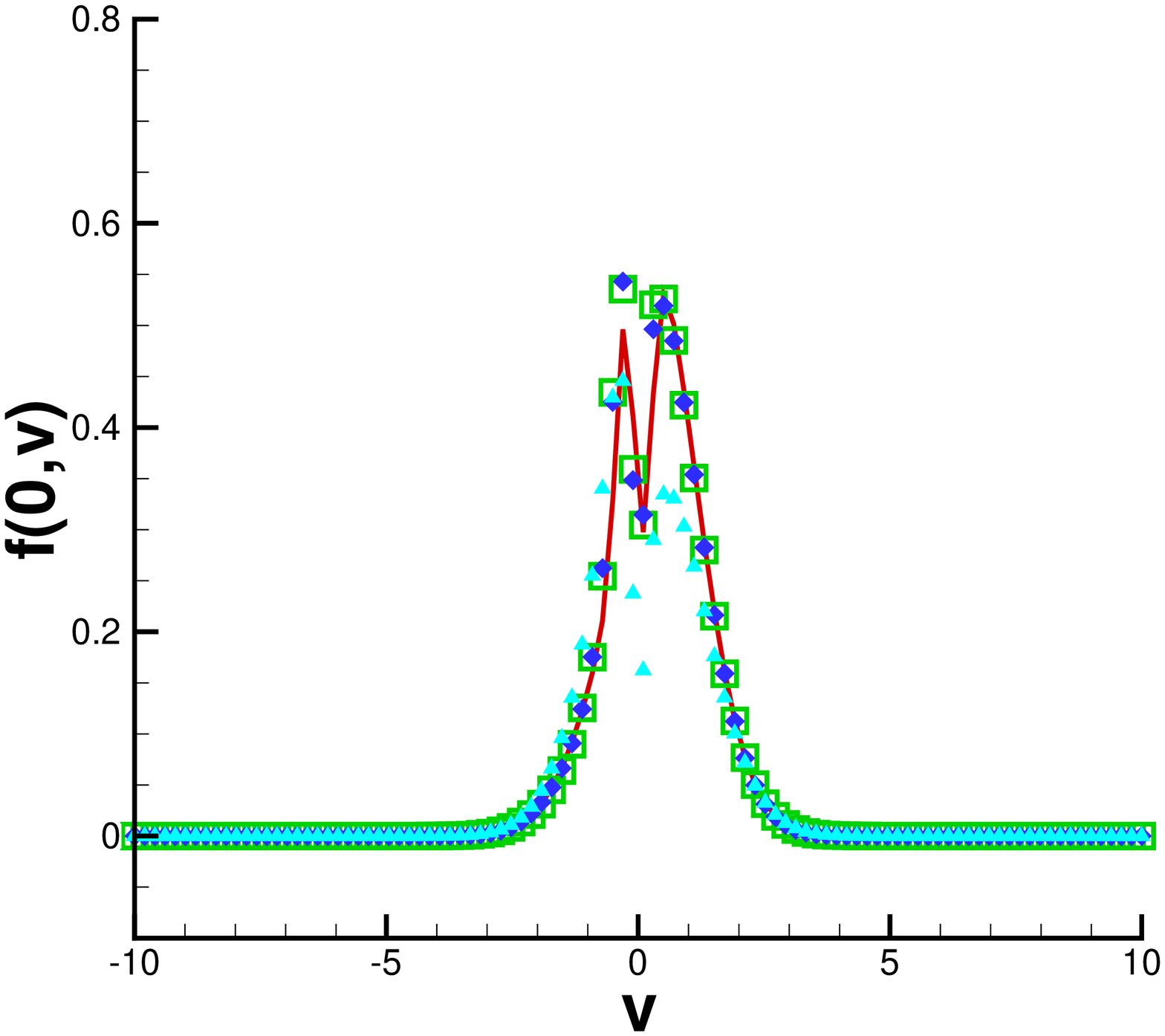} \\
\includegraphics[totalheight=1.8in]{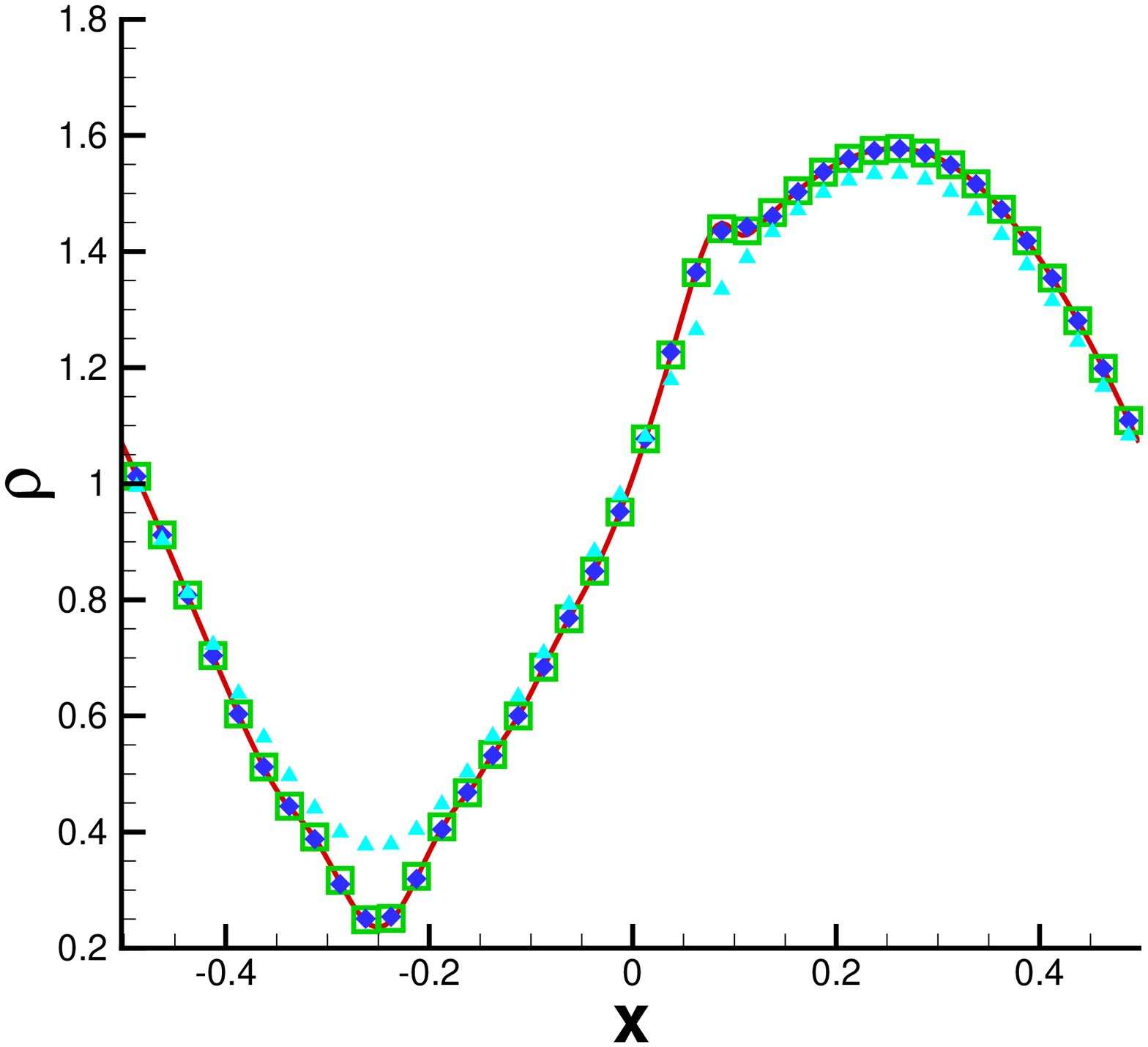},
\includegraphics[totalheight=1.8in]{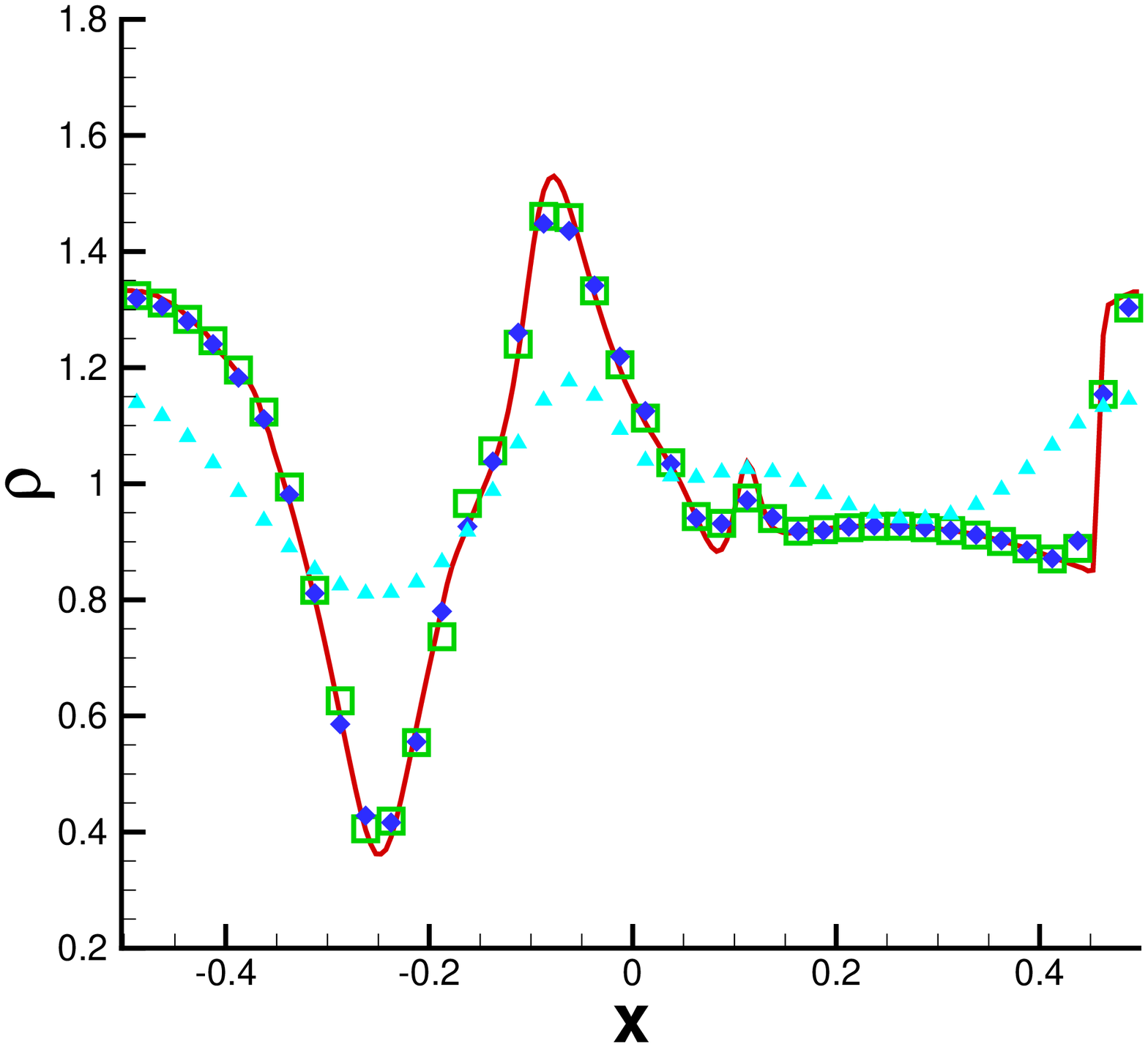},
\includegraphics[totalheight=1.8in]{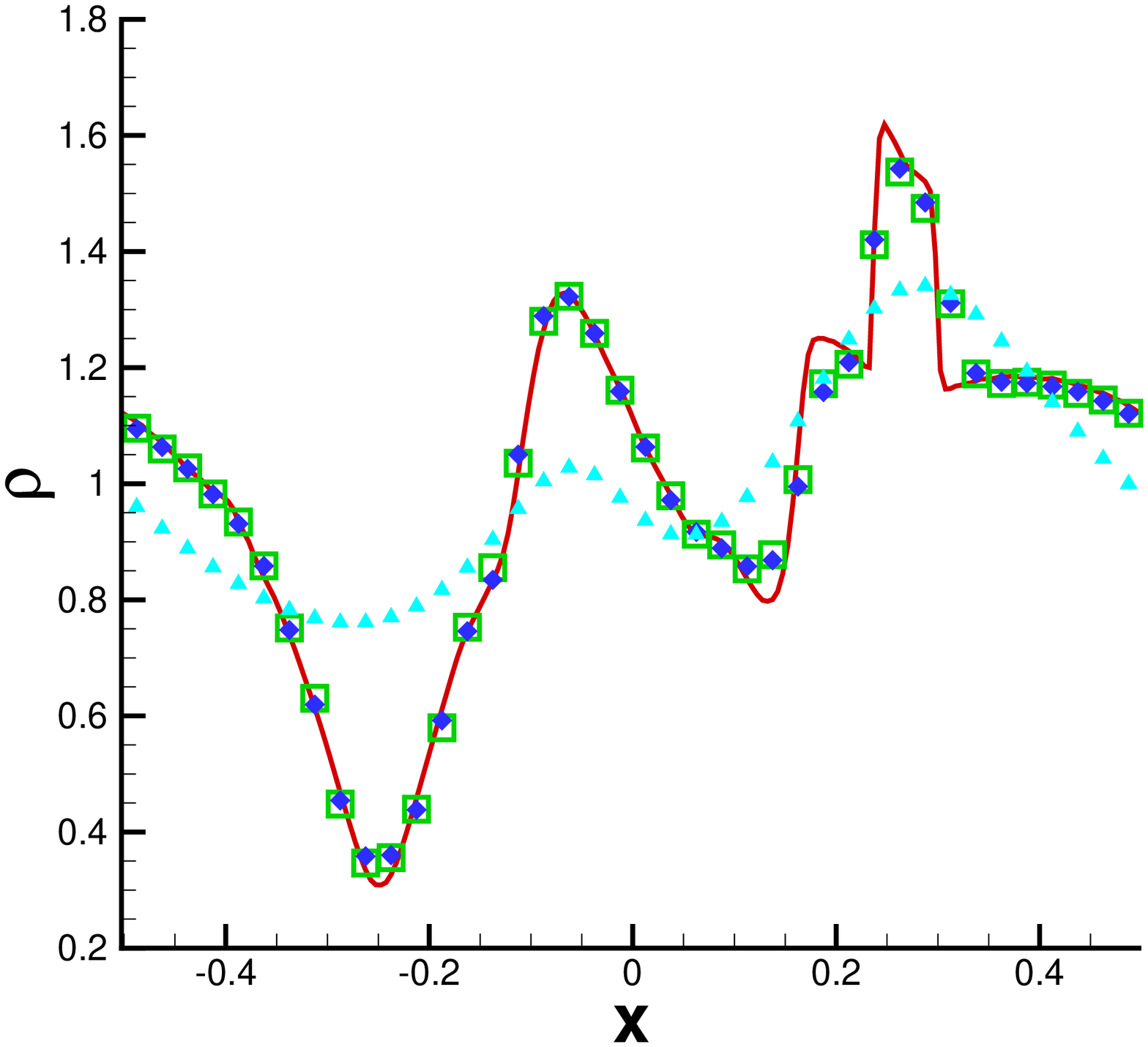} \\
\includegraphics[totalheight=1.8in]{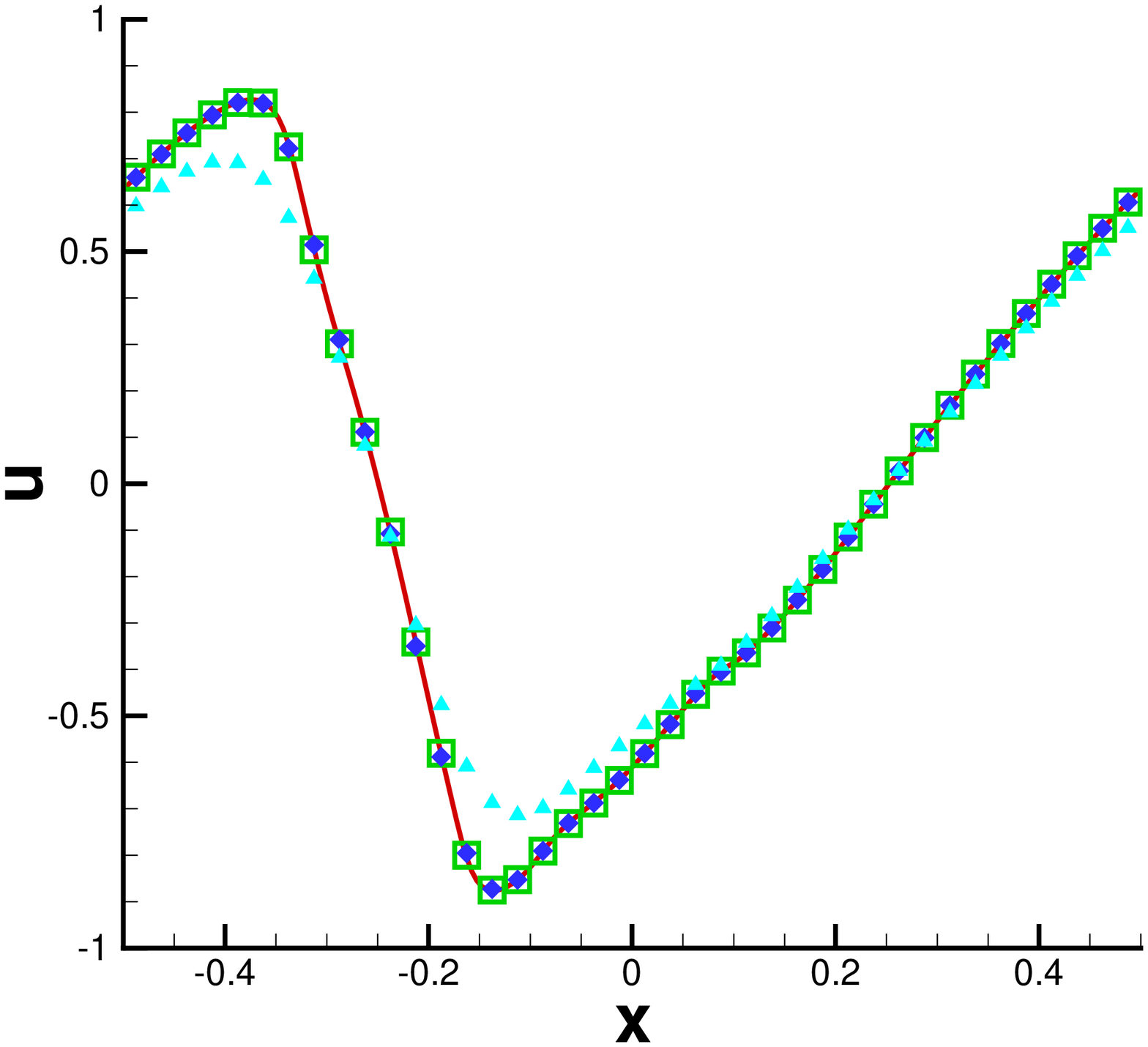},
\includegraphics[totalheight=1.8in]{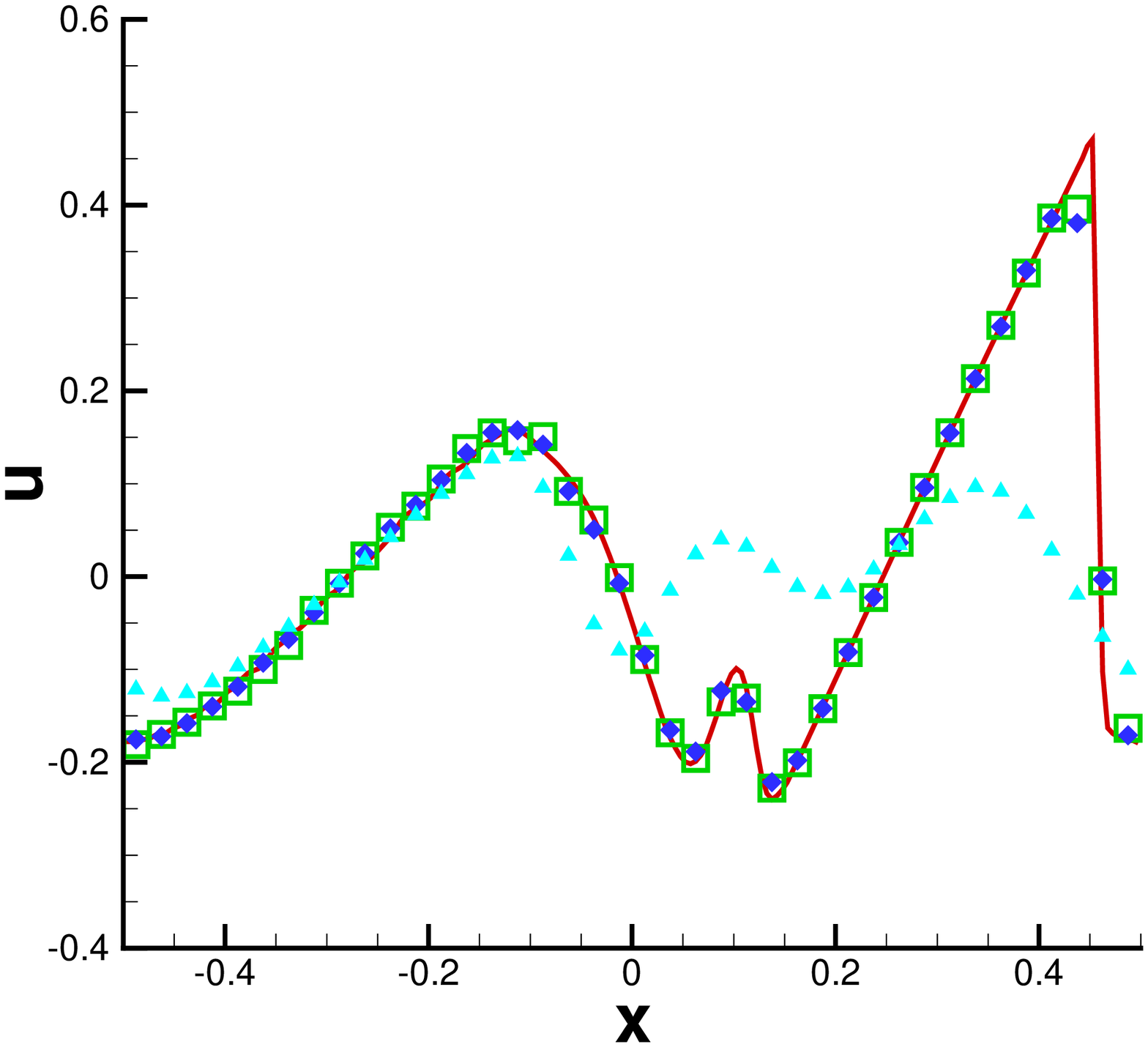},
\includegraphics[totalheight=1.8in]{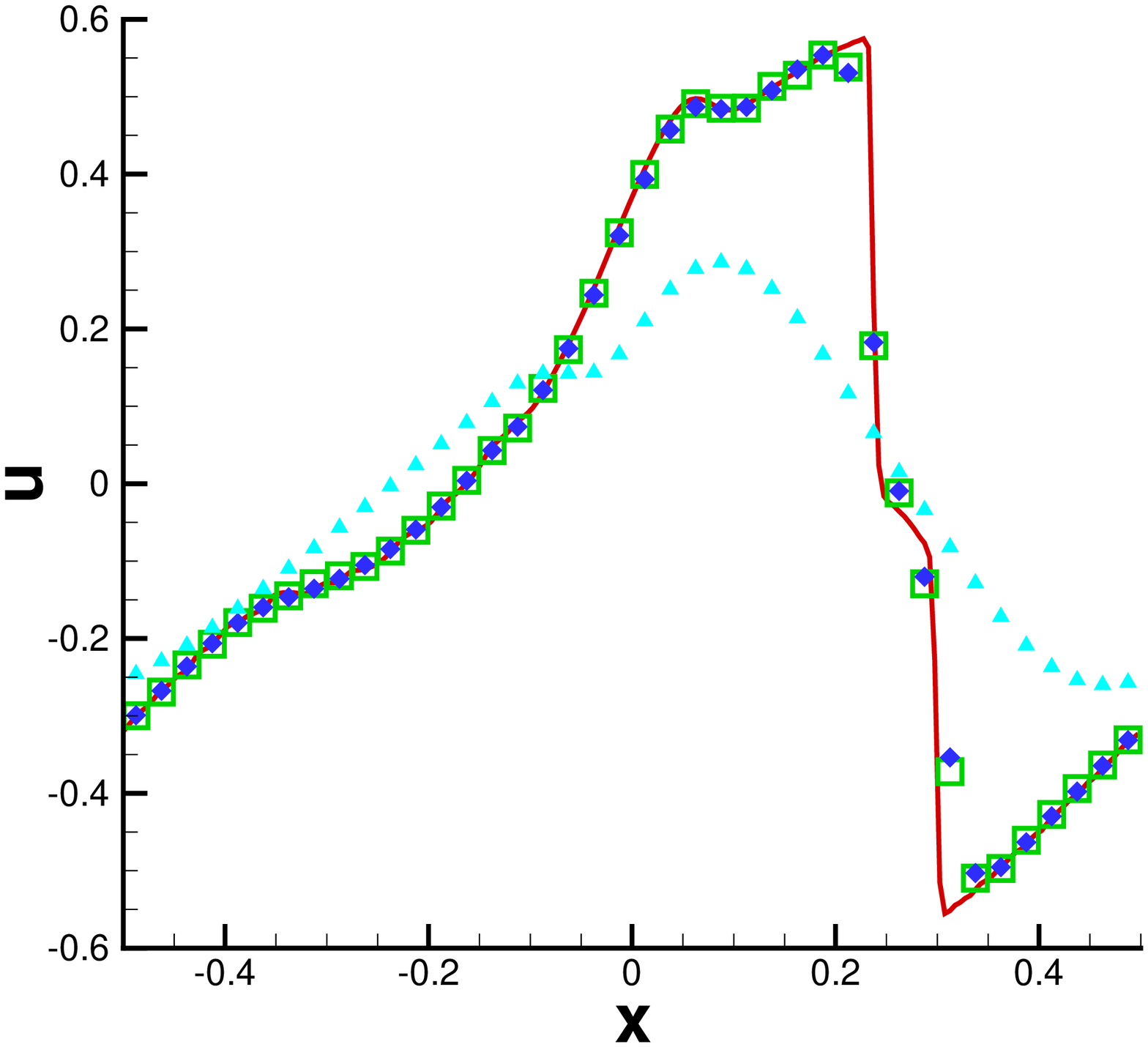} \\
\includegraphics[totalheight=1.8in]{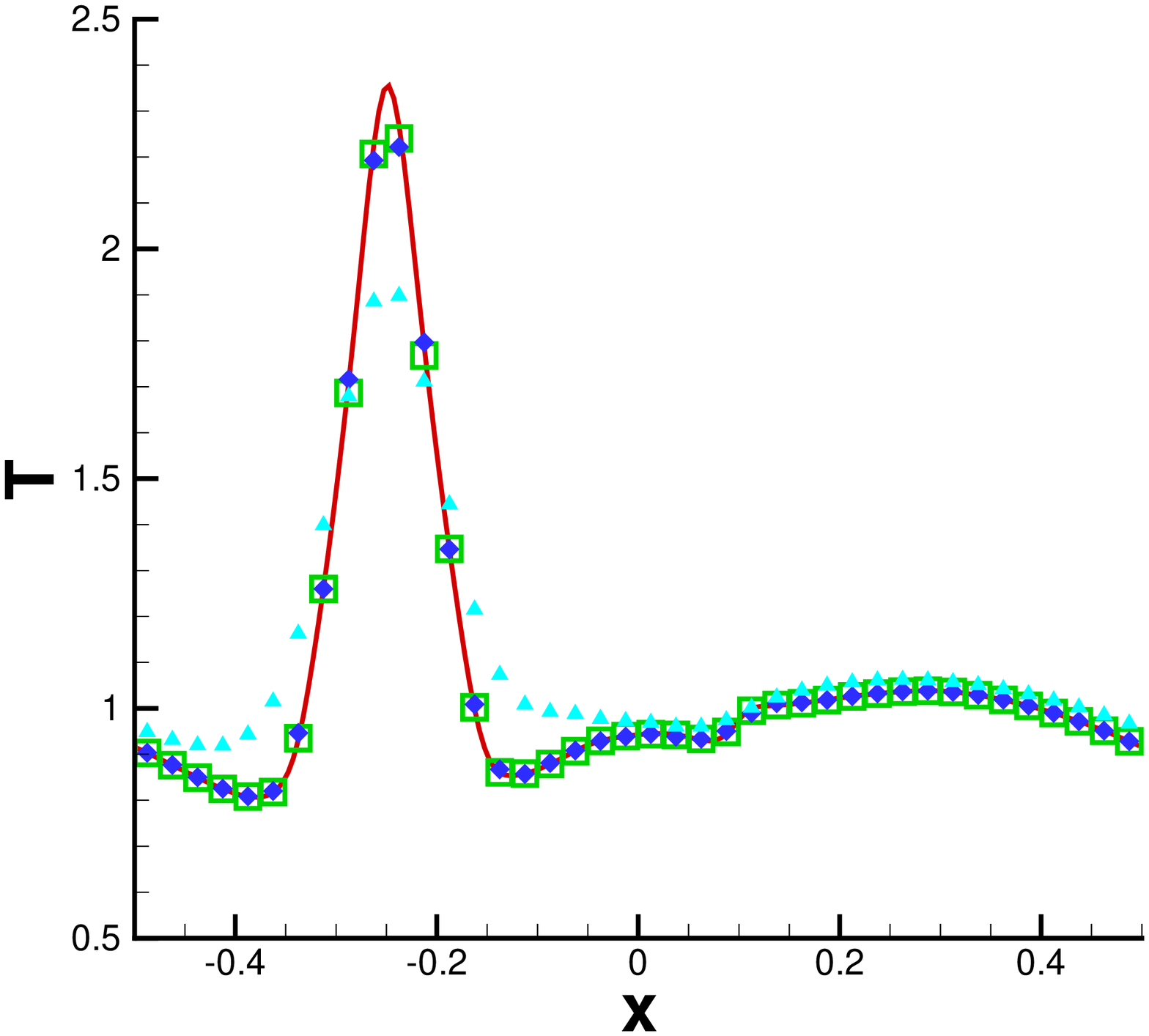},
\includegraphics[totalheight=1.8in]{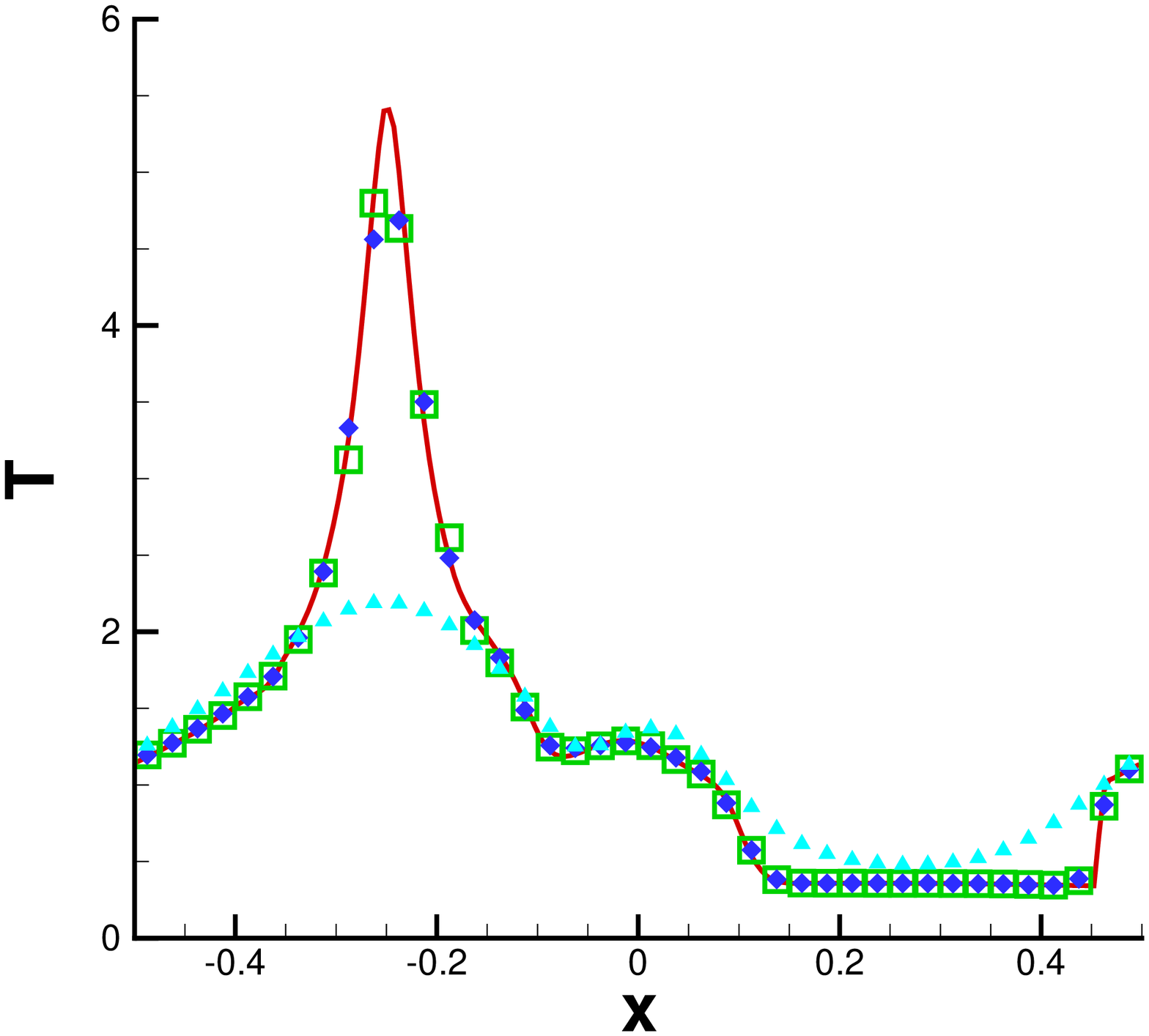},
\includegraphics[totalheight=1.8in]{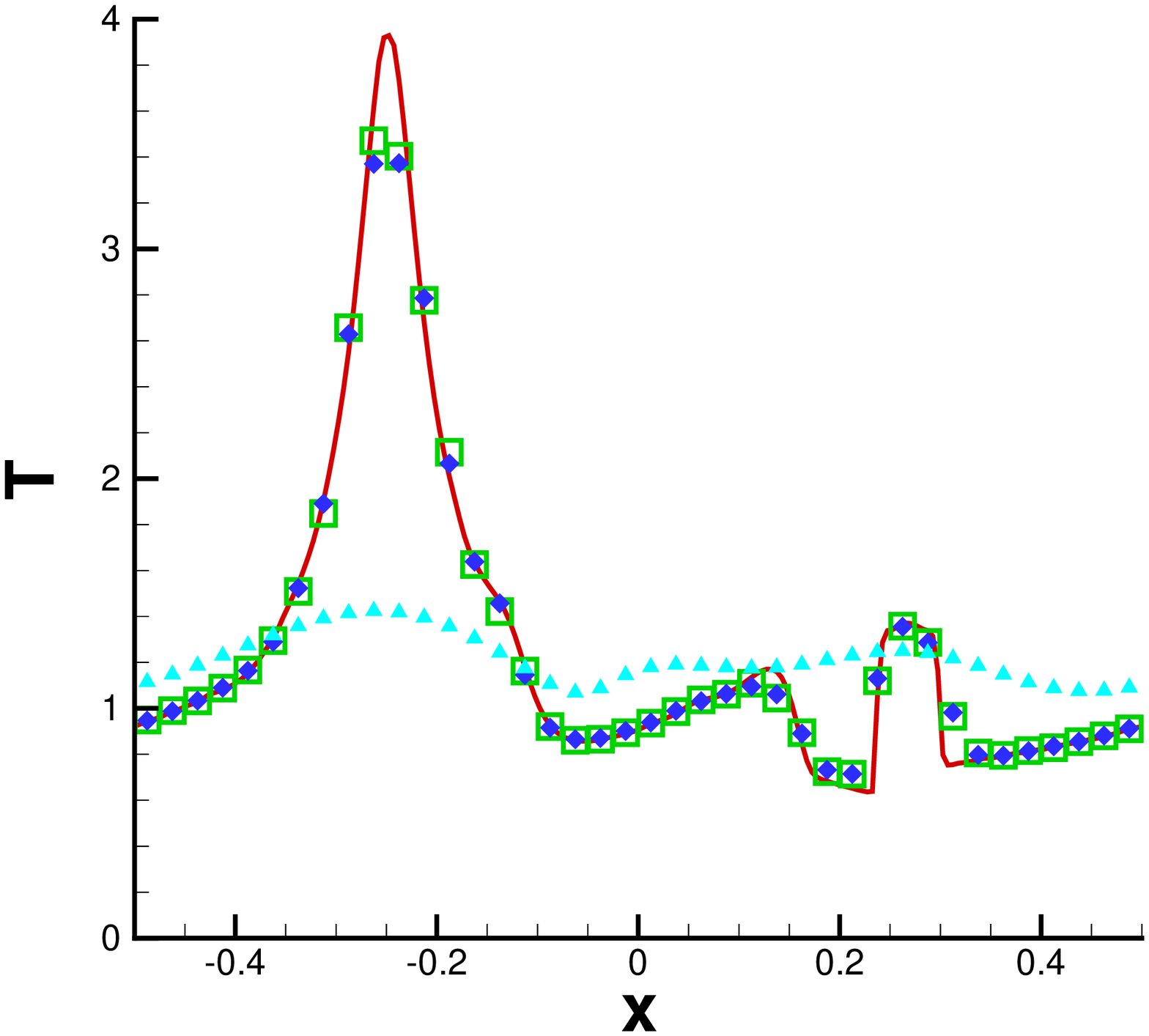}
\caption{Mixed regime problem with $\eps(x)$ in \eqref{eps} with $a_0=40$ and $\eps_0=10^{-6}$ on the domain $[-0.5,0.5]\times[-10,10]$.
$N_x=40$ and $N_v=100$. Symbols: square NDG3, diamond NDG2, delta NDG1. Solid line: reference solution of NDG3 with $N_x=200$ and $N_v=200$. From left to right: time
$t=0.1, 0.3, 0.45$. From top to bottom, the distribution function $f$ at $x=0$ along $v$ direction,
the density $\rho$, the mean velocity $u$ and the temperature $T$. With TVB limiter and $M_{tvb}=20$.}
\label{fig11}
\end{figure}

\begin{figure}[ht]
\centering
\includegraphics[totalheight=1.8in]{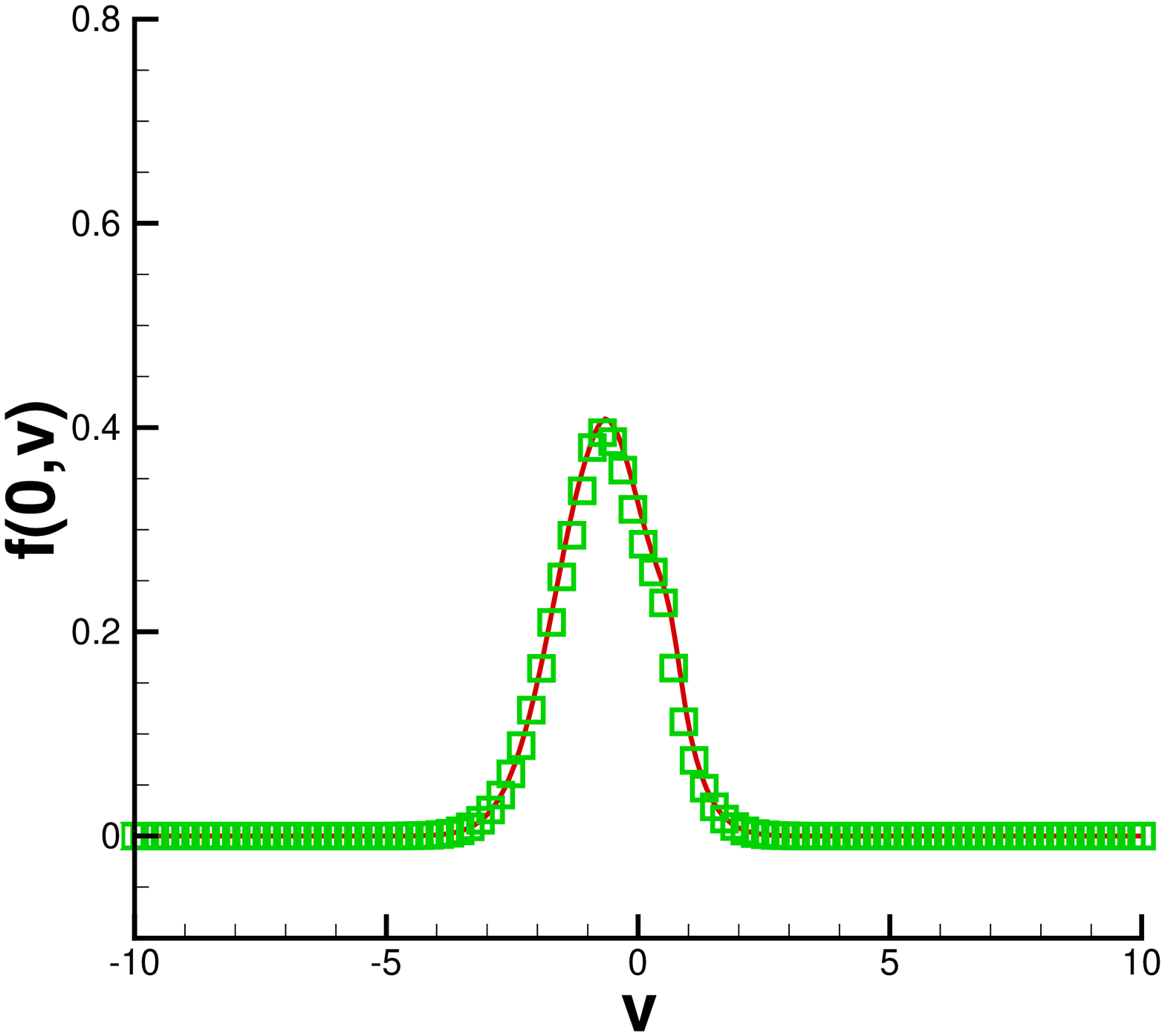},
\includegraphics[totalheight=1.8in]{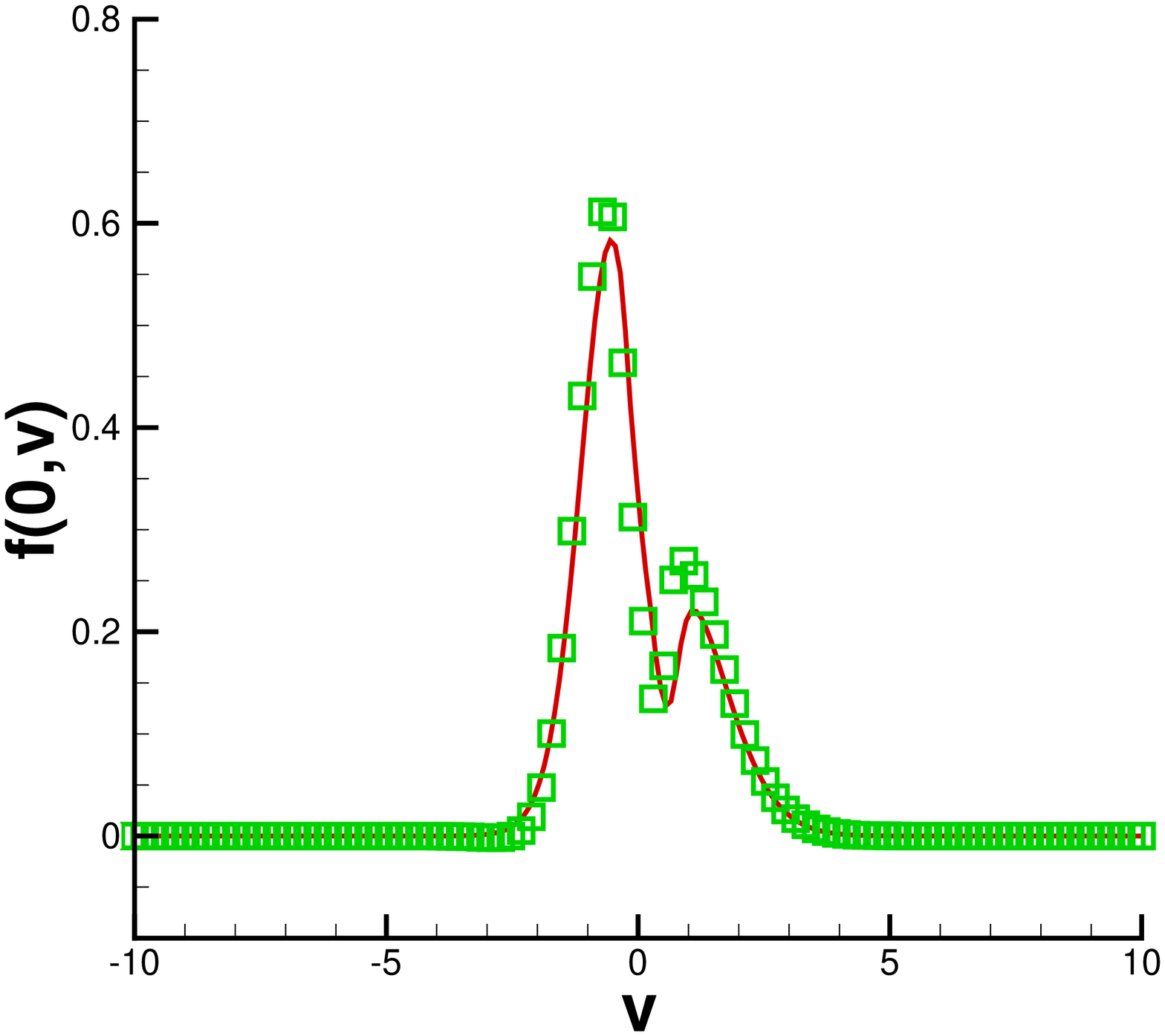},
\includegraphics[totalheight=1.8in]{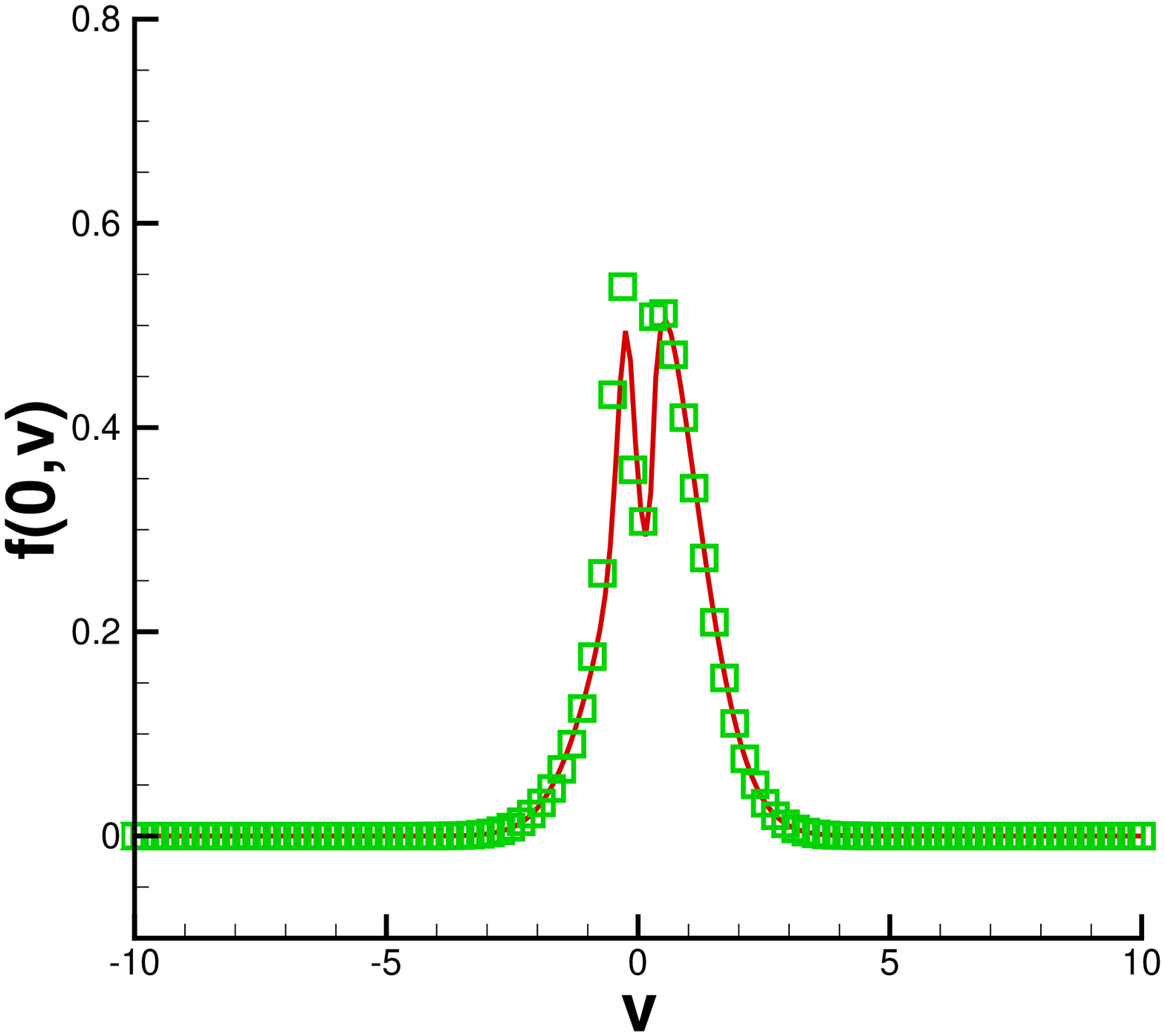} \\
\includegraphics[totalheight=1.8in]{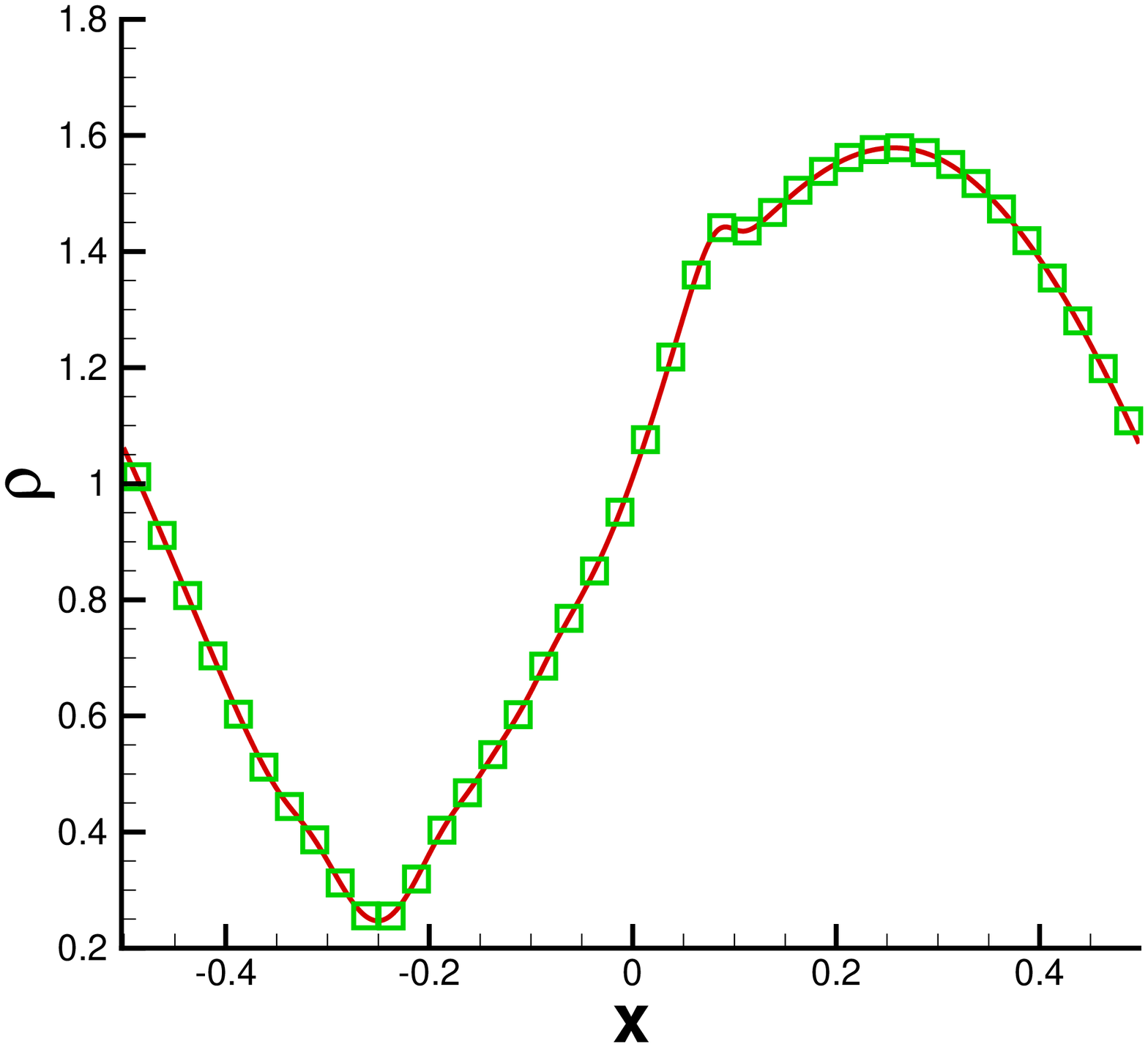},
\includegraphics[totalheight=1.8in]{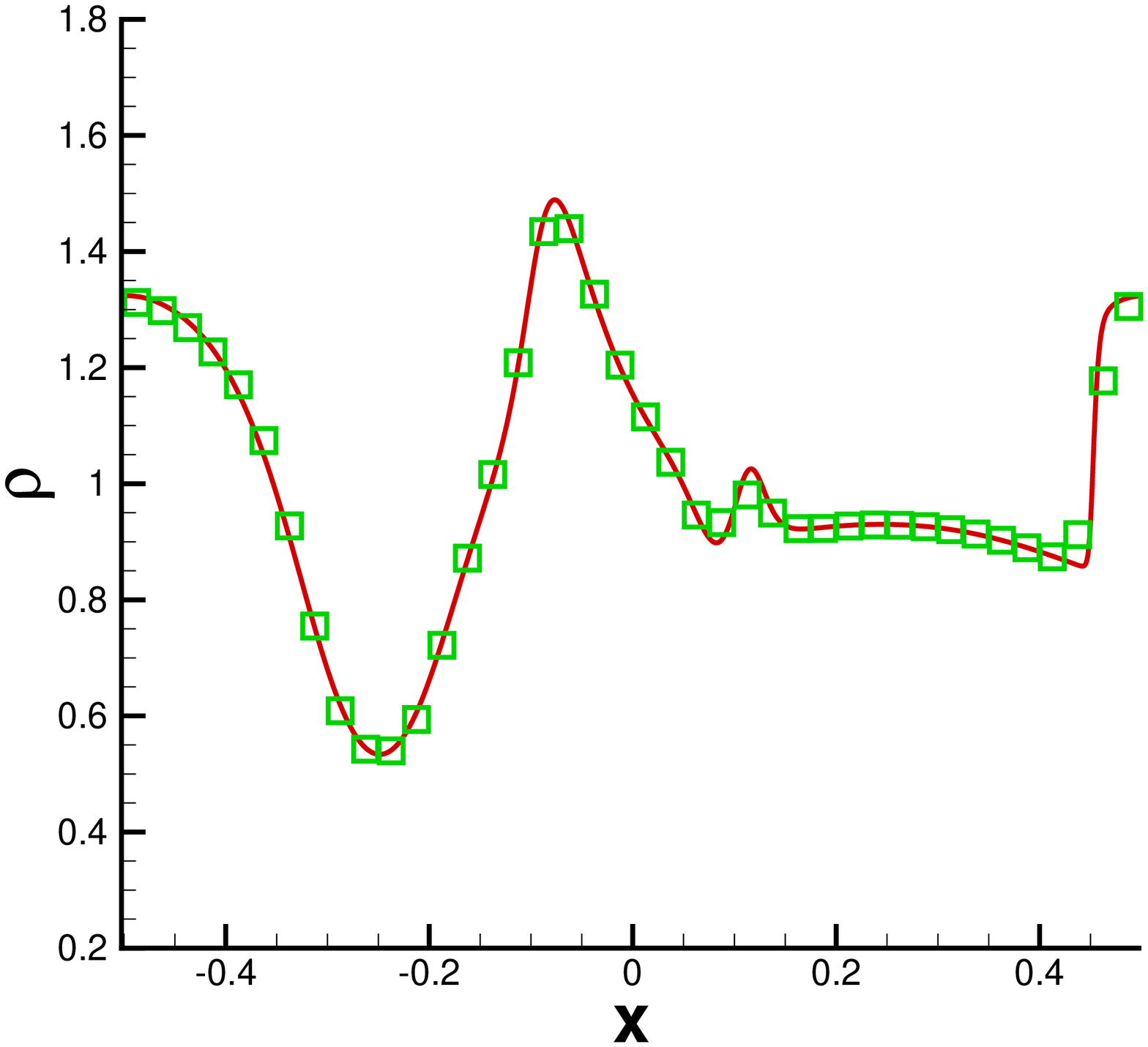},
\includegraphics[totalheight=1.8in]{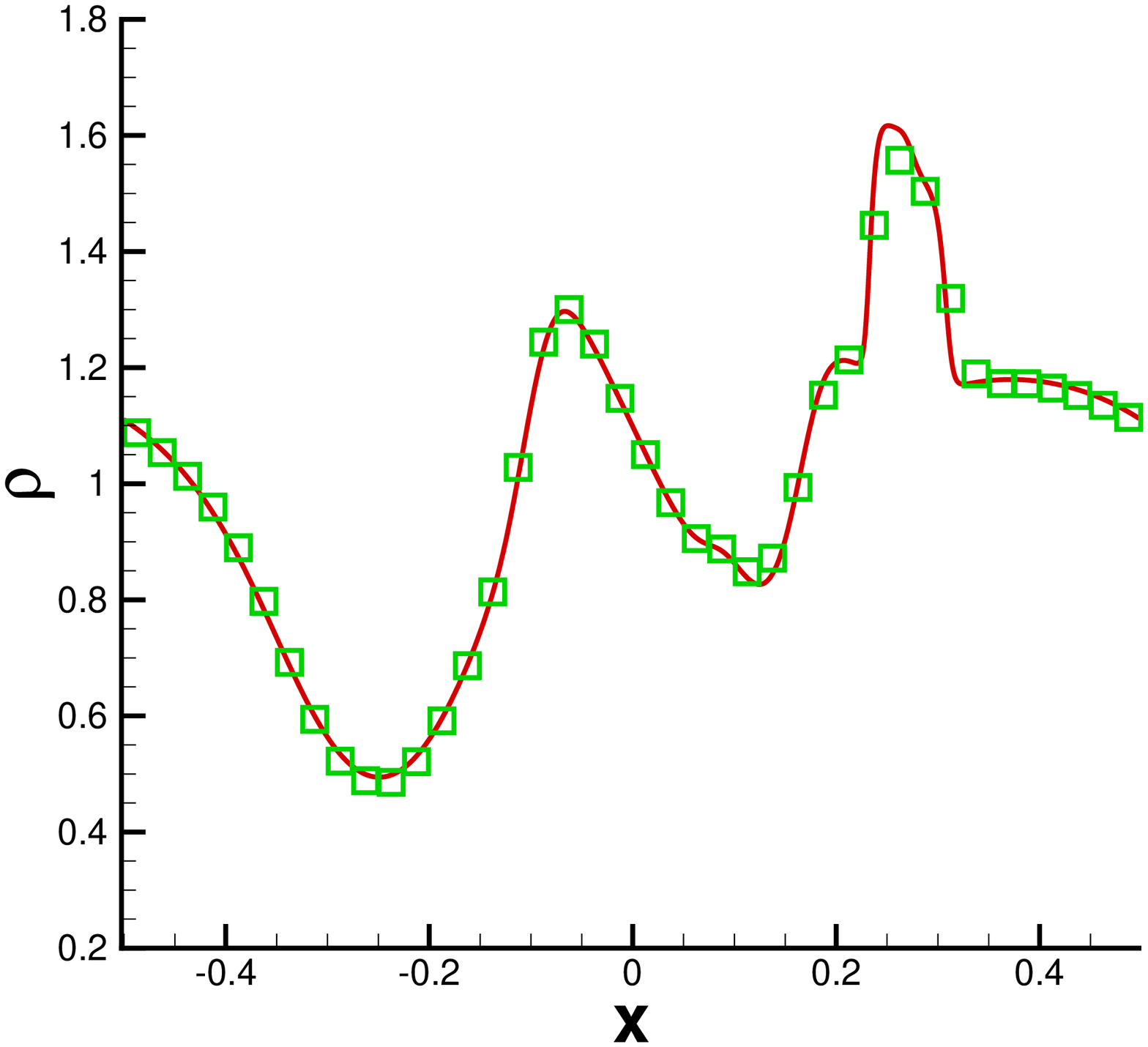} \\
\includegraphics[totalheight=1.8in]{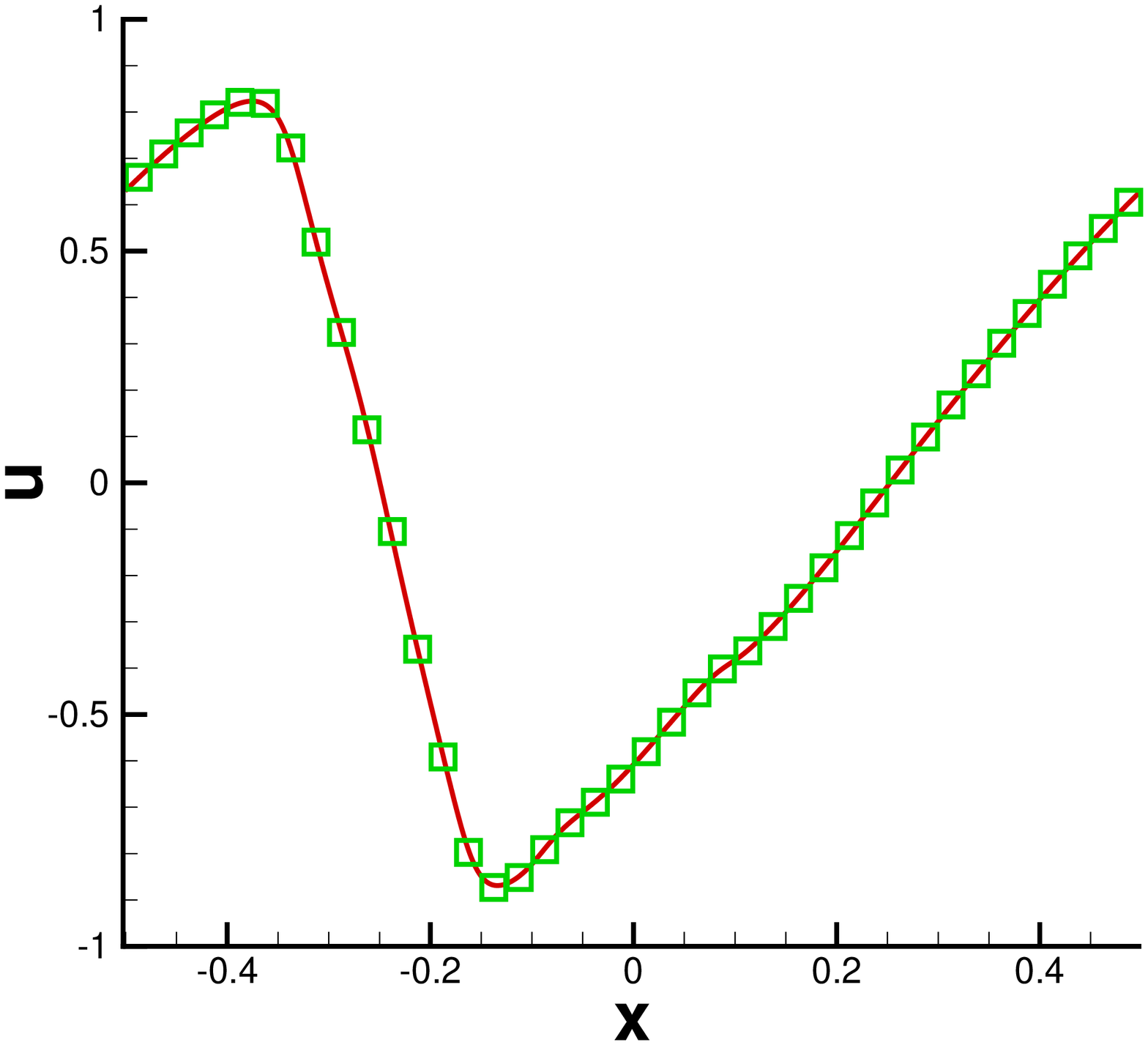},
\includegraphics[totalheight=1.8in]{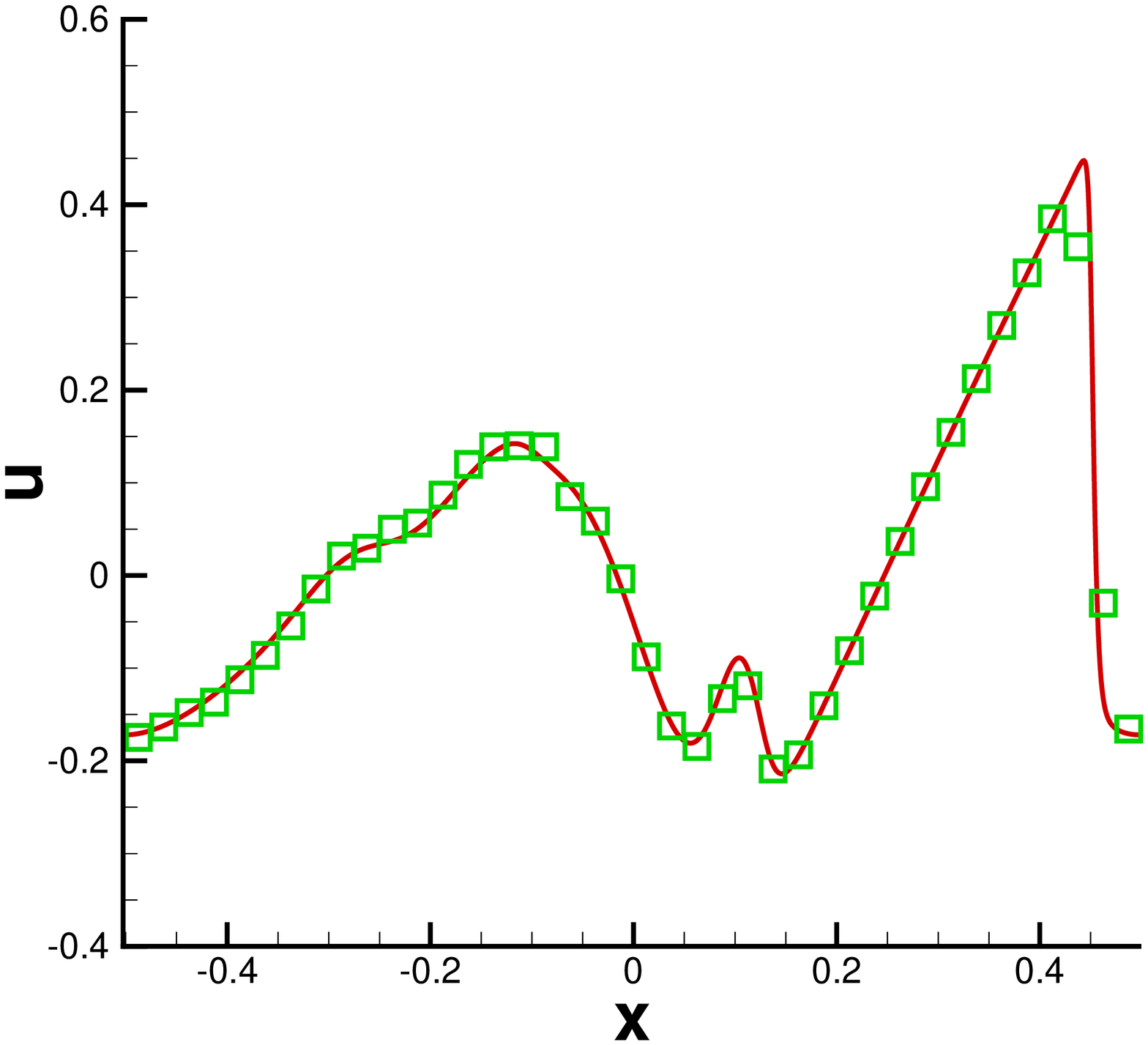},
\includegraphics[totalheight=1.8in]{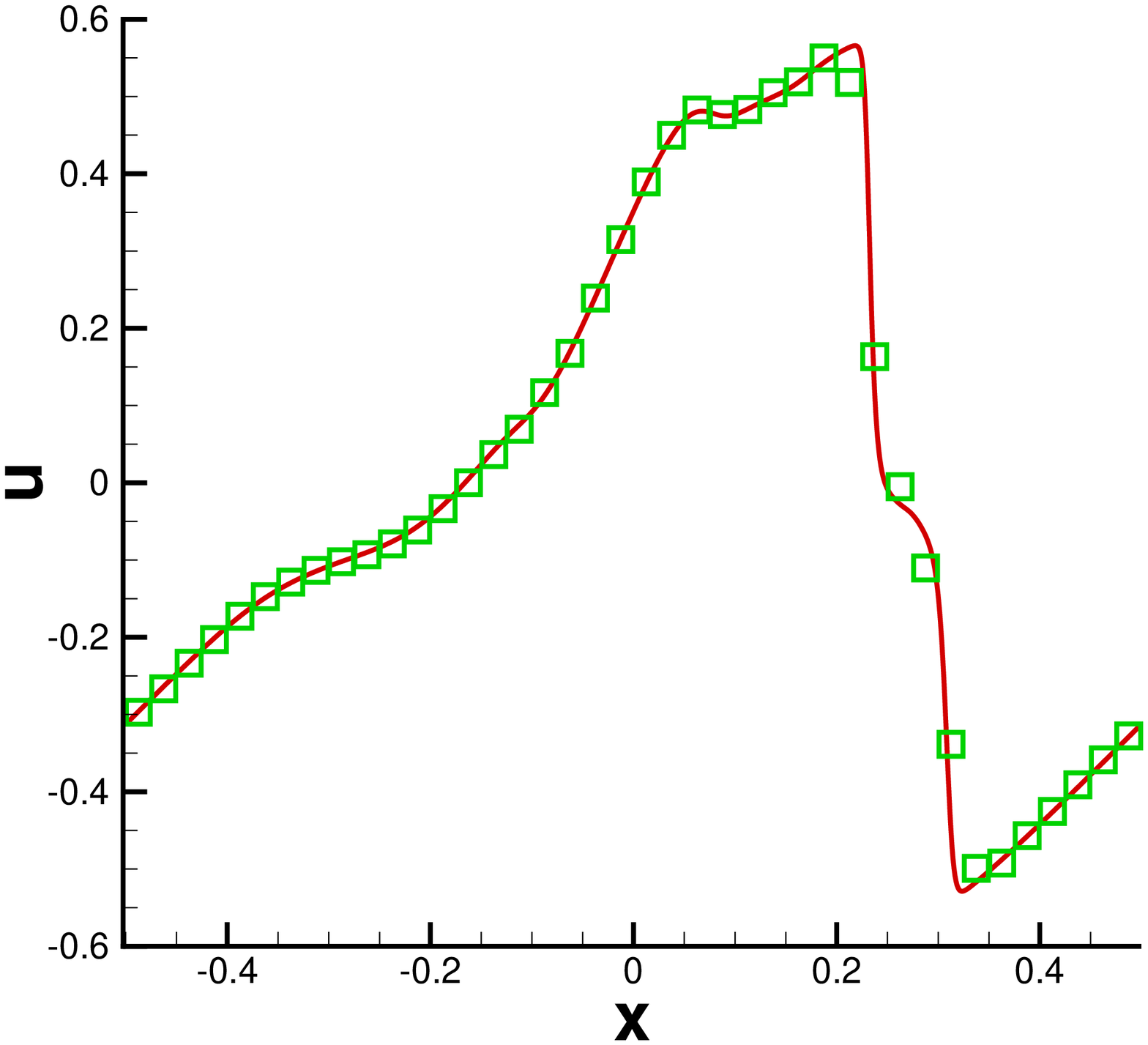} \\
\includegraphics[totalheight=1.8in]{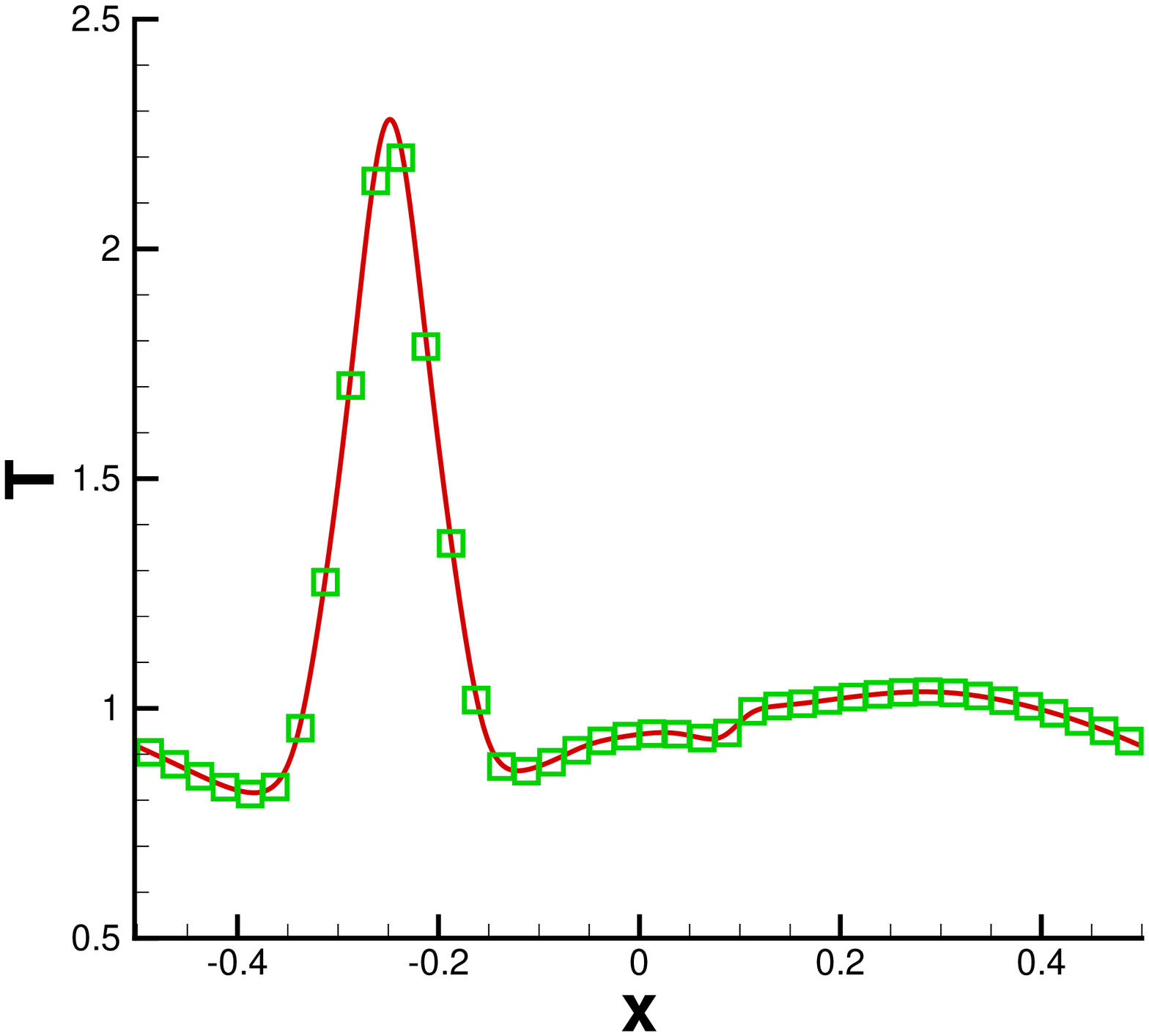},
\includegraphics[totalheight=1.8in]{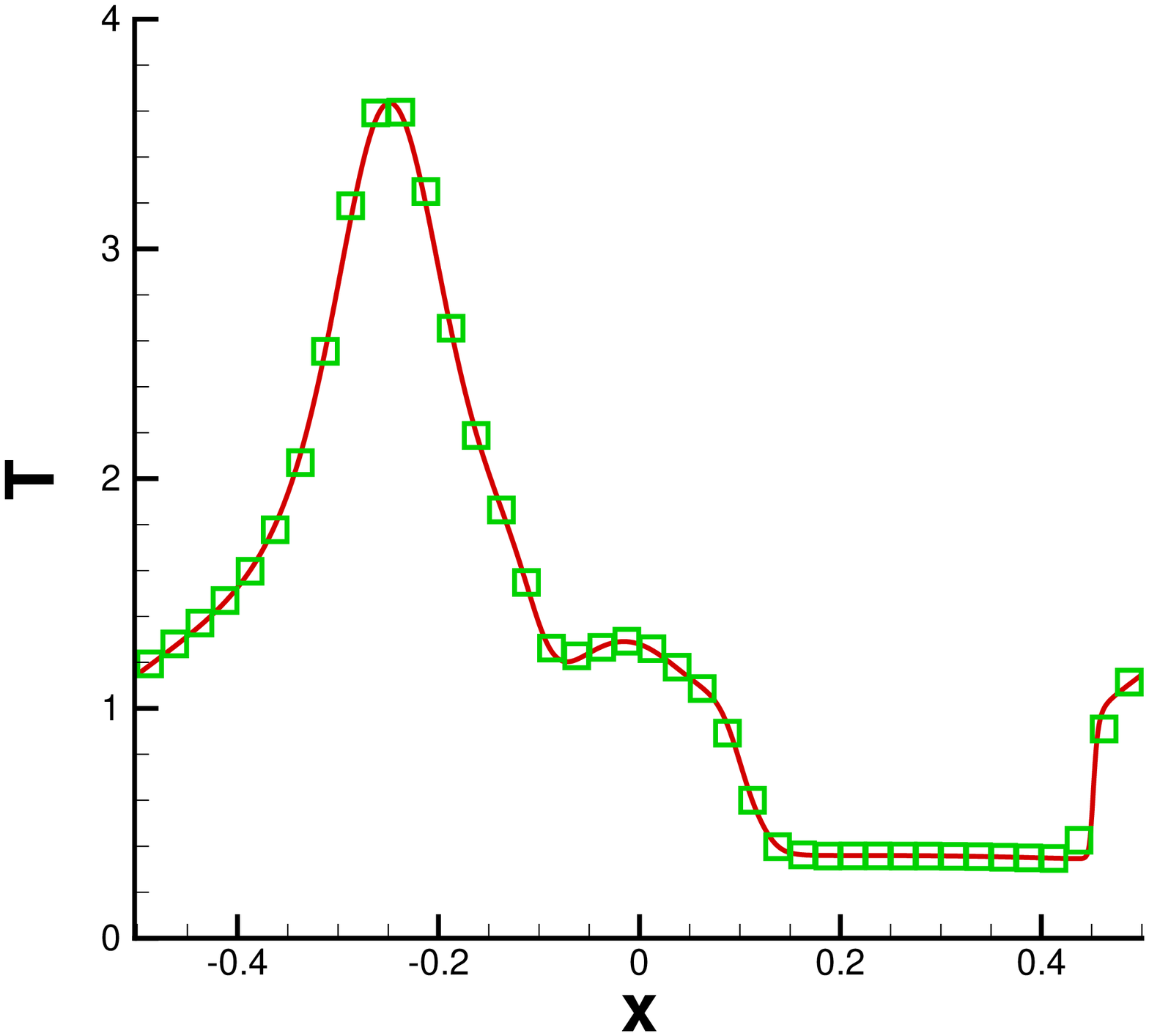},
\includegraphics[totalheight=1.8in]{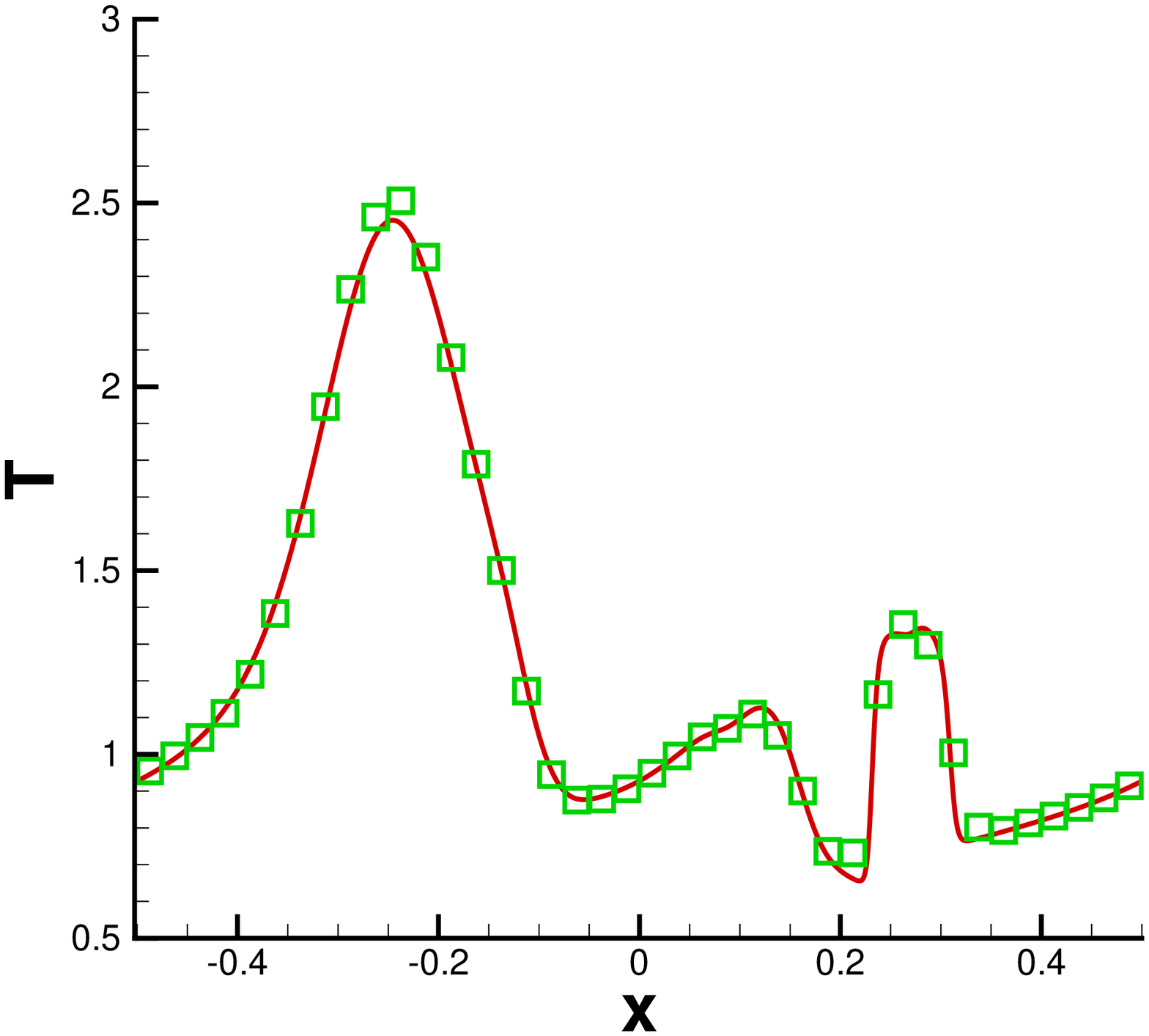}
\caption{Mixed regime problem with $\eps(x)$ in \eqref{eps} with $a_0=40$ and $\eps_0=10^{-3}$ on the domain $[-0.5,0.5]\times[-10,10]$.
Symbols: $N_x=40$ and $N_v=100$ with NDG3. Solid line: $N_x=1000$ and $N_v=100$ with NDG1 in space and the Euler forward in time to explicitly solve
the BGK model \eqref{bgk}. From left to right: time
$t=0.1, 0.3, 0.45$. From top to bottom, the distribution function $f$ at $x=0$ along $v$ direction,
the density $\rho$, the mean velocity $u$ and the temperature $T$. With TVB limiter and $M_{tvb}=20$.}
\label{fig13}
\end{figure}

\end{exa}

\section{Conclusion}
\label{sec6}
\setcounter{equation}{0}
\setcounter{figure}{0}
\setcounter{table}{0}

The work in this paper is a continuation of our research effort in \cite{JLQX_diffusive, JLQX_analysis} to develop and analyze high order asymptotic preserving schemes for kinetic equations in different scalings. We here propose high order DG-IMEX schemes for the BGK equation in a  hyperbolic scaling. Specifically, we employ high order nodal DG spatial discretizations coupled with a high order globally stiffly accurate IMEX scheme in time for the equivalent micro-macro decomposition of the BGK equation. Two versions of the schemes are proposed. While the first one is more straightforward based on the micro-macro decomposition and can be formally shown to be asymptotically equivalent to a widely used RK DG schemes for the Euler equations in the limit of $\eps \rightarrow 0$; the second version is computationally more efficient, and more importantly it allows a formal asymptotic analysis that shows the equivalence, up to $\mathcal{O}(\eps^2)$, with a local DG scheme to the compressible Navier-Stokes equations when $0<\eps\ll 1$. Extensive numerical examples are presented to demonstrate the effectiveness of the proposed methods. Extension to kinetic equations with more general collisional operators will be explored in the future. 

\bibliographystyle{siam}
\bibliography{refer}

\end{document}